\documentclass[10pt]{amsart}

\usepackage{amsmath}
\usepackage{amsthm}
\usepackage[utf8]{inputenc}
\usepackage{amsfonts}
\usepackage{hyperref}
\usepackage{amssymb}
\usepackage{stackrel}
\usepackage[english]{babel}
\usepackage{url}
\usepackage{array}
\usepackage{xcolor}
\usepackage{tikz-cd}
\usepackage[cmtip, all]{xy}
\usepackage{comment}
\usepackage{csquotes}
\usepackage[backend=bibtex, 
sorting=nyt, maxcitenames=7,
mincitenames=1, maxnames=7, abbreviate=true, block=space]{biblatex}
\newtheorem{theorem}{Theorem}[section]
\newtheorem{introtheorem}{Theorem}
\newtheorem{introconj}[introtheorem]{Conjecture}

\newtheorem{prop}[theorem]{Proposition}
\newtheorem{lem}[theorem]{Lemma}
\newtheorem{cor}[theorem]{Corollary}
\theoremstyle{definition}
\newtheorem{defn}[theorem]{Definition}
\newtheorem{nota}[theorem]{Notation}
\newtheorem{constr}[theorem]{Construction}
\newtheorem{exa}[theorem]{Example}
\newtheorem{rem}[theorem]{Remark}
\newtheorem{introrem}[introtheorem]{Remark}

\numberwithin{equation}{section}
\DeclareMathOperator{\Spec}{Spec}

\newcommand{\Pic}{{\rm Pic}}

\newcommand{\Hom}{{\rm Hom}}
\newcommand{\im}{{\operatorname{im}}}

\newcommand{\supp}{{\rm supp}\,}
\newcommand{\0}{\emptyset}

\newcommand{\sE}{{\mathcal E}}

\newcommand{\sI}{{\mathcal I}}
\newcommand{\sJ}{{\mathcal J}}

\newcommand{\sL}{{\mathcal L}}
\newcommand{\sM}{{\mathcal M}}
\newcommand{\sN}{{\mathcal N}}
\newcommand{\sO}{{\mathcal O}}

\newcommand{\sV}{{\mathcal V}}
\newcommand{\sW}{{\mathcal W}}

\newcommand{\A}{{\mathbb A}}

\newcommand{\C}{{\mathbb C}}

\newcommand{\G}{{\mathbb G}}

\renewcommand{\L}{{\mathbb L}}

\renewcommand{\P}{{\mathbb P}}
\newcommand{\Q}{{\mathbb Q}}
\newcommand{\R}{{\mathbb R}}

\newcommand{\V}{{\mathbb V}}

\newcommand{\Z}{{\mathbb Z}}

\newcommand{\sgn}{\operatorname{sgn}}

\newcommand{\BM}{{\operatorname{B.M.}}}

\renewcommand{\det}{\operatorname{det}}

\newcommand{\id}{{\operatorname{\rm Id}}}

\newcommand{\Zar}{{\text{\rm Zar}}} 
\newcommand{\Nis}{{\text{\rm Nis}}}

\newcommand{\Sch}{{\operatorname{\mathbf{Sch}}}} 
\newcommand{\colim}{{\mathop{\rm colim}}}

\newcommand{\<}{\langle}
\renewcommand{\>}{\rangle}

\renewcommand{\dim}{{\operatorname{\rm dim}}} 
\newcommand{\Div}{{\operatorname{div}}}

\newcommand{\Hilb}{\operatorname{Hilb}}

\newcommand{\Mon}{{\operatorname{Mon}}}
\newcommand{\sig}{{\operatorname{sig}}}

\newcommand{\Sm}{{\mathbf{Sm}}}
\newcommand{\GpS}{{\mathbf{GpSch}}}

\newcommand{\Sym}{{\operatorname{Sym}}}

\newcommand{\Ext}{{\operatorname{Ext}}}

\newcommand{\rnk}{{\operatorname{\text{rnk}}}} 
\newcommand{\Bl}{\text{Bl}}

\newcommand{\GW}{{\operatorname{GW}}} 
 
\newcommand{\SH}{{\operatorname{SH}}}

\newcommand{\sHom}{\mathcal{H}om}
 
\newcommand{\Aut}{{\operatorname{Aut}}}

\newcommand{\ret}{{\operatorname{\text{r\'et}}}}

\newcommand{\GL}{\operatorname{GL}}
\newcommand{\SL}{{\operatorname{SL}}}

\newcommand{\vir}{\text{\it vir}}

\newcommand{\ind}[1]{}

\newcommand{\inp}[1]{}

\newcommand{\EM}{{\operatorname{EM}}}

\newcommand{\Var}{{\operatorname{Var}}}

\newcommand{\sEnd}{\sEnd}

\begin{document}
	\title[Quadratic Donaldson-Thomas invariants]{Quadratic Donaldson-Thomas invariants for $(\P^1)^3$ and some other smooth proper toric threefolds}

\author{Marc~Levine}
\address{Marc Levine\\
Universit\"at Duisburg-Essen,
Fakult\"at Mathematik, Campus Essen, 45117 Essen, Germany}
\email{marc.levine@uni-due.de}

\author{Anna M. Viergever}
\address{Anna M. Viergever\\
Leibniz Universit\"at Hannover\\
Fakult\"at f\"ur Mathematik, Welfengarten 1, 30167 Hannover, Germany}
\email{viergever@math.uni-hannover.de} 


	\date{\today}
\thanks{The author M.L. was supported by the ERC through the project QUADAG.  This paper is part of a project that has received funding from the European Research Council (ERC) under the European Union's Horizon 2020 research and innovation programme (grant agreement No. 832833).  \\
\includegraphics[scale=0.08]{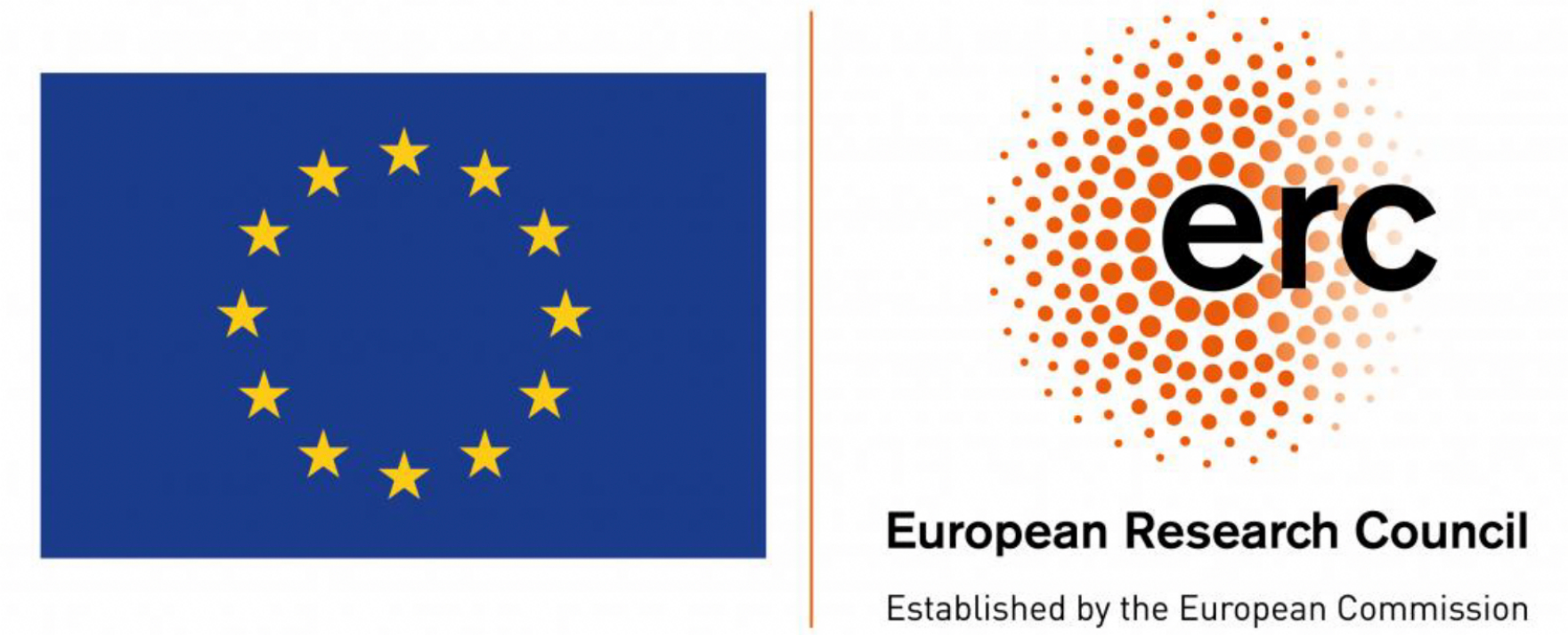}}

	\begin{abstract}
	Using virtual localization in Witt sheaf cohomology, we show that the generating series of quadratic Donaldson-Thomas invariants of $(\P^1)^3$, valued in the Witt ring of $\R$, $W(\R)\cong \Z$, is equal to $M(q^2)^{-8}$, where $M(q)$ is the MacMahon function.  This confirms a modified version of the conjecture of \cite{Viergever}. We also show that a localized version of this conjecture holds for certain iterated blow-ups of $(\P^1)^3$ and other related smooth proper toric varieties.
	\end{abstract} 
	
	\maketitle

	\tableofcontents
	
\begingroup	
	\section*{Introduction}
	\renewcommand{\thesection}{I}

	Let $X$ be a smooth projective threefold over the complex numbers. For $n\geq 1$, the Donaldson-Thomas invariant $I_n\in\Z$ of $X$ is defined to be the degree of the virtual fundamental class of the Hilbert scheme $\Hilb^n(X)$  of degree $n$, dimension zero  closed subschemes of $X$. These invariants were first introduced by Donaldson and Thomas in the papers \cite{ThomasHCI} and \cite{DonaldsonThomas}. 
	
	In their papers \cite{MaulikI, MaulikII}, Maulik, Nekrasov, Okounkov and Pandharipande consider the case of a smooth projective toric threefold $X$ over $\C$, and prove that 
\begin{equation}\label{equation: MNOP formula}
\sum_{n\geq 0}I_nq^n = M(-q)^{\deg(c_3(T_X\otimes K_X))},
\end{equation}
where $T_X$ is the tangent bundle on $X$ and $K_X$ is its canonical sheaf, and  $M(q) = \prod_{n=1}^\infty (1-q^n)^{-n}$ is the MacMahon function, which was first considered by MacMahon \cite{MacMahon} and proven by Stanley \cite{Stanley} to be the generating series of $3$-dimensional partitions. 
	
	The proof uses the virtual localization theorem \cite{GraberLVC} of Graber and Pandharipande, together with an explicit trace computation; this calculation is highly non-trivial, and is one of the main achievements of these papers. The formula \eqref{equation: MNOP formula} was conjectured by Maulik, Nekrasov, Okounkov and Pandharipande to hold for all smooth projective threefolds, which was proven in this generality by Levine-Pandharipande \cite{LevinePandharipande} and by Li \cite{Li}, relying on the toric case, but using entirely different methods from those employed in \cite{MaulikI, MaulikII}. 
	
	Levine has constructed a motivic analogue of virtual fundamental classes in \cite{LevineISNC}. This was extended to the equivariant setting in \cite{LevineVLEWC} and a virtual localization formula for this setting is proven in the same paper. Using this, one can define a motivic analogue of DT invariants for a smooth projective threefold $X$ over a field $k$, taking values in the Witt ring $W(k)$ of $k$. For these invariants to be  defined one needs to endow the Hilbert schemes $\Hilb^n(X)$ of $X$ with  {\em   orientations}, and in \cite{LevineOST},  Levine showed how to construct an orientation on each $\Hilb^n(X)$  from a spin structure on $X$, that is, an isomorphism of the canonical sheaf $K_X$ with the square of a line bundle.  
		
	We recall that $W(\C)\cong \Z/2\Z$, and when working over $\C$, one merely recovers the  invariants $I_n$ modulo 2.    However, $W(\R)\cong \Z$ via the signature map, so one can hope to find new information when working over $\R$. Another difference is that when applying virtual localization to obtain the $W(k)$-valued invariants, one cannot use a group action by a torus $\G_m^n$ as in the classical case, instead, one needs to use an action by a non-commutative group scheme admitting irreducible representations of even rank, such as $\SL_2$ or the normalizer $N$ of the diagonal torus $T_1\cong \G_m$ in  $\SL_2$, generated by $T_1$ and
\[
\sigma:=\begin{pmatrix}0&1\\-1&0\end{pmatrix}. 
\]
See the discussion in \cite[Introduction]{LevineABLEWC} for details. 
	
We will prove the following theorem. 
\begin{introtheorem}\label{theorem: main theorem} Let $k$ be a perfect field of characteristic $\neq2$. For $n\geq 0$, let $\tilde{I}_n$ be the quadratic Donaldson-Thomas  invariant of $(\P^1)^3$ over $k$.  Then in $W(k)\otimes\Q$ we have
\[
\sum_{n\geq 0}\tilde{I}_nq^n = M(q^2)^{-8}.
\]
\end{introtheorem}
Of course, if $k$ has positive characteristic, $W(k)\otimes\Q=0$, so our result says nothing. In characteristic zero, $W(k)\otimes\Q$ detects exactly the signature for each embedding of $k$ into a real closed field. Since $(\P^1)^3$ is defined over $\Q$, $W(\Q)\otimes\Q=W(\R)\otimes \Q=\Q$, and the invariants $\tilde{I}_n$ are compatible with change of the base field, one does not really achieve any greater generality by taking a general ground field $k$ rather than just using $\R$. However, the arguments involved in the computation are valid for arbitrary perfect $k$ of characteristic $\neq2$, so we prefer to frame the statement of Theorem~\ref{theorem: main theorem} in this generality.

	The proof uses Levine's virtual localization theorem together with an explicit trace computation following \cite{MaulikI, MaulikII}. 
	
	In the paper \cite{Viergever}, the first three nonzero quadratic DT invariants of $\P^3$ are computed to be $\tilde{I}_2(\P^3) = 10, \tilde{I}_4(\P^3)= 25$ and $\tilde{I}_6(\P^3)=-50$, and a conjecture is made about the general case.\footnote{What are called quadratic Donaldson-Thomas invariants in this paper are called motivic Donaldson-Thomas invariants in \cite{Viergever}.} These results and ours above agree with the following slightly modified version of this conjecture.  
	
	\begin{introconj}\label{conj:Main} Let $X$ be a smooth projective threefold over $\R$ such that  $K_X\cong L^{\otimes 2}$ for some line bundle $L$ on $X$. For $n\geq 0$, let $\tilde{I}_n(X)\in W(\R)$ be the quadratic DT invariant of $X$ and let $\text{ob}_2(X)$ be the obstruction bundle $h^1(E^{DT(2)\vee}_\bullet)$ for the Donaldson-Thomas perfect obstruction theory $E^{DT(2)}_\bullet$ on $\Hilb^2(X)$. Then in $W(\R)[[q]]=\Z[[q]]$, we have 
\[
\sum_{n\geq 0}\tilde{I}_n(X)q^n = M(\epsilon\cdot q^2)^{\epsilon\cdot\deg_\R(e(\text{ob}_2(X)))},
\]
for some $\epsilon\in\{\pm1\}$.
\end{introconj} 
The differences  between this conjecture and the one presented in (the published version of) \cite{Viergever} are the correction of a sign error in the exponent, and that the conjecture stated in \cite{Viergever} requires the scheme to have a suitable $N$-action.

Here are some comments on Conjecture~\ref{conj:Main}. We recall that $\Hilb^n(X)$ is smooth for any smooth projective $X$ and $n\le 3$ (see, e.g., \cite[Remark 7.2.5]{FG}), and hence for $n\le 3$, the obstruction bundle is a vector bundle, with $\deg_\R(e(\text{ob}_n(X))=\tilde{I}_n(X)$ (cf. \cite[Proposition 5.6]{BehrendFantechiISNC}). For $n=2$, we see that $\tilde{I}_2(X)$ agrees with the coefficient of $q^2$ in $M(\epsilon\cdot q^2)^{\epsilon \cdot\deg_\R(e(\text{ob}_2(X)))}$ for any choice of $\epsilon\in \{\pm1\}$. Working in reverse, we see that Theorem~\ref{theorem: main theorem} implies that
\[
\deg_\R(e(\text{ob}_2((\P^1_\R)^3)))=\tilde{I}_2((\P^1)^3)=-8.
\]
so Theorem~\ref{theorem: main theorem} is compatible with the Conjecture. Moreover, for $n=1,3$, the obstruction bundle has odd rank $3n=\dim_k\Hilb^n(X)$, since the Donaldson-Thomas obstruction theory for $\Hilb^n(X)$ has virtual rank zero. By \cite[Lemma 4.3]{LevineMEC}, it follows that $e(\text{ob}_n(X))=0$ in $H^{3n}(\Hilb^n(X), \sW(\det^{-1}\text{ob}_n(X))$, so $\tilde{I}_1(X)=\tilde{I}_3(X)=0$, in accordance with 
Conjecture~\ref{conj:Main}, which if true would imply that $\tilde{I}_n(X)=0$ for all odd $n$. At present, the only cases for which we know that $\tilde{I}_n(X)=0$ for all odd $n$ is for $X=(\P^1)^3$ (an immediate consequence of Theorem~\ref{theorem: main theorem}) and $X=\P^3$ (see \cite[Introduction]{Viergever}). 

We note that Theorem~\ref{theorem: main theorem}  does not seem to provide new insights into the case of $\P^3$. The main reason that $\tilde{I}_n(\P^3)$ for $n\geq 8$ could not be computed in \cite{Viergever} was that the method used  relies on the $N$-action on $\P^3$ only having isolated fixed points. It turns out, however, that the fixed points inside the Hilbert scheme in degrees $n\geq 8$ have positive dimension. On $(\P^1)^3$, the $N$-actions which we use have only isolated fixed points on $\Hilb^n$ for each $n$ up to a bound $M$, with $M$ tending to infinity as we vary our choice of $N$-action, and so the same method as used in \cite{Viergever} computes all quadratic DT invariants. Once we have this information, we use the method of the ``equivariant vertex measure'' employed in  \cite{MaulikI, MaulikII} to assemble our computations of the individual $\tilde{I}_n$ to give the formula for the generating function in Theorem~\ref{theorem: main theorem}.
	
Note also that in the conjecture we do not ask for a choice of isomorphism $K_X\cong L^{\otimes 2}$, only its existence. A chosen isomorphism $\rho:K_X\xrightarrow{\sim} L^{\otimes 2}$, i.e., a choice of a spin structure on $X$, is what is needed to construct an orientation on $\Hilb^n(X)$ for each $n$, but the choice of  $\rho$ will not affect the orientation on $\Hilb^n(X)$ for even $n$. Changing $L$  by a 2-torsion line bundle will similarly have no effect on the invariants $\tilde{I}_n(X)$ for even $n$.   As part of the conjecture is that $\tilde{I}_n(X)=0$ for odd $n$, these choices for the spin structure on $X$ should not play a role.

If  $X\to (\P^1)^3$ is a blow up at a set of $k$-points, stable under the $N$-action, the  $N$-action on $(\P^1)^3$ extends uniquely to an $N$-action on $X$. We can repeat this construction, forming an {\em $N$-equivariant iterated blow-up of $(\P^1)^3$}. This is a special case of what we call a {\em $N_3$-oriented smooth proper toric variety over $k$} $X_\Sigma$, associated to a complete regular simplicial fan $\Sigma\subset \Z^3$, which we study in  \S\ref{sec:BlowUp}. $N_3$ is the subgroup of $N^3$ generated by the torus $T_1^3$ and the element $\sigma_\delta=(\sigma, \sigma,\sigma)$, and plays a technical role in giving us the necessary flexibility to vary the $N$-action on  $X_\Sigma$ to aid in computing the $N$-equivariant DT invariants for $\Hilb^n(X_\Sigma)$ as  $n$ increases.  

 In this case, we are not able to identify the corresponding DT invariants in $W(\R)$ with the ones we get by applying the localization methods of \cite{LevineVLEWC}. However, we are still able to compute the localized $N$-equivariant DT invariants, which live in $W(\R)=\Z$, giving our two main results, Theorem~\ref{thm:MainBlowup} and Corollary~\ref{cor:MainBlowup}. Forming the corresponding generating series $\tilde{Z}^N(X,q)$,  we summarize 
 Theorem~\ref{thm:MainBlowup} and Corollary~\ref{cor:MainBlowup} as follows.
 \begin{introtheorem} Let $X=X_\Sigma$ be an $N_3$-oriented smooth proper toric variety over $k$, associated to a complete regular simplicial fan $\Sigma\subset \Z^3$. Then
\[
\tilde{Z}^N(X,q)=M(q^2)^{\widetilde{\deg}_\R(e^N(\text{ob}_2(X)))}.
\]
Moreover, 
\[
\widetilde{\deg}_\R(e^N(\text{ob}_2(X)))=-\frac{1}{2}\deg_\R(c_3(\text{ob}_1(X)))=\frac{1}{2}\deg_\R(c_3(T_X\otimes K_X)).
\]
\end{introtheorem}
This conforms with Conjecture~\ref{conj:Main}. Here $\text{ob}_1(X)$ and $\text{ob}_2(X)$ are the  obstruction bundles on $X=\Hilb^1(X)$ and $\Hilb^2(X)$, respectively, $e^N(\text{ob}_2(X))$ is an equivariant Euler class and $\widetilde{\deg}_\R(e^N(\text{ob}_2(X)))$  is its degree,  suitably localized to live in $W(\R)=\Z$. In fact, we need to introduce additional data to define the generating series and the degree of the equivariant Euler class, but one main point of Theorem~\ref{thm:MainBlowup} is that the resulting objects are in fact independent of these extra data. If one could show that $Z(X,q)=\tilde{Z}^N(X,q)$, that is, that the localized equivariant DT invariants compute the non-equivariant ones $\tilde{I}_n(X)$, this would give further evidence for Conjecture~\ref{conj:Main}, but this seems out of reach at present.

We conclude in \S\ref{sec:Construction} with giving a criterion for a complete regular simplicial fan $\Sigma\subset \Z^3$ to define a toric variety $X_\Sigma$ over $k$ that is $N_3$-oriented. We hope that this will give us additional examples besides the iterated blow-ups of $(\P^1)^3$ mentioned above, but at present, these are the only examples we have.

\begin{introrem} The main reason we have framed Conjecture~\ref{conj:Main} over the base-field $\R$ is to make sense of the expression $M(\epsilon\cdot q^2)^{\epsilon\cdot\deg_\R(e(\text{ob}_2(X)))}$ as an element of $W(\R)[[q]]$, using the signature to identify $W(\R)$ with $\Z$. Over a general field $k$, the formal sum $\sum_{n\geq 0}\tilde{I}_nq^n$ defines an element of $W(k)[[q]]$ and $\epsilon\cdot\deg_k(e(\text{ob}_2(X))$ is an element of $W(k)$. 

Recall  that a  {\em power structure} on a commutative ring $R$ is a map
\[
(1+tR[[t]])\times R\to 1+tR[[t]];\quad (f(t), r)\mapsto f(t)^r,
\]
satisfying the formal properties one would expect from an expression of the form $f(t)^r$ (see \cite[pg. 50, Definition]{GZLMH} for details). The usual power $f(t)^n$ defines the unique power structure on $\Z$. A power structure on $W(k)$ would give a meaning to the expression $M(\epsilon\cdot q^2)^{\epsilon\cdot\deg_k(e(\text{ob}_2(X)))}$, so one could generalize the above conjecture to an arbitrary base-field $k$, given a power structure on the Witt ring. There are several constructions of power structures on the Grothendieck-Witt ring (see \cite{McGarraghy}, \cite{PajPal}), but we are not aware of constructions of a power structure on the Witt ring. 

The methods developed in \cite{LevineISNC, LevineOST} do give a lifting of the Witt-valued Donaldson-Thomas invariants $\tilde{I}_n(X)$ discussed here to DT invariants $\tilde{I}^{\GW}_n(X)$ valued in the Grothendieck-Witt ring $\GW(k)$. Via the rank homomorphism $\rnk:\GW(k)\to \Z$, and the quotient map $\pi:\GW(k)\to W(k)$, one recovers the classical DT invariant $I_n(X)\in \Z$ and our Witt ring valued invariant $\tilde{I}_n(X)\in W(k)$ by
\[
I_n(X)=\rnk(\tilde{I}^{\GW}_n(X)),\ \tilde{I}_n(X)=\pi(\tilde{I}^{\GW}_n(X)).
\]
As $(\rnk,\pi):\GW(k)\to \Z\times W(k)$ is injective, one also recovers $\tilde{I}^{\GW}_n(X)$ from $(I_n(X), \tilde{I}_n(X))$, with the added information that these two invariants agree in $\Z/2$ after applying the quotient map $\Z\to \Z/2$ to $I_n(X)$ and the rank map $\overline{\rnk}:W(k)\to \Z/2$ to $\tilde{I}_n(X)$. \footnote{For $X=(\P^1)^3$, we have $\deg(c_3(T_X\otimes K_X))=-16$, and the congruence $M(-q)^{-16}\equiv M(q^{16})^{-1}\equiv M(q^2)^{-8}\mod 2$ agrees with the congruences $I_n(X)\equiv \tilde{I}_n(X)\mod 2$ for all $n$.}

It would be interesting, even in the case $k=\R$, to package the MNOP formula \eqref{equation: MNOP formula} with our Conjecture~\ref{conj:Main} to give a (conjectural) formula for the generating series $\sum_{n\geq 0}\tilde{I}^\GW_nq^n$, perhaps using one of the known power structures on $\GW(k)$. Note that the obstruction bundle $\text{ob}_1(X)$ on $\Hilb^1(X)=X$ is just the dual of $T_X\otimes K_X$, so one could replace the exponent $\deg(c_3(T_X\otimes K_X))$ in \eqref{equation: MNOP formula}  with $-\deg(c_3(\text{ob}_1(X)))$, which perhaps more closely resembles Conjecture~\ref{conj:Main} with $\epsilon=-1$. 

There is another approach to quadratic DT invariants, due to Espreafico and Walcher \cite{EspreaficoDT}. For a characteristic zero field $k$, they consider the local DT invariant for $\Hilb^n(\A^3_k)$ as an element in the Grothendieck ring of varieties $K_0(\Var_k)$, computed by Behrend, Bryan and Szendr\"oi \cite{BehrendMDT}, to which they apply an extended version 
\[
\tilde{\chi}_c(-/k):K_0(\Var_k)[\L^{-1/2}]\to \GW(k)[\langle-1\rangle^{1/2}].
\]
of the compactly supported $\A^1$-Euler characteristic of Arcila-Maya, Bethea, Opie, Wickelgren and Zakharevich \cite{ABOWZ}. This is also discussed in \cite[Introduction]{Viergever}. It would be interesting to see how to globalize this construction to give invariants for $\Hilb^n(X)$, with $X$ a smooth projective threefold with spin structure, to see if this approach would agree with the one via the motivic virtual fundamental classes given by \cite{LevineISNC, LevineOST}.

We remark that, under the hypothesis that $X$ is a smooth projective spin threefold,   $\deg(c_3(T_X\otimes K_X))$ is even, so 
\[
M(-q)^{\deg(c_3(T_X\otimes K_X))}\equiv M(q^2)^{\frac{1}{2}\deg(c_3(T_X\otimes K_X))}\mod 2 
\]
and all the odd DT invariants $I_n(X)$ vanish mod 2, consistent with our conjecture that $\tilde{I}_n(X)=0$ in $W(\R)$ for all odd $n$. In fact, we might expect that for smooth projective spin threefolds $X$, we have
\[
\deg(c_3(T_X\otimes K_X))=\epsilon\cdot 2\deg_\R(e(\text{ob}_2)),
\]
with $\epsilon$ as in  Conjecture~\ref{conj:Main}, in accordance with our result for the equivariant classes in Theorem~\ref{thm:MainBlowup}. 
\end{introrem}

\noindent{\bf Acknowledgments.}\ We would like to thank Rahul Pandharipande for generously taking the time to explain much of the foundational material and basic ideas underlying the papers \cite{MaulikI, MaulikII}. We would also like to thank Lukas Br\"oring for coming up with the algorithm that helped calculate some of the quadratic DT invariants with SAGE, which we used to check our results numerically.   
\\[5pt]
{\bf Notation}. For a noetherian scheme $S$, we let $\Sch/S$ denote the category of quasi-projective schemes over $S$, that is, $S$-schemes admitting a closed embedding into an open subscheme of a projective space $\P^n_S$ over $S$ for some $n$. We let $\Sm/S$ denote the full subcategory of smooth quasi-projective $S$-schemes. For a noetherian ring $A$, we write $\Sch/A$ and $\Sm/A$ for $\Sch/\Spec A$ and $\Sm/\Spec A$.

For $S$ a noetherian scheme of finite Krull dimension, we denote the motivic stable homotopy category over $S$ by $\SH(S)$ (or $\SH(A)$ for a noetherian ring of finite Krull dimension). We refer the reader to \cite{HoyoisSOEMHT} for the details on the construction of $\SH(S)$ and the six functor formalism for  $\SH(-)$ on  $\Sch/S$.  

Let $R=\oplus_{n\in\Z} R^n$ be a graded ring and $e\in R^d$ a homogeneous element. We write $R^n[1/e]$ for the degree $n$ homogeneous component of the localization $R[1/e]$.

Let $N$ be the normalizer of the diagonal torus $T_1\subset  \SL_2$. $N$ is generated by
\[
T_1:=\left\{\begin{pmatrix} t & 0 \\ 0 & t^{-1}\end{pmatrix}\right\}\subset \SL_2 \text{ and }\sigma := \begin{pmatrix} 0 & 1 \\ -1 & 0\end{pmatrix}.
\]
We write $t\in T_1$ to denote the matrix $\begin{pmatrix} t & 0 \\ 0 & t^{-1}\end{pmatrix}\in \SL_2$. 
\endgroup
\section{Quadratic Donaldson-Thomas invariants}
Let $k$ be a perfect field of characteristic $\neq2$, let $X$ be a smooth projective threefold over $k$ and let $K_X$ be the canonical line bundle on $X$.  We briefly recall from \cite[Section 1]{Viergever} how one constructs quadratic Donaldson-Thomas invariants $\tilde{I}_n(X)\in W(k)$ of $X$ using work of Levine, see \cite{LevineISNC}, \cite{LevineVLEWC} and \cite{LevineOST}. Since our quadratic Donaldson-Thomas  invariants will live in the Witt ring of our base-field, we begin by recalling some basic facts on this topic.
 
 \subsection{The Witt ring, the Witt sheaf and the Witt spectrum}\label{subsec:Witt}	
		
For a  noetherian scheme $X$, we let $W(X)$ denote the Witt ring of $X$. This  is the group completion of the monoid (under direct sum) of isometry classes of symmetric non-degenerate  $\sO_X$-bilinear forms $b:\sV\otimes_{\sO_X}\sV\to \sO_X$ on locally free coherent sheaves on $X$,  modulo the classes of {\em metabolic forms}, where $(\sV,b)$ is metabolic if $\sV$ admits a short exact sequence 
\[
	0\to \sL\to \sV\to \sL'\to 0
\]
with $\sL$, $\sL'$ locally free, such that $\sL=\sL^\perp$, where $\sL^\perp$ is the kernel of the map $\sV\to \sHom_{\sO_X}(\sL, \sO_X)$ induced by $b$.    Here we use the definition given by Balmer \cite[Definition 1.1.27]{Balmer} for the exact category of locally free coherent sheaves on $X$, with duality $\sV^\vee:=\sHom_{\sO_X}(\sV, \sO_X)$, going back to Knebusch's original definition \cite[\S 3.5]{Knebusch} (phrased slightly differently) of the Witt ring of a scheme. This in turn extends Witt's original definition of the Witt ring of a field \cite{Witt} to the case of schemes; see also \cite[\S7]{MilnorSBF} for a detailed discussion of the Witt ring of a commutative ring. 

For a noetherian commutative ring $R$, we write $W(R)$ for $W(\Spec R)$. We note that in case 2 is invertible in $R$, then $(P,b)$ is metabolic if and only if $(P,b)$ is hyperbolic, that is, $(P,b)$ is isometric to a space of the form $(Q\oplus Q^\vee, \phi_{can})$ with $\phi_{can}((v,\ell), (v',\ell'))=\ell(v')+\ell'(v)$.

The ring structure on $W(X)$ is induced from the tensor product over $\sO_X$, with $b\cdot b'(v\otimes w, v'\otimes w')=b(v,v')b'(w,w')$. Sending $X$ to $W(X)$ defines a presheaf of rings on noetherian schemes,  yielding in particular its sheafication $\sW$ on $\Sm/k_\Nis$.  We are assuming here that 2 is invertible in $k$, so metabolic forms are hyperbolic in the affine case.

Noting that for $(\sV,b)$ a metabolic form on $X$, the locally free sheaf $\sV$ has even rank on each connected component of $X$ gives us the rank homomorphism $\rnk:W(X)\to  H^0(X,\Z/2)$, taking $(\sV,b)$ to the rank of $\sV$ modulo 2. We recall that for $k$ an algebraically closed field $\rnk:W(k)\to \Z/2$ is an isomorphism. For $k$ a real closed field (e.g., $k=\R$), Sylvester's theorem of inertia gives us the signature map $\sgn:W(k)\to \Z$, which is an isomorphism \cite[Definitions 4.5, 4.7, Corollary 4.8]{Scharlau}.

There is a $\P^1$-spectrum $\EM(\sW)\in \SH(k)$ that represents Witt sheaf cohomology on $\Sm/k$, via a canonical and natural isomorphism of presheaves on $\Sm/k$, 
\[
\EM(\sW)^{a,b}(X)=H^{a-b}(X_\Nis,\sW).
\]
  We note that $H^0(\Spec(k),\sW) = W(k)$.

\subsection{Motivic virtual fundamental classes} 
Let $Z$ be a quasi-projective scheme over $k$. Let $E_\bullet$ be a perfect obstruction theory over $Z$, i.e. a perfect complex supported in (homological) degrees $0$ and $1$ together with a morphism $\phi_\bullet:E_\bullet\to \mathbb{L}_{Z/k}$ to the cotangent complex $\mathbb{L}_{Z/k}$ of $Z$ over $k$, such that $\phi_\bullet$ induces an isomorphism on homology sheaves in degree zero and a surjection in degree $1$. 
	
Let $\sE$ be an $\SL$-oriented ring spectrum, see for instance \cite{AnanyevskiySLPBT} for a definition and properties; we recall that  the $\P^1$-spectrum $\EM(\sW_*)$ is canonically $\SL$-oriented (see, e.g., \cite[Example 3.1]{LevineABLEWC} for details).

In the paper \cite{LevineISNC}, Levine defines a motivic virtual fundamental class $[Z,\phi_\bullet]^\vir_\sE\in \sE^{B.M.}_{0,0}(Z,\V(E_\bullet))$ associated to $E_\bullet$ on $Z$, where $\V(E_\bullet)$ is the virtual vector bundle associated to $E_\bullet$. Since $\sE$ is $\SL$-oriented, we have a canonical isomorphism $\sE^{B.M.}_{0,0}(Z,\V(E_\bullet))\cong \sE^{B.M.}_{2r,r}(Z,\V(\det E_\bullet)-\V(\sO_Z))$, where $r$ is the virtual rank of $E_\bullet$.
		
 An  {\em orientation} of $E_\bullet$ is the data of an isomorphism $\rho:\det(E_\bullet)\xrightarrow{\sim} L^{\otimes 2}$, where $L$ is a line bundle on $Z$.  An orientation $\rho$ gives us the isomorphism
 \[
 \sE^{B.M.}_{0,0}(Z,\V(E_\bullet))\cong \sE^{B.M.}_{2r,r}(Z,\V(\det E_\bullet)-\V(\sO_Z))\xymatrix{\ar[r]^{\rho_*}_\sim&}\sE^{B.M.}_{2r,r}(Z).
 \]
If we further suppose that $E_\bullet$ has virtual rank 0 and the structure map $p_Z:Z\to \Spec(k)$ is proper, we have the pushforward map 
\[
p_{Z_*}:\sE^{B.M.}_{0,0}(Z)\to \sE^{B.M.}_{0,0}(k)=\sE^{0,0}(k).
\]
The resulting composition $\deg_k^\rho\colon \sE^{B.M.}_{0,0}(Z,\V(E_\bullet))\to \sE^{0,0}(k)$ gives us the {\em degree} of the motivic virtual fundamental class
\[
\deg_k^\rho([Z,\phi_\bullet]^\vir_\sE)\in  \sE^{0,0}(k).
\]

In  \cite{LevineVLEWC}, this construction was taken over to the equivariant setting, in the case where $Z$ has an action of a smooth closed subgroup of $\GL_n$. A virtual localization formula for this situation is proven in the same paper (see \cite[Theorem 1]{LevineVLEWC}).
	
\subsection{Quadratic DT invariants of smooth projective threefolds}\label{subsec:MotDTInv}

Let $X$ be a smooth projective threefold, and for $n\geq 1$, let $\Hilb^n(X)$ be the Hilbert scheme of ideal sheaves of length $n$ with support of dimension zero on $X$. There is a perfect obstruction theory of virtual rank 0,  $\phi_{DT,n}:E^{DT(n)}_\bullet\to \L_{\Hilb^n(X)/k}$ on $\Hilb^n(X)$, which was first constructed by Thomas \cite{ThomasHCI} and given a concise reformulation in the language of the derived category by Behrend-Fantechi \cite[Lemma 1.22]{Behrend-FantechiSOTHSPT}.  We call an isomorphism $\tau:K_X\xrightarrow{\sim} L^{\otimes 2}$ a {\em spin structure} on $X$. By \cite[Theorem 6.3]{LevineOST},  a  spin structure $\tau$ on $X$ induces an orientation $\rho^\tau_n$ on $E^{DT(n)}_\bullet$ for every $n\ge1$. 
	
We can now make the following definition. 
	
\begin{defn}
Let $X$ be a smooth projective threefold with a spin structure $\tau:K_X\xrightarrow{\sim} L^{\otimes 2}$. Let $n\geq 1$. The $n$th {\em quadratic DT invariant} $\tilde{I}_n(X)\in W(k)$ is the degree of the $\EM(\sW_*)$-valued motivic virtual fundamental class of $\Hilb^n(X)$:
\[
\tilde{I}_n(X):=\deg_k^{\rho^\tau_n}([\Hilb^n(X), \phi_{DT,n}]^\vir_{\EM(\sW_*)})\in W(k).
\]		
\end{defn}
	
\begin{rem} In general, it is not clear to what extent the invariant $\tilde{I}_n(X)$ depends on the choice of the spin structure $\tau:K_X\xrightarrow{\sim} L^{\otimes 2}$. However, if $n$ is even, and the choice of line bundle $L$ is fixed, then it follows from the construction of the orientation in \cite{LevineOST} that the orientation $\rho^\tau_n$ is independent of the choice of isomorphism $\tau:K_X\to L^{\otimes 2}$. Furthermore, for $M$ a 2-torsion bundle, changing $\tau$ to the composition $K_X\xrightarrow{\tau} L^{\otimes 2}\cong (L\otimes M)^{\otimes 2}$ also leaves   $\tilde{I}_n(X)$ unchanged. Thus, for even $n$, $\tilde{I}_n(X)$ is independent of the choice of spin structure on $X$. 
\end{rem}

\section{Equivariant quadratic DT invariants} Let $G$ be a smooth affine group-scheme over $k$. 
We have a version of the quadratic Donaldson-Thomas  invariants in the $G$-equivariant setting, which play a central role in our computations of $\tilde{I}_n(X)$ by using virtual localization. For this, we use $G$-equivariant $\sE$-cohomology $\sE^{*,*}_G(-,-)$ and equivariant Borel-Moore homology,  $\sE^\BM_{G,(*,*)}(-,-)$, defined using the method of Totaro \cite{TotaroCRCS},  Edidin-Graham \cite{EdidinGrahamEIT} and Morel-Voevodsky \cite[\S 4]{MorelVoevodskyAHTS}, adapted to $\sE$-cohomology and $\sE$-Borel-Moore homology; see \cite[\S 4]{LevineABLEWC} and \cite[\S 3]{LevineVLEWC} for details. We specialize to the case of interest  $G=N$. We first recall some facts about the representation theory of $N$, and the Witt sheaf cohomology of the classifying space $BN$.

\subsection{Some representation theory} \label{subsec:NRep}  Each representation $\rho$ of $N$ on a $k$-vector space $V$ of dimension $n$ gives rise to a rank $n$ vector bundle  $V_\rho$ on $BN$ and a corresponding Euler class $e_N(V_\rho)\in H^n(BN, \sW(\det^{-1} V_\rho))$; our computation of $\tilde{I}_n^N$ relies on knowing the $N$-equivariant Euler classes for irreducible representations of $N$. Here we recall some basic facts about the representation theory of $N$ and the computation of the Euler class $e_N(V_\rho)$ for each irreducible $N$-representation $\rho$.  

As described in \cite[\S 6]{LevineAEGQF}, the group scheme $N$ has the following series of irreducible representations on finite dimensional $k$-vector spaces, where $k$ is our base-field of characteristic $\neq2$.
\begin{lem} Let $k$ be a field of characteristic $\neq2$. The group-scheme $N$ has the following irreducible representations over $k$.
\begin{enumerate}
\item The trivial representation $\rho_0:N\to \GL_1$, and the sign representation $\rho_0^-:N\to \GL_1$, with $\rho_0^-(T_1)=\{\id\}$ and $\rho_0^-(\sigma)=-\id$.
\item For $m$ a non-zero integer,  the representation $\rho_m:N\to \GL_2$,
\[
\rho_m(t)=\begin{pmatrix}t^m&0\\0&t^{-m}\end{pmatrix},  
\rho_m(\sigma)=\begin{pmatrix}0&1\\(-1)^m&0\end{pmatrix}.
\]
\end{enumerate}
For $m\neq0$, the representations $\rho_m$ and $\rho_{-m}$ are isomorphic and
 the set of representations $\{\rho_0, \rho_0^-, \rho_m\mid m>0\}$ form a complete set of representatives for isomorphism classes of irreducible $N$-representations in $k$-vector spaces.
\end{lem}

\begin{proof}
Let $V$ be a $k$-vector space and $\rho:N\to \Aut_k(V)$ an irreducible representation. We decompose $V$ into irreducible representations for $T_1$ as $V=\oplus_m O(m)^{n_m}$, where $t\in T_1$ acts by $t^m$ on the one-dimensional $k$-vector space $O(m)$. Since $\sigma\cdot t=t^{-1}\cdot \sigma$, $\rho(\sigma)$ restricts to an isomorphism $\rho(\sigma)_m:O(m)^{n_m}\to O(-m)^{n_{-m}}$ for each $m$, in particular, $n_m=n_{-m}$. Also, for $m\neq0$,  the subspace $O(m)^{n_m}\oplus O(-m)^{n_{-m}}$ is $N$-stable, and for $m=0$, the subspace $O(0)^{n_0}$ is $N$-stable, so since $V$ is irreducible, this means there is a unique $m\ge0$ with $n_m>0$.  Let $n=n_m$.

Suppose that $m=0$. Then $V=O(0)^n$, $T_1$ acts trivially on $V$, and thus the $N$-action factors through an action of $N/T_1=\Z/2$. Since we are assuming the characteristic is not 2, this implies that $n=1$ and there are exactly two possible actions: either $\sigma$ acts by the identity or $\sigma$ acts by $-\id$. This yields the irreducible representations $\rho_0$ and $\rho_0^-$; clearly $\rho_0$ is not isomorphic to $\rho_0^-$. 

If $m>0$,  choose a basis $e_1,\ldots, e_n$ for $O(m)^n$, then the rank two subspace spanned by $e_1$ and $f_1:=\rho(\sigma)(e_1)$ is an $N$-summand of $V$, so $V=ke_1\oplus k f_1$, with $\rho(t)(e_1)=t^m e_1$, $\rho(t)(f_1)=t^{-m}f_1$ and $\rho(\sigma)(f_1)=\rho(\sigma^2)(e_1)=\rho(-1)(e_1)=(-1)^me_1$. Taking the basis $e=e_1$, $f=(-1)^mf_1$ for $V$ shows that $\rho$ is isomorphic to $\rho_m$. 

The representation $\rho_{-m}$  is isomorphic to $\rho_m$, by conjugating $\rho_m$ by the matrix 
\[
\begin{pmatrix}0&1\\(-1)^m&0\end{pmatrix}. 
\]
Since the character of $T_1$ for $\rho_m$ is different from that of $\rho_{m'}$ if $m>m'>0$, the two-dimensional representations $\rho_m$, $m>0$, together with the one-dimensional representations $\rho_0$, $\rho_0^-$, form a complete set of representatives of isomorphism classes of irreducible representations of $N$. 
\end{proof}

\begin{nota}
For $m\neq0$, we let $\tilde{\sO}(m)$ denote the rank two vector bundle on $BN$ corresponding to $\rho_m$, and let $\gamma$ denote the line bundle on $BN$ corresponding to $\rho_0^-$. 
\end{nota}

\begin{rem} Since the Picard group of $BN$ is isomorphic to the group of 1-dimensional representations of $N$ over $k$ (under tensor product), we see that $\Pic(BN)\cong \Z/2$, with generator $\gamma$.
\end{rem}

\begin{defn}\label{defn:StandardBasis}
For $m\neq0$, we call a basis $e,f$ for the representation space of $\rho_m$ a {\em standard basis} if $\rho_m(t)(e)=t^m e$, $\rho_m(t)(f)=t^{-m} f$ and $\rho_m(\sigma)(e)=(-1)^mf$.
\end{defn}
For example, the basis $e,f$  of $k^2$,  $e=\begin{pmatrix}1\\0\end{pmatrix}$, $f=\begin{pmatrix}0\\1\end{pmatrix}$, is a standard basis for  $\rho_m$. If we have two standard bases $(e,f)$, $(e',f')$, then for a unique $\lambda\in k^\times$, we have $e'=\lambda e$, $f'=\lambda f$. 

A standard basis $e,f$ for $\rho_m$ gives the basis $e\wedge f$ for $\det\rho_m$, which gives the isomorphism $\det\rho_m\cong \rho_0$ for $m$ odd and $\det\rho_m\cong\rho_0^-$ for $m$ even. These isomorphisms are thus canonical up to multiplication by some $\lambda^2$, $\lambda\in k^\times$, so define canonical isomorphisms of sheaves on $BN$,
\begin{align}\label{align:DetIs}
&\sW(\det^{-1}\tilde{\sO}(m))\cong \sW&\text{for $m$ odd,}\\
&\sW(\det^{-1}\tilde{\sO}(m))\cong \sW(\gamma)&\text{for $m$ even.}\notag
\end{align}

We have the equivariant Euler class $e_N(\tilde{\sO}(m))\in H^2(BN, \sW(\det^{-1}\tilde{\sO}(m)))$. Using the canonical isomorphisms \eqref{align:DetIs}, this gives
\[
e_N(\tilde{\sO}(m))\in \begin{cases} H^2(BN, \sW) &\text{ for $m$ odd,}\\
H^2(BN, \sW(\gamma)) &\text{ for $m$ even.}
\end{cases}
\]

\subsection{Equivariant Witt sheaf cohomology}\label{subsecEquivWittCoh}

For a smooth $k$-scheme $Y$ with $N$-action,  and an $N$-linearized invertible sheaf $\sL$ on $Y$, we write 
\[
H^\BM_{N,*}(Y, \sW(\sL)):=\EM(\sW)^\BM_{N, (*,0)}(Y, \sL-\sO_Y), 
\]
for the $N$-equivariant Borel-Moore homology and 
\[
H^*_N(Y, \sW(\sL)):=\EM(\sW)^{*,0}_N(Y, \sL-\sO_Y)
\]
 for the $N$-equivariant cohomology. For $Y=\Spec(k)$ with trivial $N$-action, we often write $H^\BM_*(BN, \sW(\sL))$ and $H^*(BN, \sW(\sL))$ for  $H^\BM_{N,*}(\Spec(k), \sW(\sL))$ and $H^*_N(\Spec(k), \sW(\sL))$.

The $N$-equivariant quadratic Donaldson-Thomas  invariants will live in $H^*(BN, \sW(\sL))$, and are built out of the Euler classes of $N$-representations, so we recall here some relevant facts.

\begin{nota}\label{nota:e}
We let $e=e_N(\tilde{\sO}(1))\in H^2(BN, \sW)$, $\tilde{e}=e_N(\tilde{\sO}(2))\in H^2(BN, \sW(\gamma))$. Also, define the map $\epsilon:\Z\to \{\pm 1\}$ by $\epsilon(m)=+1$ for $m\equiv 1,2\mod 4$, $\epsilon(m)=-1$ for $m\equiv 0,3\mod 4$.
\end{nota}

Pullback via the structure map $p:BN\to \Spec k$ makes the full coefficient ring 
$H^*(BN, \sW)\oplus H^*(BN, \sW(\gamma))$ a $W(k)$-algebra. 

\begin{prop} \label{prop:BNCoh} 1. There is an isomorphism of graded
 $W(k)$-algebras
 \[
 W(k)[x,y]/(x^2-1, (x+1)y)\cong  H^*(BN, \sW)
 \]
 with $x$ of degree 0, $y$ of degree 2, and sending $y$ to $e$; the image of $x$ under this isomorphism is denoted $\<Q\>$. In particular, $H^0(BN,\sW)\cong W(k)\cdot 1\oplus W(k)\cdot \<Q\>\cong W(k)^2$ as $W(k)$-module. \\[2pt]
 2. As $H^*(BN, \sW)$-module,  $H^{*\ge2}(BN, \sW(\gamma))$ is generated by $\tilde{e}$, with the single relation $(1+\<Q\>)\tilde{e}=0$. 
 3. The $W(k)$-algebra map $W(k)[e, \tilde{e}]\to H^*(BN, \sW)\oplus H^*(BN, \sW(\gamma))$ is an inclusion, and an isomorphism in degrees $\ge1$. $H^0(BN, \sW)$ is the free $W(k)$-module $W(k)\cdot 1\oplus W(k)\cdot\<Q\>$.  $H^0(BN, \sW(\gamma)$ is the free $W(k)$-module $W(k)\cdot\tilde{Q}$, for a particular element $\tilde{Q}\in H^0(BN, \sW(\gamma)$,  with $\tilde{Q}\cdot e=\tilde{e}$, and $\tilde{Q}\cdot \<Q\>=-\tilde{Q}$.\\[2pt]
4. The map $W(k)[e, \tilde{e}]\to H^*(BN, \sW)\oplus H^*(BN, \sW(\gamma))$ induces an isomorphism $W(k)[e, e^{-1}, \tilde{e}, \tilde{e}^{-1}]\to (H^*(BN, \sW)\oplus H^*(BN, \sW(\gamma)))[1/e, 1/\tilde{e}]$, and the  localization map  $H^*(BN, \sW)\oplus H^*(BN, \sW(\gamma))\to (H^*(BN, \sW)\oplus H^*(BN, \sW(\gamma)))[1/e, 1/\tilde{e}]$ sends $\<Q\>$ to $-1$.
 \end{prop}
 This is  \cite[Theorem 11.15]{LevineABLEWC}; see also  \cite[Theorem 5.1]{LevineABLEWC}, \cite[Proposition 5.5, Lemma 6.1]{LevineMEC}.\footnote{Note that the statement of \cite[Proposition 5.5]{LevineMEC} is corrected in \cite[Theorem 5.1]{LevineABLEWC}, but the correction only involves the degree 0 component.}  
 
The equivariant Euler classes of the bundles $\tilde{\sO}(m)$ are given as follows.
\begin{lem}\label{lem:NEulerClass} For $m\neq0$, we have
\[
e_N(\tilde{\sO}(m))=\begin{cases} \epsilon(m)m\cdot e&\text{ for $m$ odd,}\\
\epsilon(m)(m/2)\cdot \tilde{e}&\text{ for $m$ even.}
\end{cases}
\]
\end{lem}

\begin{proof} For $m>0$, this is  \cite[Theorem 7.1]{LevineMEC}. Now take $m<0$ and odd. We have the isomorphism $\phi_m:\rho_m\to \rho_{-m}$ defined by $\phi_m(e_1)=-e_2$, $\phi_m(e_2)=e_1$. Since $\phi_m(e_1\wedge e_2)=e_1\wedge e_2$, $\det\phi_m$ is compatible with our chosen isomorphisms $\det \rho_m\cong \rho_0\cong \det\rho_{-m}$. 
By the functoriality of the Euler class we thus have 
\[
e_N(\tilde{\sO}(m))=\det\phi_m(e_N(\tilde{\sO}(-m)))=e_N(\tilde{\sO}(-m))=\epsilon(-m)(-m)\cdot e=\epsilon(m)m\cdot e.
\]

For $m<0$ and even, we have the isomorphism $\phi_m:\rho_m\to \rho_{-m}$ defined by $\phi_m(e_1)=e_2$, $\phi_m(e_2)=e_1$. Then $\phi_m(e_1\wedge e_2)=-e_1\wedge e_2$, so
\begin{multline*}
e_N(\tilde{\sO}(m))=\det\phi_m(e_N(\tilde{\sO}(-m)))=\langle-1\rangle e_N(\tilde{\sO}(-m))\\=-\epsilon(-m)(-m/2)\cdot \tilde{e}=\epsilon(m)(m/2)\cdot \tilde{e}.
\end{multline*}
Here we use the fact that in the Witt ring, we have $\langle-1\rangle=-1$.
\end{proof}

\subsection{Equivariant quadratic Donaldson-Thomas  invariants}\label{subsec:EMDT}

Let $X$ be a smooth projective threefold with an $N$-action, and with an $N$-linearized  spin structure $\tau:K_X\xrightarrow{\sim}L^{\otimes 2}$, that is, the line bundle $L$ is given an $N$-linearization such that the spin structure $\tau$ is $N$-equivariant, where we give $K_X$ the canonical $N$-linearization induced from the $N$-action on $X$. We assume in addition that the very ample line bundle $\sO_X(1)$ giving the projective embedding of $X$ is given an $N$-linearization. Then, by  \cite[Theorem 6.4]{LevineOST},  the $N$-action extends to $\Hilb^n(X)$ and gives $E^{DT(n)}_\bullet$ an $N$-linearization for which $\phi_{DT,n}$ is $N$-equivariant. It follows from the construction of the orientation $\rho^\tau_n$ given in \cite[Theorem 6.3]{LevineOST} that the $N$-equivariance of the spin structure $\tau$ implies that $\rho^\tau_n$ is $N$-equivariant as well. We then have (following \cite[\S 5]{LevineVLEWC}) an $N$-equivariant virtual fundamental class $[\Hilb^n(X), \phi_{DT,n}]^\vir_{N,\EM(\sW_*)}]\in H^\BM_{N,0}(\Hilb^n(X), \sW)$, with values in the $N$-equivariant Borel-Moore homology of $\EM(\sW)$.

\begin{defn} Let $X$ be a smooth projective threefold  with spin structure $\tau:K_X\xrightarrow{\sim} L^{\otimes 2}$. Let $n\geq 1$. Suppose that $X$ has an $N$-action for which $\sO_X(1)$ has an $N$-linearization and that $L$ has an $N$-linearization  making $\tau$ $N$-equivariant. The $n$th {\em equivariant quadratic DT invariant} $\tilde{I}_n^N(X)\in H^0(BN, \sW)$ is the degree of the equivariant motivic virtual fundamental class of $\Hilb^n(X)$:
\[
\tilde{I}_n^N(X):=p_{\Hilb^n(X)*}([\Hilb^n(X), \phi_{DT,n}]^\vir_{N,\EM(\sW_*)})\in H^0(BN, \sW).
\]
Here $p_{\Hilb^n(X)*}: H^\BM_{N,0}(\Hilb^n(X), \sW)\to H^\BM_{N,0}(\Spec(k), \sW)=H^0(BN, \sW)$ is the pushforward with respect to the structure map $p_{\Hilb^n(X)}: \Hilb^n(X)\to \Spec k$.
\end{defn}

\subsection{Relating $\tilde{I}_n(X)$ and  $\tilde{I}^N_n(X)$}  We retain the setting and notation of \S\ref{subsec:EMDT}. 

 We have the structure morphism $p:BN\to \Spec k$ and the ``quotient'' map $\pi:\Spec k\to BN$, inducing maps
\[
W(k)=H^0(\Spec k, \sW)\xrightarrow{p^*} H^0(BN, \sW)\xrightarrow{\pi^*}H^0(\Spec k, \sW)=W(k).
\]
Since $\pi^*\circ p^*=\id$, the pullback $p^*$ canonically identifies $W(k)$ with a summand of $H^0(BN, \sW)$; as $\pi^*(\<Q\>)=0$, this is exactly the decomposition  $H^0(BN, \sW)=W(k)\cdot 1\oplus W(k)\cdot \<Q\>$ given in Proposition~\ref{prop:BNCoh}(3).

The map $\pi^*$ is the ``forget the $N$-action'' map, which implies that
\begin{equation}\label{eqn:ForgetMap}
\pi^*(\tilde{I}^N_n(X))=\tilde{I}_n(X)\in W(k).
\end{equation}
However, as we shall see below, the $N$-localization method computes for us the image of 
$\tilde{I}^N_n(X)$ in a localization $H^*(BN, \sW)[1/m_ne]$ for some integer $m_n>0$, so we would like to know to what  this image of $\tilde{I}^N_n(X)$ tells us about our invariant of interest, $\tilde{I}_n(X)$. 

Let $m>0$ be an integer and let  $H^*(BN, \sW)[1/me]^0$ denote the degree zero part of $H^*(BN, \sW)[1/me]$. By Proposition~\ref{prop:BNCoh}(4), we have
\begin{equation}\label{eqn:LocalizedCoh}
H^*(BN, \sW)[1/me]=W(k)[e, e^{-1}, 1/m], \ H^*(BN, \sW)[1/me]^0=W(k)[1/m].
\end{equation}

\begin{prop}\label{prop:Splitting Identity}  Let $m\ge1$ be an integer.  Let $X$ be a smooth projective  threefold with an $\SL_2^m$-action, and that $\sO_X(1)$ is given an $\SL_2^m$-linearization. Suppose in addition we have an $\SL_2^m$-linearized line bundle $L$ and an $\SL_2^m$-equivariant spin structure $\tau_{\SL_2}:K_X\to L^{\otimes 2}$, where we give $K_X$ the $\SL_2^m$-linearization induced from the $\SL_2^m$-action on $X$.

Let $\iota:N\to \SL_2^m$ be a closed immersion of group-schemes. By restricting the $\SL_2^m$-action on $X$, the various $\SL_2^m$-linearizations and  $\SL_2^m$-equivariant spin structure $\tau_{\SL_2}$, we have an $N$-action on $X$ and the necessary data to define the Donaldon-Thomas invariants $\tilde{I}_n(X)\in W(k)$ and the $N$-equivariant versions $\tilde{I}_n^N(X)\in H^0(BN, \sW)$. Then
\[
\tilde{I}_n^{N} = p^*\tilde{I}_n\in W(k)\cdot 1\subset H^0(\Spec(k),\sW). 
\]
Moreover, for any integer $m>0$, the image of $\tilde{I}_n^{N}$ in the localization  
$H^*(BN, \sW)[1/me]^0=W(k)[1/m]$ is equal to the image of $\tilde{I}_n$ in $W(k)[1/m]$ under the localization map $W(k)\to W(k)[1/m]$
\end{prop}

\begin{proof} The first assertion  is \cite[Proposition 6.9]{LevineVLEWC}, the second follows by applying the localization map $H^*(BN, \sW)\to H^*(BN, \sW)[1/me]$ and using \eqref{eqn:LocalizedCoh}.
\end{proof} 
In short,  we can calculate the $\tilde{I}_n$ from the equivariant versions $\tilde{I}_n^N$. 

\begin{rem} \label{rem:LocId} The identity of Proposition~\ref{prop:Splitting Identity} is not purely formal. Indeed, without any additional information, we have by Proposition~\ref{prop:BNCoh}
\[
\tilde{I}_n^N(X)\in H^0(BN, \sW)=W(k)\cdot 1\oplus W(k)\cdot\<Q\>=p^*(W(k))\oplus W(k)\cdot\<Q\>.
\]
 If we write $\tilde{I}_n^N(X)=p^*a\oplus b\cdot \<Q\>$, then since $\<Q\>$ maps to $-1$ under the localization $H^*(BN, \sW)\to H^*(BN, \sW)[1/e]$, we see that the image of $\tilde{I}_n^N(X)$ in $H^*(BN, \sW)[1/me]^0$ is the image of $a-b$ in $W(k)[1/m]$, while 
\[
\tilde{I}_n(X)=\pi^*(\tilde{I}_n^N(X))=\pi^*(p^*a+b\cdot\<Q\>)=a
\]
since $\pi^*(\<Q\>)=0$. Thus, without knowing that $\tilde{I}_n^N(X)$ lives in the summand $p^*W(k)$ of $H^0(BN, \sW)$, knowing the image of $\tilde{I}_n^N(X)$ in a localization 
$H^*(BN, \sW)[1/me]$ tells us nothing about the invariant $\tilde{I}_n(X)$. As the virtual localization method requires passing to such a localization, the added information given in 
Proposition~\ref{prop:Splitting Identity} will be crucial for our computation of $\tilde{I}_n((\P^1)^3)$. As we shall see, the $N$-action and related structures for $X=(\P^1)^3$ all arise via restriction with respect to a suitable homomorphism $\iota_{a,b,c}N\to \SL_2^3$, allowing us to apply Proposition~\ref{prop:Splitting Identity}.

Similarly, since we have essentially no control over the integer $m$, we loose all torsion information in $W(k)$, so we end up only being able to compute the image in $W(k)_\Q$, which is given by signature information with respect to embeddings of $k$ in real closed fields.
\end{rem}

\section{The equivariant vertex measure}\label{sec:EquivVertex}
	
	We continue to work over our perfect base field $k$ of characteristic $\neq2$. Giving $\P^1$ homogeneous coordinates $X_0, X_1$, we have the generating section $\Omega=X_0dX_1-X_1dX_0$ of $K_{\P^1}(2)$, giving the isomorphism $\sO_{\P^1}(-2)\cong K_{\P^1}$ by sending a local section $\lambda$ of $\sO_{\P^1}(-2)$ to $\lambda\cdot\Omega$. Letting $p_i:(\P^1)^3\to \P^1$ denote the $i$'th projection,  we have the generating section $p_1^*\Omega\wedge p_2^*\Omega\wedge p_3^*\Omega$ of $K_{(\P^1)^3}(2,2,2)$, giving us the isomorphism $K_{(\P^1)^3}\cong  \sO(-2,-2,-2)=\sO(-1,-1,-1)^{\otimes 2}$, that is, we have a canonical choice of an orientation $\tau:K_{(\P^1)^3}\xrightarrow{\sim}\sO(-1,-1,-1)^{\otimes 2}$ on $(\P^1)^3$.  Therefore, the quadratic DT invariants $\tilde{I}_n(\P^1)^3$ are defined. 
	
\begin{nota}  For $n\geq 1$, let $\tilde{I}_n\in W(k)$ be the $n$th quadratic DT invariant of $(\P^1)^3$. 
\end{nota} 	
	
	We now discuss how to calculate the $\tilde{I}_n$ using the method of \cite{MaulikI}. The same method was used in \cite{Viergever} to calculate the first three nonzero quadratic DT invariants of $\P^3$. 
	
\subsection{The $N$-action on $\Hilb^n((\P^1)^3)$}\label{subsec:NActionP13}
	
Let $a, b, c$ be odd integers. We have the homomorphism
\[
\iota'_{a,b,c}:N\to N^3
\]
with 	$\iota'_{a,b,c}(t)=(t^a, t^b, t^c)$ and $\iota'_{a,b,c}(\sigma)=(\sigma,\sigma,\sigma)$. We let $\iota_{a,b,c}:N\to (\SL_2)^3$ be the composition of $\iota'_{a.b.c}$ with the inclusion $N^3\hookrightarrow (\SL_2)^3$. 

	The standard action of $\SL_2$ on $\P^1$ by fractional linear transformations arises from the standard representation of $\SL_2$ on $\A^2$, hence lifts canonically to an $\SL_2$-linearization of $\sO_{\P^1}(1)$. Moreover,  $\Omega$ is an $\SL_2$-invariant section of $K_{\P^1}(2)$. Thus, the product action of $(\SL_2)^3$ on $(\P^1)^3$ linearizes the very amble invertible sheaf $\sO(1,1,1)$, and 
$\tau$ is $(\SL_2)^3$-equivariant. We can then restrict to $N$ via $\iota_{a,b,c}$, giving us the $N$-action on $(\P^1)^3$ with linearization of $\sO(1,1,1)$, and such that the spin structure $\tau$ is $N$-equivariant. Explicitly, this action of $N$ on $(\P^1)^3$ is given by 
	\begin{align*}
		t\cdot ([X_0:X_1],[Y_0:Y_1],[Z_0:Z_1]) &=( [t^aX_0: t^{-a}X_1], [t^bY_0: t^{-b}Y_1], [t^cZ_0: t^{-c}Z_1])\\
		\sigma \cdot ([X_0:X_1],[Y_0:Y_1],[Z_0:Z_1]) &= ([-X_1:X_0],[-Y_1:Y_0],[-Z_1:Z_0])
	\end{align*}

\begin{rem}\label{rem:Extension} Using $(\SL_2)^3$-action on $(\P^1)^3$ and the $(\SL_2)^3$-linearizations and $(\SL_2)^3$-equivariant spin structure described above, the homomorphism $\iota_{a,b,c}:N\to (\SL_2)^3$ gives by restriction all the data needed to define the $N$-equivariant Donaldson-Thomas invariants $\tilde{I}_n^N((\P^1)^3)\in H^0(BN, \sW)$. We may also apply Proposition~\ref{prop:Splitting Identity} to use $\tilde{I}_n^N((\P^1)^3)$ to compute $\tilde{I}_n((\P^1)^3)\in W(k)$. 

 {\it A priori} $\tilde{I}_n^N((\P^1)^3)$ depends on the choice of embedding $\iota_{a,b,c}:N\to \SL_2^3$ used to define the $N$-action. However, it follows from Proposition~\ref{prop:Splitting Identity} that in fact $\tilde{I}_n^N((\P^1)^3)=p^*\tilde{I}_n((\P^1)^3)$, where $p:BN\to\Spec k$ is the structure morphism. In particular,  $\tilde{I}_n^N((\P^1)^3)$ is independent of the choice of embedding $\iota_{a,b,c}$.  
 \end{rem}

\begin{nota} We denote by $\tilde{I}_n^N\in H^0(BN,\sW)$  $n$th $N$-equivariant quadratic DT invariant $\tilde{I}_n^N((\P^1)^3)$ of $(\P^1)^3$. 
\end{nota}

\begin{nota}  We use homogeneous coordinates $([X_0:X_1],[Y_0:Y_1],[Z_0:Z_1])$ on $(\P^1)^3$. We cover $(\P^1)^3$ by the standard cover $\{U_{klm}\}_{k,l,m\in\{0,1\}}$ where $U_{klm} = \{X_kY_lZ_m\neq 0\}$.  This identifies $U_{klm}$ with $\A^3=\Spec k[x,y,z]$, via the standard affine coordinates  $x:=X_{1-k}/X_k$, $y:=Y_{1-l}/Y_l$, $z:=Z_{1-m}/Z_m$. We let $0_{klm}$ denote the corresponding origin in $U_{klm}$.  The set $\{0_{klm}\}_{k,l,m\in \{0,1\}}$ is the set of $T_1$-fixed points in $(\P^1)^3$ with respect to the embedding $\iota_{a,b,c}$, for any choice of  odd integers $a, b,c$. 

For an ideal sheaf $\sI$ on $(\P^1)^3$, we let $\sI_{klm}$ denote the restriction of $\sI$ to $U_{klm}$.  We consider $\sI_{klm}$ as an ideal in $k[x,y,z]$, using the standard affine coordinates $x,y,z$ on $U_{klm}$.

For an ideal sheaf $\sI$ on a scheme $X$, we call the support of $\sO_X/\sI$ the {\em support of $\sI$}.
\end{nota}

\begin{rem}\label{rem:PointPairs} The $N$-action on $(\P^1)^3$ has no fixed points, and has exactly  four finite orbits (each consisting of a pair of $k$-points):  
\begin{align*}
			\alpha_1 &= ([1:0], [1:0], [1:0]), ([0:1], [0:1], [0:1]) \\
			\alpha_2 &= ([1:0], [0:1], [1:0]), ([0:1], [1:0], [0:1]) \\
			\alpha_3 &= ([1:0], [1:0], [0:1]), ([0:1], [0:1], [1:0])\\
			\alpha_4 &= ([1:0], [0:1], [0:1]), ([0:1], [1:0], [1:0])
\end{align*}
For all choices of odd integers $a, b, c$, all ideal sheaves in $\Hilb^n((\P^1)^3)$ which are fixed under the $N$-action induced by $\iota_{a,b,c}$ will thus be supported on the $\alpha_i$, and will be invariant with respect to the action by $\sigma$, in particular, all $N$-fixed ideal sheaves of finite co-length have even co-length: For all $n\ge1$, we thus have
\begin{equation}\label{eqn:OddVanishing}
\Hilb^{2n-1}((\P^1)^3)^N=\0.
\end{equation}
 Note that each $\alpha_i$ is supported on exactly two of the $U_{klm}$.
\end{rem} 

\begin{defn}\label{defn:CoordWgts} We fix odd  integers $a, b, c$, and use $\iota_{a,b,c}:N\to \SL_2^3$ to define the $N$-action on $(\P^1)^3$ and $\Hilb^n((\P^1)^3)$.  

Given $k,l,m\in\{0,1\}$, we have the standard coordinates $x,y,z$ on $U_{klm}\cong \A^3$. The $N$-action on $(\P^1)^3$ restricts to a $T_1$-action on  $U_{klm}$.  We  say that $T_1$ has {\em coordinate weights} $(s_1, s_2, s_3)$ at $0_{klm}$ if for $t\in T_1$ 
and a point $(y_1,y_2,y_3)\in \A^3$, $t\cdot (y_1,y_2,y_3)=(t^{s_1}y_1, t^{s_2}y_2, t^{s_3}y_3)$; equivalently, the action of $T_1$ on the coordinate ring $k[x,y,z]$ is given by $t\cdot (x,y,z)=(t^{s_1}x, t^{s_2}y, t^{s_3}z)$. 
 \end{defn}
 	 	 
\begin{rem} \label{remark: coordinate weights} For odd integers $a, b, c$, the coordinate weights $(s_1, s_2, s_3)$ at $0_{klm}\in U_{klm}$  are given by
\begin{itemize}
	 \item  $(-2a,-2b,-2c)$ at $0_{000}$, and $(2a, 2b, 2c)$ at $0_{111}$, corresponding to $\alpha_1$. 
	 \item  $(-2a, 2b,  -2c)$ at $0_{010}$ and $(2a,  -2b,  2c)$ at $0_{101}$, corresponding to $\alpha_2$. 
	 \item  $(-2a,  -2b,  2c)$ at $0_{001}$ and $(2a, 2b,  -2c)$ at $0_{110}$, corresponding to $\alpha_3$. 
	 \item  $(-2a, 2b, 2c)$ at $0_{011}$ and $(2a,  -2b,  -2c)$ at $0_{100}$, corresponding to $\alpha_4$. 
\end{itemize}
In short, the coordinate weights at $0_{klm}\in U_{klm}$  are  $((-1)^{k+1}2a,  (-1)^{l+1}2b, (-1)^{m+1}2c)$. In particular,  $(s_1, s_2, s_3)\equiv(2,2,2)\mod 4$, and  $s_1+s_2+s_3\equiv 2\mod 4$. 
\end{rem}

Just for reference, we state the following elementary result. 
\begin{lem}\label{lem:Nonzero} Let $U\subset \A^r_\Q$ be a non-empty Zariski open subset, let $S^+, S^-,T\subset\{1,\ldots,r\}$ be subsets with $S^+\cap S^-=\0$, and let $\{(m_i, a_i)\in \Z^2\mid i\in T\}$ be a set of pairs of integers, with each $m_i>1$ and $0\le a_i<m_i$ for each $i\in T$. Then there are infinitely many $(n_1,\ldots, n_r)\in \Z^r$ such that
\begin{enumerate}
\item $(n_1,\ldots, n_r)$ is in $U(\Q)$ and $\prod_in_i\neq0$,
\item $n_i\equiv a_i\mod m_i$ for all $i\in T$,
\item $n_i>0$ for all $i\in S^+$ and $n_i<0$ for all $i\in S^-$.
\end{enumerate}
Suppose in addition that $r>1$ and there is an $i_0\in \{1,\ldots,r\}$ such that either $i_0\not\in T$ or $i_0\in T$ and $a_{i_0}$ is a unit modulo $m_{i_0}$. Then there are infinitely many $(n_1,\ldots, n_r)\in \Z^r$ satisfying (1), (2), (3) and 
\begin{enumerate}
\item[(4)] $n_1,\ldots, n_r$ generate the unit ideal in $\Z$.
\end{enumerate}
 \end{lem}

\begin{proof} We proceed by induction on $r$. Shrinking $U$ if necessary, we may assume that 
$U$ is the principle open subscheme $f\neq0$ defined by a non-zero $f\in \Z[x_1, \ldots, x_r]$.

For $r=1$, the polynomial $f(x)$ has finitely many roots, so $f(n)\neq0$ for all but finitely many $n\in \Z$. The points (1)-(3) are then clear and (4) is empty.

For $r>1$, assume the result for $r-1$ and $r=1$, and suppose there is an $i_0$ satisfying the additional hypothesis.  Reordering the set $\{1,\ldots, r\}$ we may assume that $i_0=1$. Write $f(x_1,\ldots, x_r)=\sum_{i=0}^Mf_i(x_2,\ldots, x_r)x_1^i$. Since $f\neq 0$, $f_i$ is a non-zero polynomial for at least one $i$; let $V\subset \A^{r-1}_\Q=\Spec\Q[x_2,\ldots, x_r]$ be complement of $V(f_0,\ldots, f_M)$.  Applying the induction hypothesis,  there is an infinite subset $A$ of  $\Z^{r-1}$ such that, for all $(n_2,\ldots, n_r)\in A$,  $\prod_{i=2}^rn_i\neq0$,  $f(x_1,n_2,\ldots, n_r)$ is a non-zero polynomial, and for $2\le i\le r$, $n_i\equiv a_i\mod m_i$ for $i\in T$,   $n_i>0$ for $i\in S^+$, and $n_i<0$ for $i\in S^-$.

Take some $(n_2,\ldots, n_r)\in A$. Noting that $\prod_{i=2}^rn_i\neq0$, we let $d>0$ be the greatest common divisor of $n_2,\ldots, n_r$.  If $1\in T$,  let $m_1'>1$ be the least common multiple of $m_1, d$ and let $a_1'$ be an integer with $a_1'$ a unit modulo $m_1'$ and $a_1'\equiv a_1\mod m_1$. We then replace $(a_1, m_1)$ with $(a_1', m_1')$ in condition (2) at $i=1$, giving the new condition (2').
If $1\not\in T$, and $d>1$, we let $m_1'=d$,  let $a_1'=1$, replace  $T$ with $T':=T\cup\{1\}$ with the condition  $n_1\equiv a_1'\mod m_1'$ at $i=1$, replacing (2) with this new condition (2'). Finally, if $d=1$, and $1\not\in T$, we take (2') equal to condition (2). Applying the case $r=1$ to $U':=\A^1_\Q\setminus V(f(x_1,n_2,\ldots, n_r))$,  there are infinitely many integers $n_1\in \Z\setminus\{0\}$ such that $f(n_1, n_2,\ldots, n_r)\neq0$,  $n_1\equiv a_1'\mod m'_1$ and $n_1>0$ if $1\in S^+$, $n_1<0$ if $1\in S^-$. In particular $\prod_{i=1}^rn_i\neq0$. Then $(n_1,\ldots, n_r)$ satisfies (1) and satisfies (2') and (3) for all $i=1,\ldots, r$. In case $i\in T$, then $n_1\equiv a_1'\mod m_1'$ implies $n_1\equiv a_1\mod m_1$, so  $(n_1,\ldots, n_r)$ satisfies (2) for all $i=1,\ldots, r$. If $d>1$, then since $n_1$ is a unit modulo $d$, $n_1,\ldots, n_r$ generate the unit ideal in $\Z$, and if $d=1$, then $n_2,\ldots, n_r$ already generate the unit ideal in $\Z$. Thus, $(n_1,\ldots, n_r)$ satisfies condition (4), and the induction goes through. 

In case there is no such $i_0$, we leave the set $T$ and the condition at $i=1$ unchanged, and let $V\subset\A^{r-1}_\Q$ be the same as above. The same argument as above then shows that, for each $(n_2,\ldots, n_r)\in \Z^{r-1}$ with $(n_2,\ldots, n_r)\in V$, there are infinitely many integers $n_1$ such that $(n_1,\ldots, n_r)$ satisfy (1), (2), (3), and the induction goes through as before.  
\end{proof}

\begin{lem}\label{lem:NumericalConditions} Let $M>0$ be an integer. Then there are infinitely many triples $(a,b,c)$ of odd positive integers such that for each integer $n$ with $1\le n\le M$, 
 the action of $N$ on $\Hilb^{2n}((\P^1)^3)$ defined via $\iota_{a,b,c}$ satisfies the following conditions
 \begin{enumerate}
 \item The action of $N$ on $\Hilb^{2n}((\P^1)^3)$ has only isolated fixed points.
 \item For each ideal sheaf $\sI$ of co-length $2n$, fixed by $N$, and each $k,l,m\in \{0,1\}$,  the restriction $\sI_{klm}\subset k[x,y,z]$ is a monomial ideal, supported at $0_{klm}$. 
  \item Each point $[\sI]\in  \Hilb^{2n}((\P^1)^3)^N$ has residue field $k$.
  \item For each ideal sheaf $\sI$ of co-length $2n$, fixed by $N$, and each $k,l,m\in \{0,1\}$,    the $T_1$-representation on $\Ext^i(\sI_{klm}, \sI_{klm})$ contains no trivial subrepresentation for $i=1,2$.
 \end{enumerate}
\end{lem}

\begin{proof} By Remark~\ref{remark: coordinate weights}, $T_1$ acts on $U_{klm}$ with coordinate weights $(s_1,s_2,s_3)=(\pm 2a, \pm 2b, \pm 2c)$ at $0_{klm}$. As $0_{klm}$ is the only $T_1$-fixed point in $U_{klm}$, each $N$-fixed ideal sheaf $\sI$ has restriction $\sI_{klm}$ supported at $0_{klm}$, and has co-length $m\le n\le M$.  

Let $\Hilb^m_0(\A^3)$ be the Hilbert scheme of ideals $I\subset k[x,y,z]$ of co-length $m$ and supported at 0. Since  $\Hilb^m_0(\A^3)$ is represented by the closed subscheme $\Hilb^m_{[1:0:0:0]}(\P^3)$ of $\Hilb^m(\P^3)$,  representing ideal sheaves on $\P^3$ of co-length $m$ supported at $[1:0:0:0]$, $\Hilb^m_0(\A^3)$ is a projective $k$-scheme. Let $\mathfrak{I}$ be the universal ideal sheaf on $\Hilb^m_0(\A^3)\times \A^3$, so $\mathfrak{I}$ is supported on $\Hilb^m_0(\A^3)\times 0$.  If $V=\Spec A$ is an affine open subscheme of $\Hilb^m_0(\A^3)$, then the restriction of $\mathfrak{I}$ to $V\times \A^3$ is given by an ideal $\mathfrak{I}_V\subset  (x,y,z)\subset A[x,y,z]$, containing some power of the ideal $(x,y,z)$. Since $A$ is noetherian, 
$\mathfrak{I}_V$ is finitely generated, in particular, $\mathfrak{I}_V$ is generated by finitely many polynomials $f_i\in A[x,y,z]$, so there is an upper bound $D_V$ on the degrees  (in $x,y,z$) of the $f_i$. Covering $\Hilb^m_0(\A^3)$ by finitely many affines, and letting $m$ run from 1 to $M$,  there is an integer $D>0$ such that every ideal $I\subset k[x,y,z]$ supported at 0, and having co-length $n\le M$, has a finite set of generators $f_1,\ldots, f_r\in k[x,y,z]$ with $\deg f_i\le D$. 

Consider the set of monomials $x^iy^jz^k$ of degree $\le D$. By Lemma~\ref{lem:Nonzero}, there are infinitely many triples $(a,b,c)$ of odd positive integers such that   $n_1\cdot a+n_2\cdot b+n_3\cdot c\neq 0$ with the $n_i$ integers, not all zero, and with $|n_i|\le D$, $i=1,2,3$. Thus, for such $(a,b,c)$,  distinct monomials $x^iy^jz^k$, $x^{i'}y^{j'}z^{k'}$ of degree $\le D$, have distinct weights with respect to $T_1$, since $\pm(i-i')a+\pm(j-j')b+\pm(k-k')c\neq 0$. Therefore, the representation of $T_1$ on the $k$-vector space $k[x,y,z]_{\le D}$ of polynomials of degree $\le D$ breaks up into a sum of 1-dimensional weight spaces, with the weight space of weight $w$ spanned by the unique monomial $x^iy^jz^k$ of degree $\le D$ of weight $w=is_1+js_2+ks_3$.  Now if  $I$ is a $T_1$-stable ideal of co-length $m\le M$, supported at $0$, then $I$ has generators $f_1,\ldots, f_r$ all of degree $\le D$. Writing $f_i$ as a sum of monomials (of degree $\le D$), the fact that $t\cdot f_i$ is also in $I$ and the monomials occuring in $f_i$ have distinct weights shows that each such monomial is in $I$, hence $I$ is a monomial ideal.  As there are only finitely many monomial ideals of given co-length, this shows that $[I]$ is an isolated fixed point in $\Hilb^m_0(\A^3)$. 

Now let $L$ be an extension field of $k$, and let $I\subset L[x,y,z]$ be a $T_1$-fixed ideal of co-length $m\le M$. Repeating the above argument, we find that $I$ is also a monomial ideal, hence $I$ is the base-extension of $I\cap k[x,y,z]$. In other words, the point $[I]\in \Hilb^m_0(\A^3)$ is a $k$-point, that is $k([I])=k$.

Writing an arbitrary $N$-fixed ideal sheaf $\sI$ on $(\P^1)^3$ (uniquely) as an intersection $\sI=\cap_{klm}\sI_{0_{klm}}$ with $\sI_{0_{klm}}$ supported at $0_{klm}$, then each $\sI_{0_{klm}}$ is fixed by $T_1$, and, using the $\sigma$-symmetry, $\sI_{0_{klm}}$ has co-length at most half the co-length of $\sI$. Since $\Hilb^m_{0_{klm}}(U_{klm})=\Hilb^m_{0_{klm}}((\P^1)^3)$, the above argument shows that for odd integers $a,b,c$ with $a>D(b+c)$, $b>Dc$, $c>0$, and for $n\le M$, all $N$ fixed points $[\sI]\in \Hilb^{2n}((\P^1)^3)$ are isolated fixed points, and each $\sI_{klm}\subset k[x,y,z]$ is a monomial ideal (of co-length $m\le n\le M$) with $k([\sI_{klm}])=k$. Since  $\sI=\cap_{klm}\sI_{0_{klm}}$, this  implies that $k([\sI])=k$ as well. This proves (1), (2) and (3).

For (4), let $I\subset k[x,y,z]=:A$ be a monomial ideal. Let $T:=\G_m^3$ act on $k[x,y,z]$ by $(t_1, t_2, t_3)\cdot (x,y,z)=(t_1x, t_2y, t_3z)$, inducing a $T$-action on the Ext-groups $\Ext^i_A(I,I)$. Write the character of the $T$-action on $\Ext^i_A(I,I)$ as $\chi_i(t_1,t_2,t_3)=\sum_{k\in \Z^3}n_{k,i}t^k$, with $n_{k,i}\ge0$ integers, all but finitely many being zero, and with $t^{(k_1, k_2, k_3)}:=t_1^{k_1}t_2^{k_2}t_3^{k_3}$. Let $s=(s_1, s_2, s_3)=(\pm2a, \pm2b,\pm 2c)$ be the coordinate weights of $T_1$ at $0_{klm}$ with respect to the embedding $\iota_{a,b,c}$.   The character of the $T_1$-representation on $\Ext^i_A(\sI_{klm},\sI_{klm})$ induced by $\iota_{a,b,c}$ is thus $\sum_{w\in \Z}(\sum_{k\in \Z^3, s\cdot k=w}n_{k,i})t^w$.  For $k\in \Z^3$, let  $L_{k}=\sum_{j=1}^3k_jx_j$, and consider the finite set of linear polynomials $S_{i, (klm)}=\{L_{k}\mid n_{k,i}\neq0\}$. By Lemma~\ref{lem:Nonzero}, there are infinitely many odd positive integers $(a,b,c)$ such that $L_{k}(s_1, s_2, s_3)\neq0$ for all $L_{k}\in S_{i, (klm)}$. Adding these conditions (for $i=1,2$, $(klm)\in \{0,1\}^3$) to the list of conditions for satisfying (1)-(3), it follows from Lemma~\ref{lem:Nonzero} that there are infinitely many odd positive integers $(a,b,c)$ such that (1)-(3) hold and in addition the $T_1$-representations $\Ext^i(\sI_{klm}, \sI_{klm})$ contain no trivial summand, for all $k,l,m$ and for $i=1,2$. This completes the proof.
\end{proof}

\begin{rem}  We only consider the Hilbert scheme of co-length $2n$ ideal sheaves in Lemma~\ref{lem:NumericalConditions} because of \eqref{eqn:OddVanishing}.
\end{rem}

\subsection{Equivariant version and virtual localization}\label{subsec:EquivVirtLoc} 

\begin{prop}\label{prop:LocFormula}  Let $n\ge1$ be an integer, and let $a, b,c$ be odd positive integers. If $n$ is even, assume in addition that  $a,b,c$ are chosen to satisfy the conditions of Lemma~\ref{lem:NumericalConditions}.
Then there is an integer $m_n>0$ and an  identity in 
$H^*(BN, \sW)[1/m_ne]^0=W(k)[1/m_n]$:
\begin{equation}\label{equation: virtual localization} 
\tilde{I}_n=\tilde{I}_n^N = \sum_{[\sI]\in \text{Hilb}^n((\P^1)^3)^{N}} \frac{e_N(\Ext^2(\sI,\sI))}{e_N(\Ext^1(\sI,\sI))},
	\end{equation} 
where $e_N$ is the $N$-equivariant Euler class.  
\end{prop}

\begin{proof} The identity $H^*(BN, \sW)[1/m_ne]^0=W(k)[1/m_n]$ is \eqref{eqn:LocalizedCoh}. 

 The identity $\tilde{I}_n=\tilde{I}_n^N$ in $H^*(BN, \sW)[1/m_ne]^0=W(k)[1/m_n]$ follows from Proposition~\ref{prop:Splitting Identity} and then inverting $m_ne$, since we already have this identity in  $H^0(BN, \sW)$. 

For odd $n$, we have $\Hilb^n((\P^1)^3)^N=\0$ by \eqref{eqn:OddVanishing}, so the hypotheses on the $N$-fixed points in $\Hilb^n((\P^1)^3)$ as stated for even $n$ are also trivially satisfied for odd $n$.

The formula $\tilde{I}_n^N = \sum_{[\sI]\in \text{Hilb}^n((\P^1)^3)^{N}} \frac{e_N(\Ext^2(\sI,\sI))}{e_N(\Ext^1(\sI,\sI))}$ in $H^*(BN, \sW)[1/m_ne]^0$ for suitable $m_n>0$ follows from the virtual localization theorem \cite[Theorem 1]{LevineVLEWC}, noting that  under our assumptions on $n$ and $a,b,c$, each $N$-fixed point $[\sI]\in \Hilb^n((\P^1)^3)$ is an isolated fixed point and has residue field $k$.
\end{proof} 

\begin{rem}\label{rem:Choices}  By Lemma~\ref{lem:NumericalConditions}, given an integer $M>0$, there are infinitely many triples of odd positive integers $a,b,c$ so that  the hypotheses stated in Proposition~\ref{prop:LocFormula} for even $n$ are satisfied for all $n\le 2M$.
\end{rem}

\begin{rem}
 If we take $k=\R$, the identity \eqref{equation: virtual localization} takes place in $\Z[1/m_n]$ and $\tilde{I}_n^N\in \Z[1/m_n]$ is the image of $\tilde{I}_n\in W(\R)=\Z$ under the localization map $\Z\to \Z[1/m_n]$.
 
Since we often  have no information about the integer $m_n$, we will usually just pass to $W(k)\otimes_\Z\Q$.
\end{rem}
	
\begin{cor}\label{corollary: In vanishes for n odd} Let $n\ge 1$ be an odd integer. Then there is an integer $m_n>0$ such that  $\tilde{I}_n = 0$ in $W(k)[1/m_n]$. In particular, for $k=\R$, $\tilde{I}_n = 0$ for all odd $n$.
\end{cor}
	
\begin{proof} This follows from  Proposition~\ref{prop:LocFormula}, since for odd $n$, $\Hilb^n((\P^1)^3)^N=\0$ for all choices of odd integers $a, b, c>0$ (see \eqref{eqn:OddVanishing}). 
\end{proof}
 
\subsection{Trace calculation}\label{subsec:Trace}  In this paragraph, we fix an integer $M>0$, and  we take our embedding $\iota_{a,b,c}:N\to \SL_2^3$ with respect to odd integers $a, b, c>0$ satisfying the hypotheses of Proposition~\ref{prop:LocFormula} for every even integer $n\le 2M$; by Remark~\ref{rem:Choices} there are infinitely many such choices. Let  $n$ be an even integer with $0<n\le 2M$.

Consider an $N$-fixed ideal sheaf $\sI$ of co-length $n$.  Following \cite[Section 4]{MaulikI}, we can calculate the equivariant Euler class  $\frac{e_N(\Ext^2(\sI,\sI))}{e_N(\Ext^1(\sI,\sI))}$ appearing in  \eqref{equation: virtual localization} by using the trace function $V_I(t)$ defined below. This is carried out in Proposition~\ref{prop:NEulerClassLocal} and Corollary~\ref{cor:SumFormula} in \S\ref{sub:EquivVertex}.
	
\begin{defn}\label{defn:PartitionFunction}  Let $I\subset k[x,y,z]$ be a monomial ideal of finite co-length and let 
\[
\pi_I = \{(k_1,k_2,k_3)\in\Z_{\geq 0}^3: x^{k_1}y^{k_2}z^{k_3}\notin I\}
\]
be the 3-dimensional partition associated to $I$. Given  $(s_1, s_2, s_3)\in\Z^3$, we define the {\em partition function} of $I$  to be  
\[ 
Q_I(t) = \sum_{(k_1,k_2,k_3)\in \pi_I} t^{s_1k_1 + s_2k_2 + s_3k_3}.
\] 
We define the {\em trace function} $V_I(t)$ by
\[
V_I(t) = Q_I(t) - \frac{Q_I(t^{-1})}{t^{s_1+s_2+s_3}} + Q_I(t)Q_I(t^{-1}) \frac{(1-t^{s_1})(1-t^{s_2})(1-t^{s_3})}{t^{s_1+s_2+s_3}}.
\]
\end{defn}

\begin{prop}[See Formula (9) of \cite{MaulikI} or Proposition 3.9 of \cite{Viergever}]\label{proposition: compute trace} Let $I\subset k[x,y,z]$ be the finite co-length monomial ideal let $(s_1, s_2, s_3)$ be the coordinate weights for $T_1$ at $0$ for some linear action of $T_1$ on $\A^3$. Then $V_I(t)$ is the trace of the virtual $T_1$-representation $\Ext^1(I,I) - \Ext^2(I,I)$.
\end{prop} 
 	
Consider one of the opens $U_{klm}$ as before, and let $x,y,z$ be  our standard coordinates on $U_{klm}\cong \mathbb{A}^3$.  A $T_1$-fixed ideal sheaf $\sI$ of  co-length $n\le M$   on  $U_{klm}$ supported at $0_{klm}$ is given by a monomial ideal $I\subset k[x,y,z]$; we often abuse notation by refering to $\sI$ as a monomial ideal in $k[x,y,z]$. 

We note that the $\sigma$-action exchanges $0_{klm}$ with $0_{(1-k)(1-l)(1-m)}$. Using the symmetries in the coordinate weights of Remark \ref{remark: coordinate weights}, the trace of the image $\sigma(\sI_{klm})$ on $U_{(1-k)(1-l)(1-m)}$ is $V_I(t^{-1})$.  Thus, for an $N$-fixed  ideal sheaf $\sI$ with $\sI$  supported on $\alpha_i$, and with $\sI_{klm}$ monomial, this gives the trace  $V_{\alpha_i}(t) = V_{\sI_{klm}}(t) + V_{\sI_{klm}}(t^{-1})$ for  $\Ext^1(\sI,\sI) - \Ext^2(\sI,\sI)$, where we consider ${\sI_{klm}}$ as a monomial ideal in $k[x,y,z]$. 
	
\subsection{Equivariant vertex measure}\label{sub:EquivVertex} We retain our assumptions from \S\ref{subsec:EquivVirtLoc} on the choice of embedding $\iota_{a,b,c}$ for the given integer $M>0$
	
Recall the function $\epsilon:\Z\to \{\pm1\}$ of Notation~\ref{nota:e}. For $s:=(s_1,s_2,s_3)$ and $k:=(k_1,k_2,k_3)$ in $\Z^3$, we set $s\cdot k:=\sum_is_ik_i$. For a commutative ring $A$, we write $A_\Q$ for $A\otimes_\Z\Q$.
	
\begin{prop}\label{prop:NEulerClassLocal} Let $\sI$ be an $N$-fixed ideal sheaf on $(\P^1)^3$ with $\sI$ supported on $\alpha_i=\{0_{klm},\sigma(0_{klm})\}$,  let $s:=(s_1, s_2, s_3)$ be the coordinate weights at $0_{klm}$. We suppose that the restriction $\sI_{klm}$ is a monomial ideal in the standard coordinates $x,y,z$ on $U_{klm}$, and that the $T_1$-representations $\Ext^i(0_{klm}, 0_{klm})$ contain no trivial summand, for $i=1,2$.

We write the trace $V_{\sI_{klm}}(t)$ for the virtual $T_1$-representation $\Ext^1(\sI_{klm},\sI_{klm})- \Ext^2(\sI_{klm},\sI_{klm})$ as the finite sum  $V_{\sI_{klm}}(t) = \sum_{w \in \Z} v_w t^w$, $v_w\in \Z$. Then  
\begin{equation}\label{eqn:MainEulerClassId}
\frac{e_N(\Ext^2(\sI,\sI))}{e_N(\Ext^1(\sI,\sI))} = \prod_{w\in \Z}(\epsilon(w)\cdot w)^{-v_w}\in H^*(BN, \sW)[1/e]_\Q^0=W(k)_\Q.
\end{equation}
\end{prop}
	
\begin{proof} This is roughly the same argument as used in \cite{Viergever} for the case of $\P^3$; we repeat here the main steps. 

We write $X$ for $(\P^1)^3$ and $\Hilb^m$ for $\Hilb^m(X)$.  We apply Proposition~\ref{proposition: compute trace} to $\sI_{klm}$ and  to $\sigma(\sI_{klm})$, which by assumption are each mononomial ideals in the respective standard coordinates on $U_{klm}$ and $U_{1-k,1-l,1-m}$.

On the left-hand side of \eqref{eqn:MainEulerClassId}, the Euler class is living in degree 0 in a localization of 
$H^*(BN, \sW(\det E^{DT(n)}_\bullet))$. The expression on the right-hand side will arise from decomposing the virtual $N$-representation $\Ext^2(\sI,\sI)-\Ext^1(\sI,\sI)$ as a virtual sum of irreducible representations $\rho_w$, with Euler class of the corresponding bundle $\tilde{\sO}(w)$ living in 
$H^2(BN, \sW(\det^{-1}\tilde{\sO}(w)))$, and then taking the corresponding signed product of the Euler classes. We use the orientation for $E^{DT(n)}_\bullet$ to give an isomorphism $H^*(BN, \sW(\det E^{DT(n)}_\bullet))\cong H^*(BN, \sW)$, and we use the orientation given by a choice of standard basis for $\rho_w$ to give an isomorphism 
$H^2(BN, \sW(\det^{-1}\tilde{\sO}(w))\cong H^2(BN, \sW(\gamma))$. After applying this latter isomorphism, the fact that 
$\Ext^2(\sI,\sI)-\Ext^1(\sI,\sI)$ has virtual rank zero means that the corresponding signed product of Euler classes $e_N(\tilde{\sO}(w))$ will live in degree zero in the same localization of $H^*(BN, \sW)$ as for the left-hand side. However, to arrive at the desired identity, we need to show that the orientation isomorphism for the left-hand side is compatible with the orientation isomorphism used for the right-hand side, which amounts to showing that we can find a decomposition of the virtual $N$-representation $\Ext^2(\sI,\sI)-\Ext^1(\sI,\sI)$ as a virtual sum of irreducible representations $\rho_w$, compatible with the two given orientations. Here are the details.

We have $V_{\sI_{klm}}(t)=\sum_{w\in\Z}v_wt^w$, with only finitely many $v_w$ non-zero and with $v_0=0$.  For each $m\ge1$, our orientation on $\Hilb^m$ is an isomorphism $\tau_m:\det E^{DT(m)}_\bullet\xrightarrow{\sim} \sM^{\otimes 2}$ for some invertible sheaf $\sM$. Note that the fact that $E^{DT(m)}_\bullet$ has virtual rank 0 implies that $\dim_k\Ext^1(\sJ,\sJ)=\dim_k\Ext^2(\sJ,\sJ)$ for each  co-length $m$ ideal sheaf $\sJ$ on $X$. At a point $[\sJ]\in \Hilb^m(k)$, we choose a local generator $\lambda$ for $\sM$, and call $k$-bases $a_1,\ldots, a_r$ for $\Ext^1(\sJ,\sJ)$, and  $b_1,\ldots, b_r$ for $\Ext^2(\sJ,\sJ)$  oriented bases if $(a_1\wedge\ldots\wedge a_r)^{-1}\otimes (b_1\wedge\ldots\wedge b_r)=c^2\tau_m^{-1}(\lambda^2)$ for some $c\in k^\times$, where we use the canonical isomorphism
\[
\det E^{DT(m)}_\bullet\otimes k([\sJ])\cong \det\Ext^1(\sJ,\sJ)^{-1}\otimes\det\Ext^2(\sJ,\sJ) .
\]
 
For the case at hand, we write $\sI=\sI_{0_{klm}}\cap \sigma(\sI_{0_{klm}})$ with $\sI_{0_{klm}}$ supported at $0_{klm}$. Recall that $n=2n'$ is even,  $\sI_{0_{klm}}$ is an ideal sheaf of co-length $n'$, and we have $\dim_k\Ext^1(\sI_{0_{klm}},\sI_{0_{klm}})=\dim_k\Ext^2(\sI_{0_{klm}},\sI_{0_{klm}})$. Moreover, $\sI_{0_{klm}}$ is a monomial ideal with respect to the standard affine coordinates on $U_{klm}$. 

Our orientation on  $\Hilb^{n'}$ gives us oriented bases $a_1,\ldots, a_r$ for $\Ext^1(\sI_{0_{klm}}
, \sI_{0_{klm}})$ and $b_1,\ldots, b_r$ for $\Ext^2(\sI_{0_{klm}}, \sI_{0_{klm}})$, that is, $(a_1\wedge\ldots  \wedge a_r)^{-1}\otimes(b_1\wedge\ldots\wedge b_r)$ is (up to a square) the given orientation on $\det E_\bullet\otimes k([\sI_{0_{klm}}])$. Since the orientation on $\Hilb^{n'}$ is $N$-equivariant, this implies that $(\sigma(a_1)\wedge\ldots  \wedge \sigma(a_r))^{-1}\otimes(\sigma(b_1)\wedge\ldots\wedge \sigma(b_r))$ is (again, up to a square) the given orientation on $\det E_\bullet\otimes k([\sigma(\sI_{0_{klm}})])$, so $\sigma(a_1),\ldots, \sigma(a_r)$ and $\sigma(b_1),\ldots\,\sigma(b_r)$ give oriented bases for $\Ext^1(\sigma(\sI_{0_{klm}})
,\sigma(\sI_{0_{klm}}))$, respectively, for $\Ext^2(\sigma(\sI_{0_{klm}}),
\sigma(\sI_{0_{klm}}))$. 

It follows from the construction of the orientation on $\Hilb^n$ given in \cite[Theorem 6.3]{LevineOST} that 
$(a_1\wedge\ldots  \wedge a_r)^{-1}\otimes (\sigma(a_1)\wedge\ldots  \wedge \sigma(a_r))^{-1} \otimes(b_1\wedge\ldots\wedge b_r)\otimes(\sigma(b_1)\wedge\ldots\wedge \sigma(b_r))$ is the given orientation on 
\begin{align*}
\det E_\bullet\otimes k([\sI])=\det(\Ext^1(&\sI_{0_{klm}}, \sI_{0_{klm}})\oplus \Ext^1(\sigma(\sI_{0_{klm}}), \sigma(\sI_{0_{klm}})))\\
&\otimes \det(\Ext^2(\sI_{0_{klm}}, \sI_{0_{klm}})\oplus \Ext^2(\sigma(\sI_{0_{klm}}) \sigma(\sI_{0_{klm}})))^{-1}.
\end{align*}
We then rearrange the bases to be of the form
\[
a_1, \sigma(a_1)), \ldots, a_r, \sigma(a_r), b_1, \sigma(b_1), \ldots, b_r, \sigma(b_r). 
\]
As we have the identity of determinants
\begin{multline*}
(a_1\wedge\ldots  \wedge a_r)^{-1} \otimes (\sigma(a_1)\wedge\ldots  \wedge \sigma(a_r))^{-1} \otimes(b_1\wedge\ldots\wedge b_r)\otimes(\sigma(b_1)\wedge\ldots\wedge \sigma(b_r))\\=
[a_1\wedge\sigma(a_1)\wedge \ldots  \wedge a_r\wedge\sigma(a_r)]^{-1} \otimes[b_1\wedge\sigma(b_1)\wedge\ldots\wedge b_r\wedge\sigma(b_r)]
\end{multline*}
this new pair of bases is still oriented. 

For an integer $w$, let $O(w)$ be the 1-dimensional $T_1$-representation with $t\cdot \lambda =t^w\lambda$ for $\lambda\in O(w)$, thus, $V_{\sI_{klm}}(t)$ is the character of the virtual $T_1$-representation $\oplus_{w\in \Z}O(w)^{v_w}$, in other words, $\oplus_{w\in \Z}O(w)^{-v_w}$ is  the virtual $T_1$-representation $\Ext^2(\sI_{klm},\sI_{klm})- \Ext^1(\sI_{klm},\sI_{klm})$. If we now take the $a_i$ and $b_j$ to be basis elements for  $\Ext^1(\sI_{0_{klm}}, \sI_{0_{klm}})$ and $\Ext^2(\sI_{0_{klm}}, \sI_{0_{klm}})$ corresponding to the individual weight spaces $O(w)$ occurring in these $T_1$-representations, we can change these basis elements by multiplying by a suitable scalar to yield oriented bases. There may be cancellations in expressing the virtual representation $V_{\sI_{klm}}(t)$ as a difference of actual representations, but as the extra terms occur once in the numerator of the determinant and once in the denominator, these can be safely ignored.  Since each $w$ is even and non-zero,  each pair of bases $(a_i,\sigma(a_i))$ or $(b_j,\sigma(b_j))$, is a standard basis  for $\rho_w$, as described in Definition~\ref{defn:StandardBasis}. As the bases of pairs $a_1, \sigma(a_1),\ldots$ and $b_1, \sigma(b_1),\ldots$ are oriented bases,  we can use these bases to compute the equivariant Euler class for $\Ext^2-\Ext^1$. Also, the order of the terms $\rho_w$ does not matter, since each of these representations has rank 2, so changing the order of  the $\rho_w$ does not affect the determinant (or the Euler class).

We have thus written the virtual $N$-representation $\Ext^2(\sI, \sI)-\Ext^1(\sI, \sI)$ as the direct sum $\oplus_{w\in\Z}\rho_w^{-v_w}$, or as virtual bundle on $BN$, as $\oplus_{w\in\Z^3}\tilde{\sO}(w)^{-v_w}$, with the orientations on $\tilde{\sO}(w)$ given by the standard bases yielding the orientation on $\Ext^2(\sI, \sI)-\Ext^1(\sI, \sI)$ induced  by the  given orientation for $E^{DT(m)}_\bullet$. By Lemma~\ref{lem:NEulerClass} and the multiplicativity of the Euler class, we have
\[
\frac{e_N(\Ext^2(\sI,\sI))}{e_N(\Ext^1(\sI,\sI))} =\prod_{w\in\Z^3}(\epsilon(w)(w/2)\tilde{e})^{-v_w}
\]
where we are working in a localization of $H^*(BN,\sW)\times H^*(BN,\sW(\gamma))$ with $2\tilde{e}$ inverted. Since $E^{DT(m)}_\bullet$ has virtual rank 0 we have $\sum_wv_w=0$, so the powers of $\tilde{e}$ and of $2$ all cancel, yielding
\[
\frac{e_N(\Ext^2(\sI,\sI))}{e_N(\Ext^1(\sI,\sI))} =\prod_{w\in\Z^3}(\epsilon(w) w)^{-v_w}\in 
H^*(BN,\sW)[1/me]^0
\]
for a suitable integer $m>0$, which yields the desired formula.
\end{proof}

\begin{cor}\label{cor:SumFormula} Let $\sI\subset \sO_{(\P^1)^3}$ be an $N$-stable ideal sheaf of finite co-length $n\le 2M$, and write $\sI=\cap_{i=1}^4\sI_{\alpha_i}$ with $\sO_{(\P^1)^3}/ \sI_{\alpha_i}$ supported on $\alpha_i$. Then	 
\[
\frac{e_N(\Ext^2(\sI,\sI))}{e_N(\Ext^1(\sI,\sI))} = \prod_{i=1}^4 \frac{e_N(\Ext^2(\sI_{\alpha_i},\sI_{\alpha_i}))}{e_N(\Ext^1(\sI_{\alpha_i},\sI_{\alpha_i}))}\in (H^*(BN, \sW)[1/e]_\Q)^0=W(k)_\Q.
\]
\end{cor}

\begin{proof} This follows from the same line of reasoning as in \cite[Formula (9)]{MaulikI} or \cite[Lemma 3.4]{Viergever}.

We  write $\sI=\cap_{i=1}^4\sI_{\alpha_i}$, where $\sI_{\alpha_i}$ is an ideal sheaf with $\sI_{\alpha_i}$ supported on $\alpha_i$. By our assumption on the co-length of $\sI$ and our choice of embedding $\iota_{a,b,c}$, $\sI_{klm}$ is a monomial ideal for every $k,l,m$, and the representation $\Ext^i(\sI_{klm},\sI_{klm})$ contains no trivial summand for $i=1,2$.

Each $\alpha_i$ is a single orbit of $N$, acting on two $k$-points $0_{klm}$ and $\sigma(0_{klm})$, and we are considering the Ext-groups $\Ext^i(\sI_{\alpha_i},\sI_{\alpha_i})$ as $N$-representations, or equivalently, as $N$-equivariant vector bundles on the quotient $\alpha_i/N=\Spec(k)$. Similarly, we consider the Ext-groups $\Ext^i(\sI,\sI)$ as $N$-equivariant vector bundles over the $k$-point $[\sI]\in \Hilb^n((\P^1)^3)$. As $N$-representations we have
\[
\Ext^j(\sI,\sI)=\oplus_{i=1}^4\Ext^j(\sI_{\alpha_i},\sI_{\alpha_i}).
\]
Moreover, arguing as in the proof of Proposition~\ref{prop:NEulerClassLocal}, a concatenation of oriented bases for the virtual representation $\Ext^1(\sI_{\alpha_i},\sI_{\alpha_i})-\Ext^2(\sI_{\alpha_i},\sI_{\alpha_i})$ gives oriented bases for the virtual representation $\Ext^1(\sI,\sI)-\Ext^2(\sI,\sI)$. The Corollary then follows from the multiplicativity of the equivariant Euler class.
\end{proof}

Following \cite{MaulikI, MaulikII}, we  make the following definition. 
	
\begin{defn} \label{def:EVM}  For a monomial ideal $I\in k[x,y,z]$, let 
\[
V_I(t) = \sum_{k\in \Z^3} v_k(I)t^{s\cdot k}
\] 
be the trace function (see Definition~\ref{defn:PartitionFunction}), where  $s=(s_1,s_2,s_3)$ is to be considered as an arbitrary element of $\Z^3$. Let $\Mon_\ell$ denote the set of monomial ideals $I\subset k[x,y,z]$ of finite co-length $\ell$. The \textit{quadratic equivariant vertex measure} is the function 
\[ 
W(s,q) := 1+\sum_{\ell\ge1}q^{2\ell}\cdot\left[ {\sum_{\substack{I\subset k[x,y,z]\\ I\in \Mon_\ell}}} \prod_{k\in \Z^3} (\epsilon(s\cdot k)\cdot (s\cdot k))^{-v_k(I)}\right].
\]
\end{defn} 
 
\begin{rem}\label{rem:QuadEquivVM} Let $(s_1,s_2,s_3)$ be the coordinate weights of $T_1$ at $0_{klm}$ with respect to our given embedding $\iota_{a,b,c}$, let 
$\alpha=\{0_{klm}, \sigma(0_{klm})\}$ and let $\Mon_\ell(\alpha)$ be the set of $N$-stable   ideal sheaves $\sI$ on $(\P^1)^3$ of co-length $\ell$ and supported on $\alpha$, and such that the restriction $\sI_{klm}$ to $U_{klm}$ is a monomial ideal.\footnote{$\sI_{klm}$ is monomial for every $N$-stable $\sI$ of co-length $\ell\le 2M$, by our assumptions on $a,b,c$.}  

Taking $s=(s_1,s_2,s_3)$ to be the coordinate weights of $T_1$, we can  write $V_I(t)$ as $V_I(t)=\sum_{w\in \Z}v_w(I)t^w$, where $v_w(I)=\sum_{k, s\cdot k=w}v_k(I)$. Substituting this into the expression for $W(s,q)$ yields
\[ 
W(s,q) = 1+\sum_{\ell\ge1}q^{2\ell}\cdot\left[ {\sum_{\substack{I\subset k[x,y,z]\\ I\in \Mon_\ell}}} \prod_{w\in \Z} (\epsilon(w)\cdot w)^{-v_w(I)}\right].
\]
Then by Proposition~\ref{prop:NEulerClassLocal}, we have
\[ 
W(s,q) = 1+\sum_{1\le \ell\le 2M}q^{2\ell}\left[\sum_{\sI\in \Mon_\ell(\alpha)} \frac{e_N(\Ext^2(\sI,\sI))}{e_N(\Ext^1(\sI,\sI))}\right] \in W(k)_\Q[[q]]/(q^{2M+1}).
\] 
for $s=(s_1,s_2,s_3)$ the coordinate weights of $T_1$ at $0_{klm}$ with respect to our given embedding $\iota_{a,b,c}$.
\end{rem}

\subsection{Serre duality} 
Let $I\subset k[x,y,z]$ be a monomial ideal, and suppose we are given a $T_1$-action on $k[x,y,z]$ with coordinate weights $(s_1, s_2, s_3)$. We now define a splitting of the trace $V(t)$ associated to $I$ following \cite[Section 4.11]{MaulikI}, which we will need in the proof of Theorem \ref{theorem: main theorem}. 
	
\begin{constr} Let $Q_I(t)$ be the partition function for the monomial ideal $I\subset k[x,y,z]$, with respect to given coordinate weights $(s_1, s_2, s_3)\in \Z^3$ for $T_1$ acting on $k[x,y,z]$ (see Definition~\ref{defn:PartitionFunction}).
		
Define
\[
V_I^+(t) = Q_I(t) - Q_I(t)Q_I(t^{-1})\frac{(1-t^{s_1})(1-t^{s_2})}{t^{s_1+s_2}} 
\]
and 
\[
V_I^-(t) = -\frac{Q_I(t^{-1})}{t^{s_1+s_2+s_3}} + Q_I(t)Q_I(t^{-1})\frac{(1-t^{s_1})(1-t^{s_2})}{t^{s_1+s_2+s_3}}.
\]
This is a splitting of $V_I(t)$ in the sense that $V_I(t) = V_I^+(t) + V_I^-(t)$. 
\end{constr}

\begin{lem}\label{lemma: splitting} The splitting $V_I(t) = V_I^+(t) + V_I^-(t)$ defined above satisfies 
\[
V_I^+(t^{-1}) = -V_I^-(t)t^{s_1+s_2+s_3}.
\] 
\end{lem}
	
\begin{proof} We directly check that 
\begin{align*}
V_I^+(t^{-1}) &= Q_I(t^{-1}) - Q_I(t^{-1})Q_I(t)t^{s_1+s_2}(1-t^{-s_1})(1-t^{-s_2}) \\
		&= Q_I(t^{-1}) - Q_I(t)Q_I(t^{-1})(1-t^{s_1})(1-t^{s_2}) \\
		&= -V_I^-(t)t^{s_1+s_2+s_3}
\end{align*}
as desired. 
\end{proof}
	
We can now prove the following proposition. We retain the integer $M>0$ and the choice of embedding $\iota_{a,b,c}:N\to \SL_2^3$ from \S\ref{subsec:Trace}.
	
\begin{prop}\label{proposition: signs} Let $\sI$ be an $N$-invariant  ideal sheaf of co-length $2n$, supported on $\alpha_i=\{0_{klm}, \sigma(0_{klm})\}$, and let $(s_1, s_2,s_3)$ be the coordinate weights for $T_1$ at $0_{klm}$.\\[5pt]
1.  Write $\sI=I_{0_{klm}}\cap\sigma(I_{0_{klm}})$ with
$I_{0_{klm}}$ supported at $0_{klm}$. We suppose that $I:=I_{klm}$ is a monomial ideal,   and let $V_I(t) = \sum_{w\in \Z} v_w(I) t^w$ be the trace for $I$. Then we have  
\begin{enumerate}
\item[i)] $\prod_{w\in\Z}\epsilon(w)^{v_w(I)}=(-1)^n$.
\item[ii)] If $\sI$ has co-length $\le 2M$, then 
\[
\frac{e_N(\Ext^2(I,I))}{e_N(\Ext^1(I,I))} = (-1)^n\prod_{w\in \Z} w^{-v_w(I)}.
\]
\end{enumerate}
2. We have
\[
W(s,q) = 1+\sum_{\ell\ge1}(-q^2)^\ell\cdot\left[{\sum_{\substack{I\subset k[x,y,z]\\ I\in \Mon_\ell}}} \prod_{w\in \Z}  w^{-v_w(I)}\right]\in \Z[[q]].
\]
\end{prop}

\begin{rem} By our choice of embedding $\iota_{a,b,c}$, the conditions of 
Proposition~\ref{proposition: signs}(1) are satisfied if $\sI$ has co-length $\le 2M$.
\end{rem}
	
\begin{proof} For (1), (ii) follows from (i) and Proposition~\ref{prop:NEulerClassLocal}\eqref{eqn:MainEulerClassId}. (2) follows from (1i)
 and the definition of $W(s,q)$ (Definition~\ref{def:EVM}), after rewriting $W(s,q)$  following Remark~\ref{rem:QuadEquivVM}.  

To prove (1i), consider the splitting $V_I(t) = V_I^+(t) + V_I^-(t)$ defined above, and write
$V_I^+(t)$ as $V_I^+(t)=\sum_{w\in \Z}v_w^+(I)t^w$.
Let  $V_I^+$ be the set of $w\in \Z$ such that the coefficient $v_w^+(I)\neq0$, so we have
\[
V_I^+(t)=\sum_{w\in V_I^+}v_w^+(I)t^w.
\]
We recall that $Q_I(t)=\sum_{k\in \pi_I}t^{k\cdot s}$, and that each  $s_i$ is even and $\sum_is_i\equiv 2\mod 4$. In particular, every term $t^w$ occurring in $Q_I(t)$ has even weight $w$. Thus, $w$ is even for every  $w\in V_I^+$.
				
 Consider a non-zero term $v_w^+(I)t^w$ in  $V_I^+(t)$.  By Lemma \ref{lemma: splitting}, there is a corresponding term $-v_w^+(I)t^{-w - s_1-s_2-s_3}$ in $V_I^-(t)$.  As $w$ is even, we see that $w=-w\mod 4$. As $s_1+s_2+s_3\equiv 2\mod 4$, we have that 
\[
\epsilon(w)\epsilon(-w-s_1-s_2-s_3) = -1. 
\]
Let $k=\sum_{w\in V_I^+}v^+_w(I)$. Since 
\[
V_I(t)=V_I^+(t)+V_I^-(t)=\sum_{w\in V_I^+}(v_w^+(I)t^w-v_w^+(I)t^{-w - s_1-s_2-s_3})
\]
we see that 
\begin{equation}\label{eqn:SignIdentity}
\prod_{w\in \Z}\epsilon(w)^{v_w(I)}=
\prod_{w\in V_I^+}[\epsilon(w)\epsilon(-w-s_1-s_2-s_3)]^{v_w^+(I)}=(-1)^k.
\end{equation}

We note that 
\begin{align*}
\det(V_I(t)) &= \det(V_I^+(t))\cdot \det(V_I^-(t))\\
&= \prod_{w\in V_I^+}t^{v_w^+(I)w}t^{(-v_w^+(I))(-w - s_1-s_2-s_3)} = t^{k(s_1+s_2+s_3)}t^{2\sum_{w\in V_I^+} v_w^+(I)\cdot w}.
\end{align*}
Let $K=\sum_{w\in V_I^+}v_w^+(I)\cdot w$.  Since each $w\in V_I^+$ is even,  $K\equiv 0\mod 2$. 
		
Next, we claim that
\[
\det V_I^+(t)=\det Q_I(t),\ \det V_I^-(t)=\det(-Q_I(t^{-1})t^{-s_1-s_2-s_3}).
\]
Indeed, for $\det V_I^+(t)$, we have  $\det((1-t^{s_1})(1-t^{s_2})t^{-s_1-s_2})=1$, and $(1-t^{s_1})(1-t^{s_2})t^{-s_1-s_2})$ is the character of a virtual representation of virtual rank 0. Thus, letting $R$ be the virtual rank of  $Q_I(t)Q_I(t^{-1})$, we have
\begin{multline*}
 \det(Q_I(t)Q_I(t^{-1})(1-t^{s_1})(1-t^{s_2})t^{-s_1-s_2})\\= \det(Q_I(t)Q_I(t^{-1}))^0\cdot(\det((1-t^{s_1}) (1-t^{s_2})t^{-s_1-s_2}))^R=1,
 \end{multline*}
 and thus
 \[
 \det V_I^+(t)=\det Q_I(t)\cdot \det(Q_I(t)Q_I(t^{-1})(1-t^{s_1}(1-t^{s_2})t^{-s_1-s_2}))^{-1}=\det Q_I(t).
 \]
 The proof for $\det V_I^-(t)$ is similar. 
 
 Noting that $\pi_I$ has $n$ elements, this gives us another computation of $\det V_I(t)$,
 \begin{align*}
 \det V_I(t)=\det V_I^+(t)\cdot \det V_I^-(t)&=\det Q_I(t)\cdot  \det(-Q_I(t^{-1})t^{-s_1-s_2-s_3})\\
 &=t^{\sum_{k\in \pi_I}k\cdot s}\cdot t^{\sum_{k\in \pi_I}k\cdot s} \cdot t^{n(s_1+s_2+s_3)}\\
 &=t^{\sum_{k\in \pi_I}2k\cdot s}\cdot  t^{n(s_1+s_2+s_3)}.
 \end{align*}
Comparing with our previous computation gives
 \[
 k(s_1+s_2+s_3)+2K=n(s_1+s_2+s_3)+\sum_{k\in \pi_I}2k\cdot s
 \]
 Recalling that $s_i\equiv 0\mod 2$ for each $i$, and $s_1+s_2+s_3\equiv2\mod 4$, reducing modulo 4 gives
 \[
 2k\equiv 2n \mod 4
 \]
 or $k\equiv n\mod 2$, which together with \eqref{eqn:SignIdentity} proves (1i).
 \end{proof}

\begin{cor}\label{cor: signs} Let $\sI$ be an $N$-stable ideal sheaf of co-length $2n\le 2M$, and write $\sI=\cap_{i=1}^4\sI_{\alpha_i}$. For each $i$, write $\alpha_i=\{0_{klm}, \sigma(0_{klm})\}$, Let $s^i:=(s_1^i, s_2^i, s_3^i)$ be the coordinate weights at $0_{klm}$, let $V_i(t)=\sum_{w\in \Z}v_w^it^w$ be the trace of $\Ext^1(\sI_{0_{klm}},\sI_{0_{klm}})-\Ext^2(\sI_{0_{klm}},\sI_{0_{klm}})$ as virtual $T_1$-representation. Then
\[
\frac{e_N(\Ext^2(\sI,\sI))}{e_N(\Ext^1(\sI,\sI))}=(-1)^n\prod_{i=1}^4 
\prod_{w\in \Z} w^{-v^i_w}.
\]
\end{cor}

\begin{proof} Let $2n_i$ be the co-length of $\sI_{\alpha_i}$, so $n=\sum_{i=1}^4n_i$, in particular $n_i\le M$ for each $i$. Thus,  by our assumption on the embedding $\iota_{a,b,c}$, $\sI_{klm}$ is a monomial ideal for all $k,l,m$.  

By Corollary~\ref{cor:SumFormula}, we have
\[
\frac{e_N(\Ext^2(\sI,\sI))}{e_N(\Ext^1(\sI,\sI))}=
\prod_{i=1}^4\frac{e_N(\Ext^2(\sI_{\alpha_i},\sI_{\alpha_i}))}{e_N(\Ext^1(\sI_{\alpha_i},\sI_{\alpha_i}))}
\]
and  by Proposition~\ref{proposition: signs}, we have
\begin{align*}
\prod_{i=1}^4\frac{e_N(\Ext^2(\sI_{\alpha_i},\sI_{\alpha_i}))}{e_N(\Ext^1(\sI_{\alpha_i},\sI_{\alpha_i}))}&=\prod_{i=1}^4 
(-1)^{n_i}\prod_{w\in \Z} w^{-v^i_w}\\
&=(-1)^n\prod_{i=1}^4 
\prod_{w\in \Z} w^{-v^i_w}
\end{align*}
as desired.
\end{proof}
	
\section{Proof of Theorem \ref{theorem: main theorem}}\label{sec:MainProof}
 Recall from the Introduction the generating function 
 \[
 Z((\P^1)^3,q)=1+\sum_{n\ge 1}\tilde{I}_n q^n\in W(k)[[q]]
 \]
 for the quadratic Donaldson-Thomas  invariants for $(\P^1)^3$. We consider $\Z$ as a subring of $W(k)$ via the unit map $\Z\to W(k)$.
 
In this section, we will prove Theorem \ref{theorem: main theorem}. We start by recalling the construction of the equivariant vertex measure as used in \cite{MaulikI} and \cite{MaulikII}. 
	
\begin{constr} Let $I\subset k[x,y,z]$ be a monomial ideal. For  $s:=(s_1, s_2, s_3)\in \Z^3$, we have the  partition function $Q_I(t)$ and trace function $V_I(t)=\sum_{k\in \Z^3}(s\cdot k)^{v_k(I)}$ of Definition~\ref{defn:PartitionFunction}. Following \cite[Section 4]{MaulikI},   the {\em equivariant vertex measure} is defined as
\[
W'(s,q) = 1+\sum_{\ell\ge1}q^{\ell}\cdot \left[{\sum_{\substack{I\subset k[x,y,z]\\I\in \Mon_\ell}}} \prod_{k\in \Z^3} (s\cdot k)^{-v_k(I)}\right].
\]
If we rewrite $V_I(t)$ as $V_I(t)=\sum_{w\in \Z}w^{v_w(I)}$ for a fixed choice of $s\in \Z^3$, then this rewrites $W'(s,q)$ as 
\[
W'(s,q) = 1+\sum_{\ell\ge1}q^{\ell}\cdot \left[{\sum_{\substack{I\subset k[x,y,z]\\I\in \Mon_\ell}}} \prod_{k\in \Z^3} w^{-v_w(I)}\right].
\]

The $s_i$ will in practice be specialized to the coordinate weights of the $T_1$-action  at the appropriate origin $0_{klm}$.  
\end{constr}
	
\begin{proof}[Proof of Theorem \ref{theorem: main theorem}]
	By Corollary \ref{corollary: In vanishes for n odd}, $\tilde{I}_n$ vanishes if $n$ is odd. Therefore, assume that $n$ is even from now on.  
		
For each value of $s=(s_1, s_2, s_3)\in \Z^3$, it follows from  Proposition~\ref{proposition: signs}(2) that
\[
W(s_1,s_2,s_3,q) = W'(s_1,s_2,s_3,-q^2) 
\]
in $\Z[[q]]$.  
		
Applying \cite[Theorem 1]{MaulikII}, we have that 
\[
W'(s_1,s_2,s_3,q) = M(-q)^{-\frac{(s_1+s_2)(s_1+s_3)(s_2+s_3)}{s_1s_2s_3}},
\]
which implies that 
\begin{equation}\label{eqn:VertexIdentity}	
W(s_1,s_2,s_3,q) = M(q^2)^{-\frac{(s_1+s_2)(s_1+s_3)(s_2+s_3)}{s_1s_2s_3}}.
\end{equation}

Now take an integer $M>0$ and choose odd integers $a, b, c>0$ so that the $N$-action given by the embedding $\iota_{a,b,c}$ satisfies the hypotheses of Proposition~\ref{prop:LocFormula}, and thus for all $N$-fixed ideal sheaves $\sI$ of co-length $n\le 2M$, the restrictions $\sI_{klm}$ are all monomial ideals and the $T_1$-representations $\Ext^i(\sI_{klm},\sI_{klm})$, $i=1,2$, contain no trivial summand. By Remark~\ref{rem:Choices}, there are infinitely many such $a,b,c$.

Substituting the coordinate weights from Remark~\ref{remark: coordinate weights} for $(s_1, s_2, s_3)$,  and using \eqref{eqn:VertexIdentity}, we conclude that 
\begin{equation}\label{eqn:WIdentity}
\vbox{
$W(-2a,-2b,-2c,q)W(-2a,2b,-2c,q)W(-2a,-2b,2c,q)W(-2a,2b,2c,q)$

\ \hfill$=  M(q^2)^{-8},$
}
\end{equation}
since the sum over all pairs of opens of $(s_1+s_2)(s_1+s_3)(s_2+s_3)/(s_1s_2s_3)$  is $8$. 

Let $n$ be an even integer with $n\le 2M$. For $\sI$ an $N$-stable ideal sheaf on $(\P^1)^3)$ of   co-length $n$, write $\sI=\cap_{i=1}^4\sI_{\alpha_i}$ with $\sI_{\alpha_i}$ supported on $\alpha_i$.  It follows from Proposition~\ref{prop:Splitting Identity}, the localization formula \eqref{equation: virtual localization} and Corollary~\ref{cor:SumFormula}  that in $W(k)_\Q$, we have
\begin{align*}
\tilde{I}_n&=\sum_{[\sI]\in \text{Hilb}^n((\P^1)^3)^{N}} \frac{e_N(\Ext^2(\sI,\sI))}{e_N(\Ext^1(\sI,\sI))}\\
&=\sum_{[\sI]\in \text{Hilb}^n((\P^1)^3)^{N}}  \prod_{i=1}^4 \frac{e_N(\Ext^2(\sI_{\alpha_i},\sI_{\alpha_i}))}{e_N(\Ext^1(\sI_{\alpha_i},\sI_{\alpha_i}))}.
\end{align*}
By Remark~\ref{remark: coordinate weights} and Remark~\ref{rem:QuadEquivVM}, we thus have the identity
\begin{multline*}
Z((\P^1)^3,q)= \\
W(-2a,-2b,-2c,q)W(-2a,2b,-2c,q)W(-2a,-2b,2c,q)W(-2a,2b,2c,q).
\end{multline*}
in $W(k)_\Q[[q]]/(q^{2M+1})$.
But then \eqref{eqn:WIdentity} shows that $Z((\P^1)^3,q)=M(q^2)^{-8}$ in $W(k)_\Q[[q]]/(q^{2M+1})$, and  as $M>0$ is arbitrary, it follows that $Z((\P^1)^3,q)=M(q^2)^{-8}$ in $W(k)_\Q[[q]]$, proving 
Theorem \ref{theorem: main theorem}.
\end{proof} 
	
\begin{rem} Using a SAGE-implementation for the  procedure described in the proof of Proposition~\ref{prop:NEulerClassLocal}, similar to the one used in \cite{Viergever}, one can compute directly that for $(\P^1)^3$, the first quadratic DT invariants are $-8, 12, 48$ and $-98$, in agreement with the statement of Theorem \ref{theorem: main theorem}. \end{rem}

\section{Equivariant DT invariants of $N_3$-oriented toric varieties}\label{sec:BlowUp} 

In this section, we discuss some extensions of our results for the DT invariants of $(\P^1)^3$ to the case of what we call an {\em  $N_3$-oriented smooth proper toric threefold $X_\Sigma$}; see Definition~\ref{defn:N3Orient} for details. In \S\ref{sec:Construction}, we will give a criterion for $X_\Sigma$ to be $N_3$-oriented, entirely in terms of the fan $\Sigma$. 

Throughout this section, we let $k$ be a perfect field which is not of characteristic $2$. 

\subsection{Recollections on toric varieties}\label{subsec:ToricRecoll}
We give a brief overview on some basic facts about toric varieties, to set the stage and fix notation; we refer the reader to \cite{CoxTV} for details.\footnote{The treatment in \cite{CoxTV} is for the base-field $k=\C$, but we have checked that the proofs of the results we recall here extend easily to arbitrary $k$. In the interest of simplifying the notation for our particular setting, some of our notations and conventions differ from those used in {\it loc. cit.}}  We fix a base-field $k$ and let $\GpS/k$ denote the category  of
group schemes over $k$.

Fix an integer $n>0$. Let $\G_m=\Spec k[x,x^{-1}]$ with its usual structure of a group scheme over $k$. Let $\sN$ be the co-character lattice for $\G_m^n$, that is,  the group of homomorphisms $f:\G_m\to \G_m^n$, and let $\sM$ be the dual character lattice of homomorphisms $\chi:\G_m^n\to \G_m$. Identifying $\Hom_{\GpS/k}(\G_m, \G_m)$ with $\Z$ by $n\mapsto f_n$, $f_n(t)=t^n$, the duality of $\sN$ and $\sM$ is given by $\<\chi,f\>=\chi\circ f$.   Let $x_i:\G_m^n\to \G_m$ be the projection on the $i$th factor, which we can also view as a function on $\G_m^n$ via the open immersion $\G_m=\Spec k[x,x^{-1}]\subset \A^1_k=\Spec k[x]$. This gives us the isomorphism  $k[\sM]\cong \Gamma(\G_m^n, \sO_{\G_m^n})=k[x_1^{\pm1},\ldots, x_n^{\pm1}]$, and identifies $\sM$ with the subgroup of
$k[x_1^{\pm1},\ldots, x_n^{\pm1}]^\times$ generated by $x_1,\ldots, x_n$. This gives us an isomorphism $\sM\cong\Z^n$; the dual isomorphism $\sN\cong \Z^n$ gives $\sN$ the basis $i_1,\ldots, i_n$, with $i_j:\G_m\to \G_m^n$ the co-character $i_j$ with $x_i\circ i_j(t)=1$ for $i\neq j$, $x_j\circ i_j(t)=t$.

We will usually identify $\sN$ with $\Z^n$ by this basis and write the group law for $\sN$ additively; we continue to write the group law for $\sM$, as a group of characters, multiplicatively.

\begin{defn} A   {\em  cone in $\sN$} is a submonoid $\kappa$ of $\sN$ of the form
	\[
	\kappa=\sN\cap C_\kappa,
	\]
	where $C_\kappa\subset \sN_\R\cong \R^n$ is a convex polyhedral cone generated by finitely many elements $v_1,\ldots, v_r$ of $\sN$:
	\[
	C_\kappa=\{\sum_{i=1}^r\lambda_iv_i\mid \lambda_i\in \R, \lambda_i\ge0\}\subset \sN_\R. 
	\]
	
	Call a cone $\kappa=\sN\cap C_\kappa$ a {\em simplicial cone} if $C_\kappa$ is the convex cone generated by $v_1,\ldots, v_r\in \sN$, such that the $v_i$ are linearly independent over $\R$, and there is a non-zero $\R$-linear function $h:\sN_\R\to \R$ such that $h(v_i)>0$ for $i=1,\ldots, r$. Such a simplicial cone $\kappa$ is said to have dimension $r$. If $\kappa$ is a simplicial cone defined as above via  $\R$-linearly independent elements $v_1,\ldots, v_r$ of $\sN$, a {\em face} $\kappa'$ of  $\kappa$ is a simplicial subcone $\kappa'=C\cap \kappa$,  where $C\subset C_\kappa\subset \sN_\R$ is the convex cone generated by some subset of $\{v_1,\ldots, v_r\}$.
	
	For a simplicial cone $\kappa=\sN\cap C_\kappa$ with $C_\kappa$, generated by $v_1,\ldots, v_r\in \sN$, we may rescale each $v_i$ so that $v_i$ is a generator of the monoid $\R^+\cdot v_i\cap \sN$. We then say that $v_1,\ldots, v_r$ are {\em minimal generators} of $C_\kappa$. We say that $\kappa$ is {\em regular} if the set of minimal generators of $C_\kappa$ are part of a $\Z$-basis of $\sN$; in this case, $\kappa$ is generated as a monoid by the minimal generators $v_1,\ldots, v_r$ of $C_\kappa$.  
\end{defn}

Let $\Sigma$ be a simplicial fan in $\sN$, that is, $\Sigma$  is non-empty finite set of simplicial cones in $\sN$, satisfying
\begin{itemize}
	\item For $\kappa\in  \Sigma$ and $\kappa'\subset \kappa$ a face of $\kappa$, we have 
	$\kappa'\in \Sigma$.
	\item For $\kappa, \kappa'\in  \Sigma$, the intersection  $\kappa\cap \kappa'$ in $\sN$ is a common face of $\kappa$ and $\kappa'$.
\end{itemize}
We order $\Sigma$ by inclusion, making $\Sigma$ a partially ordered set, which we consider as a category with a unique map $\kappa'\to \kappa$ exactly when $\kappa'\subset \kappa$. We let $\Sigma(m)\subset \Sigma$ denote the set of dimension $m$ cones in $\Sigma$, and we will assume that $\Sigma(n)$ is the set of maximal elements in $\Sigma$.  If each $\kappa\in \Sigma$ is regular, we call $\Sigma$ a {\em regular} fan.

For each simplicial cone $\kappa$ in $\Sigma$,  we let $\kappa^\vee\subset \sM$ be the dual cone
\[
\kappa^\vee=\{\chi\in\sM\mid \<\chi,f\>\ge0\ \forall f\in \kappa\}.
\]
We let $R_\kappa\subset k[\sM]$ be the $k$-subalgebra generated by $\kappa^\vee$ and let $U_\kappa=\Spec R_\kappa$.  By \cite[Theorem 1.3.12]{CoxTV}, $U_\kappa$ is smooth over $k$ if and only if $\kappa$ is regular.

If $\kappa'\subset \kappa$ then there is a $c\in \kappa^\vee$ such that $(\kappa')^\vee=\kappa^\vee[c^{-1}]$ (see \cite[Proposition 1.3.16]{CoxTV}). Thus 
$R_{\kappa'}$ is the localization $R_{\kappa}[1/c]$ of $R_{\kappa}$ and the inclusion  $R_\kappa\subset R_{\kappa'}$ induces an open immersion $j_{\kappa'\subset \kappa}:U_{\kappa'}\hookrightarrow U_\kappa$. We thus have the functor
\[
U_{(-)}: \Sigma\to \Sm/k
\]
sending $\kappa$ to $U_\kappa$ and $\kappa'\subset\kappa$ to $j_{\kappa'\subset \kappa}$. 

\begin{defn} The toric variety over $k$ associated to a simplicial fan $\Sigma\subset \sN\cong \Z^n$  is the colimit
	\[
	X_\Sigma:=\colim_{\kappa\in \Sigma}U_\kappa.
	\]
\end{defn}
If we need to specify the base-field $k$, we write $X_\Sigma/k$, but usually the context will specify the base-field.

In fact, $X_\Sigma$ exists as a separated, normal finite type $k$-scheme and the canonical map $j_\kappa:U_\kappa\to X_\Sigma$ is an open immersion for each $\kappa\in  \Sigma$ (see \cite[Theorem 3.1.5 and proof]{CoxTV}). We consider $j_\kappa$ as the inclusion of $U_\kappa$ as an open subscheme of $X_\Sigma$ and often write $j_{\kappa'\subset \kappa}$ simply as an inclusion $U_{\kappa'}\subset U_{\kappa}$. Since each $\kappa\in  \Sigma$ contains $0$ as a face, each $U_{\kappa}$ contains $U_0=\G_m^n$ as a dense open subscheme; in particular, $X_\Sigma$ is integral, of pure dimension $n$ over $k$.  If $\Sigma$ is regular, then for each $\kappa\in \Sigma(d)$, $U_\kappa$ is isomorphic to $\A^d\times\G_m^{n-d}$. In addition,  $X_\Sigma$ is smooth over $k$ if and only if $\Sigma$ is regular \cite[Theorem 3.1.19]{CoxTV}. 

The product $\G_m^n\times \G_m^n\to \G_m^n$ defining $\G_m^n$ as a group scheme over $k$ is given by the co-product
\[
\nabla:k[\sM]\to k[\sM]\otimes_kk[\sM],\ \nabla(t)=t\otimes t \text{ for each }t\in \sM.
\]
For each $\kappa\in  \Sigma$, $\nabla$ restricts to a coproduct
\[
\nabla_\kappa:k[\kappa^\vee] \to k[\sM]\otimes_kk[\kappa^\vee] 
\]
defining a $\G_m^n$-action on $U_\kappa$, compatible with the restriction maps $j_{\kappa'\subset \kappa}$, hence extending to a $\G_m^n$-action on $X_\Sigma$:
\[
m_\Sigma:\G_m^n\times X_\Sigma\to X_\Sigma,
\]
with each $U_\kappa$ stable under this $\G_m^n$-action. The $\G_m^n$-action on $U_0=\G_m^n$ is just the multiplication on $\G_m^n$.

$X_\Sigma$ has the open cover $\{U_\kappa\mid \kappa\in \Sigma(n)\}$, so each $\G_m^n$-fixed point $x\in X_\Sigma^{\G_m^n}$ is a closed $\G_m^n$-orbit in some $U_\kappa$. In fact, for each $\kappa\in \Sigma$, $U_\kappa$ contains a unique closed orbit $0_\kappa$ (see \cite[Theorem 3.2.6]{CoxTV}), so we have
\[
X_\Sigma^{\G_m^n}=\{0_\kappa\mid \kappa\in\Sigma(n)\},
\]
and we have a bijection $\Sigma(n) \leftrightarrow X_\Sigma^{\G_m^n}$ associating $\kappa$ to $0_\kappa$. 

For each $1$-dimensional cone $\kappa^*\in \Sigma(1)$, $U_{\kappa^*}$ is isomorphic to $\A^1_k\times \G_m^{n-1}$. Indeed, $\kappa^*$ has a single generator $v_1^{\kappa^*}$ and, as  $v_1^{\kappa^*}$ is not divisible in $\sN$, $\sN$ contains a subgroup $\sN'\cong \Z^{n-1}$ with $\sN=\Z\cdot v_1^{\kappa^*}\oplus \sN'$.  Giving $\sN'$ a basis $v_2',\ldots, v_n'$, we have the dual basis $x_1^{\kappa^*}, x_2', \ldots, x_n'$ for $\sM$, and then $(\kappa^*)^\vee$ is the submonoid generated by  $x_1^{\kappa^*}, x_2^{\prime\pm1},\ldots, x_n^{\prime\pm1}$ and 
$U_{\kappa^*}=\Spec k[x_1^{\kappa^*}, x_2^{\prime\pm1},\ldots, x_n^{\prime\pm1}]\cong 
\A^1_k\times \G_m^{n-1}$. Thus $U_{\kappa^*}$ contains the unique closed orbit $0_{\kappa^*}=\{0\}\times \G_m^{n-1}$, a codimension one closed subscheme. 

The closure $D_{\kappa^*}=\overline{0_{\kappa^*}}\subset X_\Sigma$ is then a $\G_m^n$-stable codimension one integral closed subscheme of $X_\Sigma$, and each such is of the form  $D_{\kappa^*}$ for a unique $\kappa^*\in \Sigma(1)$. Moreover, for $\kappa\in \Sigma$ we have
\begin{equation}\label{eqn:ComplementDiv}
	X_\Sigma\setminus U_\kappa =\cup_{\substack{\kappa^*\in \Sigma(1)\\\kappa^*\not\subset \kappa}} D_{\kappa^*}.
\end{equation}
The assertions  in this paragraph all follow from the ``Orbit-Cone correspondence'' \cite[Theorem 3.2.6]{CoxTV}.

Suppose $\Sigma$ is regular. Then each $\kappa\in \Sigma(n)$ is the free abelian monoid on generators $v_1^\kappa,\ldots, v_n^\kappa$,  and dually,  $\kappa^\vee$ is the free abelian monoid on generators $x_1^\kappa,\ldots, x_n^\kappa$ that form a basis for $\sM$ dual to $v_1^\kappa,\ldots, v_n^\kappa$:
\[
\<x_i^\kappa, v_j^\kappa\>=\delta_{ij},
\]
where $\delta_{ij}$ is the Kronecker delta. 
Thus
\[
U_\kappa=\Spec k[x_1^\kappa,\ldots, x_n^\kappa]\cong  \A^n_k,
\]
with $0_\kappa$ the origin in $\A^n_k$. We call $x_1^\kappa,\ldots, x_n^\kappa$ {\em standard affine coordinates for $U_\kappa$}; note that $\{x_1^\kappa,\ldots, x_n^\kappa\}\subset \sM$ is uniquely determined by $\kappa$.  

\subsection{$N_3$-oriented toric threefolds}\label{subsec:N3OrientedThreefolds}

\begin{nota} We let $N_3\subset N^3\subset \SL_2^3$ be the subgroup scheme of $N^3$  generated by $T_1^3\subset N^3$ and $\sigma_\delta:=(\sigma, \sigma,\sigma)$. We write $T$ for $T_1^3\subset N_3$. 
	
	Identifying $T_1$ with $\G_m$ by sending $t\in \G_m$ to $\begin{pmatrix}t&0\\0&t^{-1}\end{pmatrix}\in T_1$ identifies $T$ with $\G_m^3$; via this identification we write $(t_1, t_2, t_3)\in T$ for $t_1, t_2, t_3\in \G_m$.
	
	From now on, when we say that $X$ is a smooth proper toric threefold over $k$, we will always mean that $X=X_\Sigma$ for $\Sigma$ a regular complete simplicial fan in $\sN\cong \Z^3$. 
\end{nota}

\begin{defn}\label{defn:N3Orient} Let $X=X_\Sigma$ be a smooth proper toric threefold  over $k$. Suppose  there is an action of $N_3$ on $X$, defined by a homomorphism $\rho:N_3\to \Aut_k(X)$.  We say that $X$ is {\em $N_3$-oriented} if the following conditions hold.
	\begin{enumerate}
		\item\label{defn:N3Orient:1} Let $\G_m^3\subset \Aut_k(X)$ be the inclusion given by the action of $\G_m^3$ on $X$ as a toric variety. Then $\rho$ restricts to a surjective homomorphism $\rho:T\to \G_m^3$. 
		\item \label{defn:N3Orient:1bis}  $X$ has a spin structure $\tau:K_X\xrightarrow{\sim} L^{\otimes 2}$, and the invertible sheaf  $L$ is given an $N_3$-linearization such that $\tau$ is $N_3$-equivariant with respect to the canonical $N_3$-linearization on $K_X$ induced by the action $\rho$.
		\item\label{defn:N3Orient:2} $X$ admits an $N_3$-linearized very ample invertible sheaf $\sO_X(1)$.
		\item\label{defn:N3Orient:3} $\rho(-1,-1,-1)=\id_X$.
		\item\label{defn:N3Orient:4} Let $\{0_\kappa\}_{\kappa\in \Sigma(3)}$ be the set of $\G_m^3$-fixed points on $X$, and let $U_\kappa\subset X$ be the corresponding $\G_m^3$-stable affine space  containing $0_\kappa$.  Let $x_1^\kappa,x_2^\kappa,x_3^\kappa$ be standard affine coordinates on $U_\kappa$, and write  
		\begin{equation}\label{eqn:CoordWeightMatrix}
			\rho(t_1, t_2, t_3)^*(x_j^\kappa)=t_1^{s_{1j}^\kappa}t_2^{s_{2j}^\kappa}t_3^{s_{3j}^\kappa}\cdot x_j^\kappa
		\end{equation}
		for $(t_1, t_2, t_3)\in T$, with the  $s_{ij}^\kappa$ integers. Then $\sum_is_{ij}$ is even for each $j=1,2,3$, $\sum_js_{ij}$ is even for each $i=1,2,3$,  and $\sum_{i,j}s_{ij}\equiv 2\mod 4$.
	\end{enumerate}
	An {\em $N_3$-orientation on $X$} is the data of an $N_3$-action on $X$ together with an $N_3$-equivariant spin structure as in \eqref{defn:N3Orient:1bis} and an $N_3$-linearized very ample invertible sheaf $\sO_X(1)$ such that the above conditions hold.
	
	We call the matrix of integers $(s_{ij}^\kappa)_{1\le i,j\le 3}$ satisfying \eqref{eqn:CoordWeightMatrix} a {\em coordinate weight matrix  for $T$ at $0_\kappa$}. 
\end{defn}
Note that two different matrices are both coordinate weights for $T$ at $0_\kappa$ if and only if they differ by a permutation of the columns.

\begin{lem}\label{lem:SigmaInvolution} Let $\Sigma$ be a regular complete simplicial fan in $\sN\cong \Z^3$,  and suppose the $k$-scheme $X:=X_\Sigma$ is given an $N_3$-orientation, with $N_3$-action given by $\rho:N_3\to \Aut_k(X)$.  Then the fan $\Sigma$ is stable under the automorphism $-\id_\sN:\sN\to \sN$, and for each  $\kappa\in\Sigma(3)$, we have $
	\rho(\sigma_\delta)(U_\kappa)=U_{-\id_{\sN}(\kappa)}$.
\end{lem} 

\begin{proof}  Since the map $\rho:T\to \G_m^3$ is surjective, the inclusion
	\[
	X^T\subset X^{\G_m^3}=\{0_\kappa\}_{\kappa\in \Sigma(3)}
	\]
	is a bijection. For $t\in T$, we have $t\cdot \sigma_\delta=
	\sigma_\delta\cdot t^{-1}$ so
	\[
	\rho(t)(\rho(\sigma_\delta)(0_\kappa))=
	\rho(\sigma_\delta)(\rho(t^{-1})(0_\kappa))=
	\rho(\sigma_\delta)(0_\kappa),
	\]
	hence $\rho(\sigma_\delta)(0_\kappa)$ is in $X^T=X^{\G_m^3}$. Thus, there is a unique $\kappa'\in \Sigma(3)$ with $\rho(\sigma_\delta)(0_\kappa)=0_{\kappa'}$. Since $\rho(\sigma_\delta^2)=\rho(-1,-1,-1)=\id_X$, we have $\rho(\sigma_\delta)^2=\id_X$, and thus 
	$\rho(\sigma_\delta)(0_{\kappa'})=0_\kappa$.  Similarly, letting  $U=\rho(\sigma_\delta)(U_\kappa)$, we have $0_{\kappa'}\in U$, and  $\rho(\sigma_\delta)(U)=U_\kappa$. 
	
	For every affine $\G_m^3$-stable open subscheme $V$ of $X$, $X\setminus V$ is a union of integral codimension one $\G_m^3$-stable closed subschemes of $X$. Also, the codimension one $\G_m^3$-stable integral closed subschemes of $U_\kappa$ are exactly the three coordinate hyperplanes $x_j^\kappa=0$, where $x_1^\kappa, x_2^\kappa, x_3^\kappa$ are standard affine coodinates on $U_\kappa$. In particular, for $\kappa^*\in \Sigma(1)$,  $D_{\kappa^*}\cap U_\kappa\neq\0$ if and only if  $0_\kappa\in D_{\kappa^*}$, which is equivalent to $\kappa^*\subset \kappa$. By \eqref{eqn:ComplementDiv}, we thus have
	\[
	U_\kappa=X\setminus\cup_{\substack{\kappa^*\in  \Sigma(1)\\0_\kappa\not\in D_{\kappa^*}}} D_{\kappa^*}.
	\]
	Similarly
	\[
	U_{\kappa'}=X\setminus\cup_{\substack{\kappa^*\in \Sigma(1)\\0_{\kappa'}\not\in  D_{\kappa^*}}} D_{\kappa^*}.
	\]
	
	Since  $\rho(t)\circ \rho(\sigma_\delta)= \rho(\sigma_\delta)\circ \rho(t)^{-1}$ for all $t\in T$, and since $\rho:T\to \G_m^3$ is surjective, we see that $U$ is a $\G_m^3$-stable, affine open subscheme of $X$. Similarly $\rho(\sigma_\delta)(D_{\kappa^*})$ is a $\G_m^3$-stable integral codimension one closed subscheme of $X$, so $\rho(\sigma_\delta)(D_{\kappa^*})=D_{\kappa^{**}}$ for some $\kappa^{**}\in \Sigma(1)$, $\rho(\sigma_\delta)(D_{\kappa^{**}})=D_{\kappa^*}$, and  $0_\kappa\not\in D_{\kappa^*}$ if and only if  $\rho(\sigma_\delta)(0_\kappa)\not\in D_{\kappa^{**}}$ . Since $0_{\kappa'}=\rho(\sigma_\delta)(0_\kappa)$, we have 
	\begin{align*}
		U&=\rho(\sigma_\delta)(U_\kappa)\\
		&=X\setminus \cup_{\substack{\kappa^*\in \Sigma(1)\\ 0_{\kappa}\not\in D_{\kappa^*}}}\rho(\sigma_\delta)(D_{\kappa^*})\\
		&=X\setminus \cup_{\substack{\kappa^{**}\in \Sigma(1)\\ 0_{\kappa'}\not\in D_{\kappa^{**}}}}D_{\kappa^*}\\
		&=U_{\kappa'}.
	\end{align*}
	
	It remains to show that $\kappa'=-\id_\sN(\kappa)$.  For this, consider some $\lambda\in \Sigma(3)$. If we   decompose $k[U_{\lambda}]$ as the direct sum of weight-spaces for the action of $\G_m^3$, we have 
	\[
	k[U_{\lambda}]=\oplus_{\chi\in \lambda^\vee} k\cdot x_\chi,
	\]
	with $x_\chi(t\cdot u)= \chi(t)\cdot x_\chi(u)$ for each $u\in U_{\lambda}$, $t\in \G_m^3$. Thus, if $k[U_{\lambda}]=k[y_1,y_2, y_3]$, with  $y_j\in  k\cdot \chi_j$ for some $\chi_j\in \sM$, $i=1,2,3$, then after reordering, we have $\chi_j=x^\lambda_j$ and $y_j=a_j\cdot x^\lambda_j$ for some $a_j\in k^\times$.   
	
	Since $\sigma_\delta\cdot t=t^{-1}\cdot\sigma_\delta$ for $t\in T$, we have
	\begin{align*}
		\rho(t)^*(\rho(\sigma_\delta)^*(x_j^\kappa))&=
		\rho(\sigma_\delta)^*(\rho(t^{-1})^*(x_j^\kappa))\\
		&=\rho(\sigma_\delta)^*((x_j^\kappa)(\rho(t)^{-1})\cdot x_j^\kappa)\\
		&=(x_j^\kappa)^{-1}(\rho(t))\cdot\rho(\sigma_\delta)^*(x_j^\kappa)
	\end{align*}
	for all $t\in T$. 
	Since  $\rho:T\to \G_m^3$ is surjective, this implies
	\[
	(\rho(\sigma_\delta)^*(x_j^\kappa))(t\cdot u)=(x_j^\kappa)^{-1}(t)\cdot(\rho(\sigma_\delta)^*(x_j^\kappa))(u)
	\]
	for $u\in k[U_{\kappa'}]$, $t\in \G_m^3$, in other words, $\rho(\sigma_\delta)^*(x_j^\kappa)$ is in the weight-space $k\cdot (x_j^\kappa)^{-1}$.   Since $k[U_{\kappa'}]=k[U]=k[\rho(\sigma_\delta)^*(x_1^\kappa),\rho(\sigma_\delta)^*(x_2^\kappa),\rho(\sigma_\delta)^*(x_3^\kappa)]$, it follows that, after reordering, we have
	\begin{equation}\label{eqn:Inverse}
		x_j^{\kappa'}=(x_j^\kappa)^{-1},\ j=1,2,3.
	\end{equation}
	
	Let $v_1^\kappa, v_2^\kappa, v_3^\kappa$ be the generators of $\kappa$ dual to 
	$x_1^\kappa,x_2^\kappa,x_3^\kappa$, and  $v_1^{\kappa'}, v_2^{\kappa'}, v_3^{\kappa'}$ be the generators of $\kappa'$ dual to 
	$x_1^{\kappa'},x_2^{\kappa'},x_3^{\kappa'}$.  By \eqref{eqn:Inverse}, we have $v_j^{\kappa'}=-v_j^\kappa$, $j=1,2,3$, hence  $\kappa'=-\id_{\sN}(\kappa)$.
\end{proof}

In summary, writing $-\kappa$ for $-\id_\sN(\kappa)$, we have shown the following.
\begin{prop} Take $\kappa\in \Sigma(3)$\\[5pt]
	1. The action of the involution $\rho(\sigma_\delta)$ on $X$  satisfies $\rho(\sigma_\delta)(U_\kappa) = U_{-\kappa}$ and 
	$\rho(\sigma_\delta)(0_\kappa) = 0_{-\kappa}$. \\[2pt]
	2.  If $x_1^\kappa,x_2^\kappa, x_3^\kappa$ are standard coordinates for $U_\kappa$, then
	$(x_1^\kappa)^{-1},(x_2^\kappa)^{-1}, (x_3^\kappa)^{-1}$ are standard coordinates for $U_{-\kappa}$.
\end{prop}

\begin{exa} Supppose we have a  toric threefold $X=X_\Sigma$ over $k$ for $\Sigma$ a complete regular simplicial fan. Suppose further that $X$ is given an $N_3$-orientation, where the $N_3$ action is given by $\rho_X:N_3\to \Aut_kX$. For $\kappa\in \Sigma(3)$, let $\pi: X^{(\kappa)}\to X$ be the blow-up $\Bl_{\alpha_\kappa}X$, with exceptional divisor $E_\kappa=E_{0_\kappa}\amalg E_{0_{-\kappa}}$.  Then $X^{(\kappa)}$ has a canonical structure of a smooth proper toric threefold with fan $\Sigma^{(\kappa)}$ the usual subdivision of $\Sigma$, and the $N_3$-action lifts canonically to an $N_3$-action $\rho_{X^{(\kappa)}}:N_3\to \Aut_kX^{(\kappa)}$  on $X^{(\kappa)}$.  Mapping $\Pic(X)$ to  $\Pic(X^{(\kappa)})$ by $\pi^*$ gives the isomorphism
	\[
	\Pic(X^{(\kappa)})\cong \Pic(X)\oplus \Z\cdot [\sO_{X^{(\kappa)}}(E_{0_\kappa})]
	\oplus \Z\cdot [\sO_{X^{(\kappa)}}(E_{0_{-\kappa}})], 
	\]
	and we have a canonical natural isomorphism 
	\[
	K_{X^{(\kappa)}}\cong \pi^*K_X\otimes \sO_{X^{(\kappa)}}(2E_\kappa).
	\]
	As the invertible sheaf $\sO_{X^{(\kappa)}}(E_\kappa)$ arises from the graded $\Sym^*\sI_{\alpha_\kappa}$-module  $\Sym^*\sI_{\alpha_\kappa}[-1]$, 
	$\sO_{X^{(\kappa)}}(E_\kappa)$  has a canonical $N_3$-linearization induced from the $N_3$-action on $X$. 
	Thus, the given $N_3$-equivariant spin structure $\tau:K_X\xrightarrow{\sim}L^{\otimes 2}$, gives  $X^{(\kappa)}$   the induced $N_3$-equivariant spin structure $\tau^{(\kappa)}:K_{X^{(\kappa)}}\xrightarrow{\sim}(\pi^*L\otimes\sO_{X^{(\kappa)}}(E_\kappa))^{\otimes 2}$.  
	
	Similarly, if $\sO_X(1)$ is an $N_3$-linearlized very ample invertible sheaf on $X$, then there is a $d_0>0$ such that for all $d\ge d_0$, the sheaf $\pi^*\sO_X(d)\otimes\sO_{X^{(\kappa)}}(-E_\kappa)$ is very ample on $X^{(\kappa)}$.
	
	Lying over $0_\kappa$, one has three $\G_m^3$-fixed points $0_{\kappa,i}$, $i=1,2,3$, in $X^{(\kappa)}$, with standard coordinates $x_1^\kappa, x_2^\kappa/x_1^\kappa, x_3^\kappa/x_1^\kappa$ for $i=1$, $x_1^\kappa/x_2^\kappa, x_2^\kappa, x_3^\kappa/x_2^\kappa$ for $i=2$ and $x_1^\kappa/x_3^\kappa, x_2^\kappa/x_3^\kappa, x_3^\kappa$ for $i=3$. One can then easily check that the orientation conditions of Definition~\ref{defn:N3Orient}\eqref{defn:N3Orient:4} pass from $X$ to $X^{(\kappa)}$. Finally, since $\rho_X(-1,-1,-1)=\id_X$, we have $\rho_{X^{(\kappa)}}(-1,-1,-1)=\id_{X^{(\kappa)}}$.
	
	In short, an $N_3$-orientation on $X$ gives us an $N_3$-orientation on $X^{(\kappa)}$, canonically defined once we choose a $d\ge d_0$ as above. One can repeat this construction to $X^{(\kappa)}$, giving us the notion of an {\em $N_3$-oriented iterated blow-up} of an $N_3$-oriented smooth proper toric threefold $X$:
	\[
	Y=Y_r\to Y_{r-1}\to\ldots\to Y_1\to Y_0=X
	\]
	with an $N_3$-orientation on $Y$ induced from that on $X$.
\end{exa}

\begin{exa}\label{exa:P13Blowup} Give $(\P^1)^3$ its usual structure of a toric variety over $k$, using the fan $\Sigma\subset \Z^3$ given by the eight coordinate octants. We identify $\Sigma(3)$ with $(\Z/2)^3$ by letting $(a_1, a_2, a_3)$ correspond to the octant $\{(n_1, n_2, n_3)\in \Z^3\mid (-1)^{a_j}n_j\ge0\}$. The $\SL_2^3$-action on $(\P^1)^3$ described in \S\ref{subsec:NActionP13} restricts to $N_3\subset \SL_2^3$ to give an $N_3$-action on $(\P^1)^3$, with $N_3$-equivariant orientation $\tau:K_{(\P^1)^3}\to \sO_{(\P^1)^3}(-1,-1,-1)^{\otimes 2}$ and $N_3$-linearized very ample invertible sheaf $\sO_{(\P^1)^3}(1,1,1)$. 
	
	At each $0_{klm}\in (\P^1)^3$, the coordinate weight matrix $(s_{ij}^{klm})_{1\le i,j\le 3}$ for the $T$-action has $s_{ij}^{klm}=\pm 2$ for each $i,j$; the conditions of Definition~\ref{defn:N3Orient} are then easy to check. Thus,  we have an $N_3$-orientation for 
	$(\P^1)^3$. This gives us an infinite set of examples of $N_3$-oriented smooth proper toric threefolds by taking the family of all $N_3$-oriented iterated blow-ups of $(\P^1)^3$.
\end{exa}

\subsection{Induced $N$-actions}\label{subsec:NAction}
Let $X=X_\Sigma$ be an $N_3$-oriented smooth proper toric threefold over $k$, with set $\Sigma(3)$ of maximal simplices, $N_3$-equivariant spin structure $\tau:K_X\xrightarrow{\sim}L^{\otimes 2}$ and very ample $N_3$-linearized invertible sheaf $\sO_X(1)$.

Given a triple of odd integers $a,b,c$, the embedding $\iota_{a,b,c}:N\to \SL_2^3$ factors through $N_3$, so we can restrict the above data  to $N$, giving $X$ an $N$-action, an $N$-equivariant spin structure and an $N$-linearized very ample invertible sheaf. Since $\iota_{a,b,c}(-1)=(-1,-1,-1)$, the action factors through $N/\<-1\>$ and the image $\bar\sigma\in 
N/\<-1\>$ acts via the action of $\sigma_\delta$. 

\begin{lem}\label{lem:GenericWeights} Given an $N_3$-oriented smooth proper toric variety $X$ over $k$,  there is a non-zero polynomial $f\in\Z[x_1, x_2, x_3]$ such that for all odd positive integers $a,b,c$ with $f(a,b,c)\neq0$, and giving $X$ the $T_1$-action induced from the $N_3$-action via $\iota_{a,b,c}$,  we have
	\begin{enumerate}
		\item  The coordinate weights $(s_1^\kappa, s_2^\kappa, s_3^\kappa)$ for $T_1$ at $0_\kappa$ are all non-zero, for each $\kappa\in \Sigma(3)$.
		\item  The inclusion $X^T\subset X^{T_1}$ is a bijection.
\end{enumerate}\end{lem}

\begin{proof}  Since $X^T=\{0_\kappa\mid\kappa\in \Sigma(3)\}$, (2) follows from (1).
	
	For each $\kappa\in \Sigma(3)$, $U_\kappa$ is $T$-stable, hence is also  $T_1$-stable for each choice of  $a,b,c$. Suppose $T$ has coordinate weight matrix $(s_{ij}^\kappa)$ at $0_\kappa\in U_\kappa$, with respect to a choice of standard coordinates $x_1^\kappa, x_2^\kappa,x_3^\kappa$. Then, with respect to the $T_1$-action on $U_\kappa$ given by $\iota_{a,b,c}$, $T_1$ has coordinate weights
	$(s_1^\kappa, s_2^\kappa,s_3^\kappa)=(a,b,c)\cdot (s_{ij}^\kappa)$.
	
	Since the integer-valued matrices $(s_{ij}^\kappa)$ are all invertible in $M_{2\times 3}(\Q)$, the conditions $s_j^\kappa\neq0$ for all $\kappa\in \Sigma(3)$ and all $i=1,2,3$ for given $(a,b,c)\in \Z^3$, can be expressed as the simultaneous non-vanishing of the non-zero linear polynomials $L_{\kappa, j}(x_1, x_2, x_3)$ defined by
	\[
	(L_{\kappa, 1}(x_1, x_2, x_3),L_{\kappa, 2}(x_1, x_2, x_3),L_{\kappa, 3}(x_1, x_2, x_3)):=(x_1, x_2, x_3)\cdot (s_{ij}^\kappa). 
	\]
	We can thus take $f=\prod_{\kappa\in \Sigma(3)}\prod_{j=1}^3L_{\kappa, j}$.
\end{proof}

\begin{rem} By Lemma~\ref{lem:Nonzero}, there are infinitely many odd positive $a,b,c$  which satisfy Lemma~\ref{lem:GenericWeights}(1,2).   From now on, we will always assume that $a,b,c$ are so chosen.  
\end{rem}

\subsection{Independence} Let $X$ be an $N_3$-oriented smooth proper toric threefold over $k$, as in \S\ref{subsec:N3OrientedThreefolds}. Having chosen odd positive integers $a,b,c$, the embedding $\iota_{a,b,c}:N\to N_3\subset \SL_2^3$ gives us an $N$-action on $\Hilb^n(X)$. We want to study the resulting $N$-equivariant DT invariants $\tilde{I}_n^{N,a,b,c}(X)\in H^0(BN,\sW)$. However, since the $N$-action does not extend to an action of $\SL_2^3$ on $X$, we cannot conclude as we did for $(\P^1)^3$ that the $\tilde{I}_n^N(X)$ are independent of the choice of $a,b,c$ (see Remark~\ref{rem:Extension}). As we needed to change our choice of $a,b,c$ throughout the course of the proof of Theorem~\ref{theorem: main theorem}, we need some control on the dependence of the $\tilde{I}_n^{N,a,b,c}(X)$ on the choice of $a,b,c$. For this, we will use the fact that $\iota_{a,b,c}$ factors through $N_3$, so we have $\tilde{I}_n^{N,a,b,c}(X)=\iota_{a,b,c}^*\tilde{I}_n^{N_3}(X)$. We therefore need to understand the dependence of the pullback map $\iota_{a,b,c}^*:H^0(BN_3, \sW)_\Q\to
H^0(BN, \sW)_\Q$ on the choice of $a,b,c$, which is the object of this  section (see 
Proposition~\ref{prop:RealRealization} below for a precise statement).

We assume our base-field $k$ is a subfield of $\R$, and we give $BN(\R)$ and $BN_3(\R)$ the topology induced from the classical topology on $\R$.  We work over our base-field $k$, so all group-schemes will be group-schemes over $k$, for instance, we write $\GL_n$ for $\GL_n/k$. 

As preparation, recall that for a subgroup scheme $G$ of $\GL_n$, we construct the Ind-scheme $BG$ as 
\[
BG:=\colim_m G\backslash U_{n,m}
\]
where $U_{n,m}$ is the open subscheme of the affine space $M_{n\times m}\cong \A^{nm}_k$ of $n$ by $m$ matrices given by the matrices of  rank $n$, and $G$ acts (freely) on $U_{n,m}$ by restriction of the usual left action of $\GL_n/k$ on $M_{n\times m}$. We let $W_{n,m}\subset  M_{n\times m}$ be the closed complement $M_{n\times m}\setminus U_{n,m}$ and note that $W_{n,m}$ has codimension $m-n+1$.\footnote{The dimension of $n$ by $m$ matrices of rank $n-1$ is counted by adding the dimension of $n-1$ independent rows of size $m$ to the ways of writing the remaining row as a linear combination of the others $=(n-1)m+(n-1)=nm-(m-n+1)$. To conclude, note that the $n$ by $m$ matrices of rank $n-1$ is  dense in matrices of rank $<n$.} If we need to indicate the base-field $k$ explicitly, we write this as $BG/k$.

We recall that $\colim_m U_{n,m}$ is $\A^1$-contractible (see \cite[Proposition 4.2.3]{MorelVoevodskyAHTS}). We briefly recall the argument that shows that, given two closed immersions of group-schemes $i_j:G\hookrightarrow \GL_{n_j}/k$, $j=1,2$. we have a canonical $\A^1$-equivalence of $\colim_m G\backslash U_{n_1,m}$ with $\colim_m G\backslash U_{n_2,m}$, where $G$ acts on $U_{n_j,m}$ via $i_j$. For this, consider the action of $G$ on $U_{n_1,m}\times U_{n_2,m}$ via the diagonal embedding $(i_1, i_2):G\to \GL_{n_1}\times \GL_{n_2}$, giving projections
\[
\pi_{j,m}:G\backslash U_{n_1,m}\times U_{n_2,m}\to G\backslash U_{n_j,m}
\]
$j=1,2$. Since the action of $G$ on $U_{n_1,m}$ is free (for $m\ge n_1$), we have the free action of $G$ on $U_{n_1,m}\times M_{n_2\times m}$, making 
\[
G\backslash U_{n_1,m}\times M_{n_2\times m}\xrightarrow{\bar{\pi}_{1,m}} G\backslash U_{n_1, m}
\]
a vector bundle over $G\backslash U_{n_1, m}$, in particular, an $\A^1$-weak equivalence. Similarly,
\[
G\backslash M_{n_1\times m}\times U_{n_2, m}\xrightarrow{\bar{\pi}_{2,m}} G\backslash U_{n_2, m}
\]
is an $\A^1$-weak equivalence, and the open immersions 
\[
G\backslash U_{n_1,m}\times U_{n_2,m}\hookrightarrow G\backslash U_{n_1,m}\times M_{n_2\times m}
\]
\[
G\backslash U_{n_1,m}\times U_{n_2,m}\hookrightarrow G\backslash M_{n_1\times m}\times U_{n_2, m}
\]
induce $\A^1$-equivalences on the colimit, giving the asserted $\A^1$-weak equivalence.

Suppose we have a group-scheme $G$ over $k$ and $k$-schemes $Y,U$. Giving $Y$ a right $G$-action and $U$ a free left $G$-action, we give $Y\times_kU$ the free left $G$-action $g\cdot (y,u)=(y\cdot g^{-1}, gu)$, and write $Y\times^GU$  for the categorical quotient $G\backslash(Y\times_kU)$. In all our applications, this quotient will exist as a smooth quasi-projective $k$-scheme. The morphisms and $\A^1$-weak equivalences described in the preceeding paragraph extend to give morphisms and $\A^1$-weak equivalences after replacing $G\backslash (-)$ with $Y\times^G(-)=G\backslash Y\times_k(-)$. 

Using the closed immersions $\SL_2\subset \GL_2$, $\SL_2^3\subset \GL_2^3\subset \GL_6$, $\iota:N\hookrightarrow \SL_2$,  $\iota_{a,b,c}:N\hookrightarrow \SL_2^3$, $i:N_3\hookrightarrow \SL_2^3$, we have natural isomorphisms
\begin{align*}
	&(N\backslash_\iota \SL_2)\times^{\SL_2}U_{2, m}\cong N\backslash U_{2,m}\\
	&(N\backslash_{\iota_{a,b,c}} \SL_2^3)\times^{\SL_2^3}U_{6, m}\cong N\backslash_{\iota_{a,b,c}} U_{6,m}\\
	&(N_3\backslash\SL_2^3)\times^{\SL_2^3}U_{6,m}\cong N_3\backslash U_{6,m}.
\end{align*}
in particular, these all induce homeomorphisms on the respective spaces of $\R$-points (in the classical topology). Noting that for $m\gg0$, $U_{2, m}(\R)$ and $U_{6, m}(\R)$ are connected, and $\SL_2(\R)$ and $\SL_2^3(\R)=\SL_2(\R)^3$ are connected Lie groups, we have    bijections for all $m\gg0$
\begin{align}\label{align:pi0N}
	&\pi_0((N\backslash_\iota \SL_2)(\R))\cong \pi_0([(N\backslash_\iota \SL_2)\times^{\SL_2}U_{2, m}](\R))
	\\\notag
	&\hskip 100pt\cong \pi_0((N\backslash U_{2,m})(\R))\cong \pi_0(BN(\R))\\ \label{align:pi0Nabc}&\pi_0((N\backslash_{\iota_{a,b,c}} \SL_2^3)(\R))\cong \pi_0([(N\backslash_{\iota_{a,b,c}} \SL_2^3)\times^{\SL_2^3}U_{6, m}](\R))\\\notag
	&\hskip 100pt\cong \pi_0((N\backslash_{\iota_{a,b,c}} U_{6,m})(\R))
	\cong \pi_0(BN(\R))\\
	\label{align:pi0N3}
	&\pi_0((N_3\backslash\SL_2^3)(\R))\cong \pi_0([(N_3\backslash\SL_2^3)\times^{\SL_2^3}U_{6,m}](\R))\\\notag
	&\hskip 100pt\cong \pi_0([N_3\backslash U_{6,m}](\R))\cong \pi_0(BN_3(\R)).
\end{align}
Moreover, letting 
\[
\pi_{a,b,c}:(N\backslash_{\iota_{a,b,c}}\SL_2^3)\to (N_3\backslash \SL_2^3)(\R)
\]
be the map induced by $\iota_{a,b,c}:N\to N_3$, the diagram
\begin{equation}\label{eqn:Pi0Commute}
	\xymatrix{
		\pi_0((N\backslash_{\iota_{a,b,c}} \SL_2^3)(\R))\ar[r]^-\sim\ar[d]^{\pi_0(\pi_{a,b,c}(\R))}&\pi_0(BN(\R))\ar[d]^{\pi_0((B\iota_{a,b,c})(\R))}\\
		\pi_0((N_3\backslash\SL_2^3)(\R))\ar[r]^-\sim&\pi_0(BN_3(\R))
	}
\end{equation}
commutes.

In addition, taking $\R$-points in the diagram 
\[
\xymatrix{
	N\backslash_{(\iota,\iota_{a,b,c})} U_{2,m}\times M_{6\times m}\ar[d]&N\backslash_{(\iota,\iota_{a,b,c})} U_{2,m}\times U_{6,m}\ar[r]\ar[l]&N\backslash_{(\iota,\iota_{a,b,c})} M_{2\times m}\times U_{6,m}\ar[d]\\
	N\backslash_{\iota} U_{2,m} &&N\backslash_{\iota_{a,b,c}}  U_{6,m}
}
\]
gives us bijections for all $m\gg0$
\begin{multline}\label{mult:pi0}
	\pi_0(BN(\R))\cong \pi_0((N\backslash_{\iota} U_{2,m})(\R))\\\cong \pi_0((N\backslash_{(\iota,\iota_{a,b,c})} U_{2,m}\times U_{6,m})(\R))\cong \pi_0((N\backslash_{\iota_{a,b,c}}  U_{6,m})(\R)).
\end{multline}
Finally, we have corresponding diagrams of isomorphisms of $\<\bar\sigma\>\cong \Z/2$-schemes and bijections on the $\<\bar\sigma\>$-sets  $\pi_0$ for $\C$-points and for $\R$-points if we replace $N$ with $T_1$ and $N_3$ with $T$. 

By Hilbert's Theorem 90, taking quotients of a $k$-scheme $Y$ by a free $T_1$-action or a free $T$-action induces a surjection on $F$-points for each field $F\supset k$. 

\begin{defn}\label{defn:PlusMinus} 1. Let $Y$ be an irreducible smooth quasi-projective $\R$-scheme with a free action of $\<\bar\sigma\>=\Z/2$. Since $Y$ is an $\R$-scheme, we have an action of $\C$-conjugation, $z\mapsto \bar{z}$, on $Y(\C)$. Let $Y(\C)^-\subset Y(\C)$ be the set of $\C$-points
	\[
	Y(\C)^-=\{z\in Y(\C)\mid \bar\sigma(z)=\bar{z}\}.
	\]
	Define $(\<\bar\sigma\>\backslash Y)(\R)^+:=\<\bar\sigma\>\backslash Y(\R)$ and 
	$(\<\bar\sigma\>\backslash Y)(\R)^-:=\<\bar\sigma\>\backslash Y(\C)^-$.
	\\[2pt]
	2. Let $W$ be an  irreducible smooth quasi-projective $\R$-scheme with a free action of $N$. We suppose that the restriction of the $N$-action to $T_1$ represents $W$ as a principle $T_1$-bundle over an irreducible smooth quasi-projective $\R$-scheme $Y$,  which we give the induced free $\<\bar\sigma\>$-action.  Define $(N\backslash W)(\R)^\pm:=(\<\bar\sigma\>\backslash Y)(\R)^\pm$.
\end{defn}

\begin{lem}\label{lem:PlusMinus} Let $Y$ be an irreducible smooth quasi-projective $\R$-scheme with a free action of $\<\bar\sigma\>=\Z/2$.  Then $(\<\bar\sigma\>\backslash Y)(\R)$ is a disjoint union
	\[
	(\<\bar\sigma\>\backslash Y)(\R)=(\<\bar\sigma\>\backslash Y)(\R)^+\amalg (\<\bar\sigma\>\backslash Y)(\R)^-.
	\]
	2. Let $W$ be an irreducible smooth quasi-projective $\R$-scheme with a free action of $N$, such that the restriction to a $T_1$-action represents $W$ as a principle $T_1$-bundle over an
	irreducible smooth quasi-projective $\R$-scheme $Y$, $p:W\to Y$. Then $(N\backslash W)(\R)$ is a disjoint union
	\[
	(N\backslash W)(\R)=(N\backslash W)(\R)^+\amalg (N\backslash W)(\R)^-.
	\]
	
\end{lem}

\begin{proof} For (1), let $\pi:Y\to \<\bar\sigma\>\backslash Y$ be the quotient map. Since the $\bar\sigma$-action is free,  a closed point $y\in \<\bar\sigma\>\backslash Y$ has residue field $\R$ if and only if the inverse image closed subscheme $\pi^{-1}(y)$ satisfies one of the following.
	\begin{enumerate}
		\item[i)] $\pi^{-1}(y)(\C)=\pi^{-1}(y)(\R)=\{y_1, y_2\mid y_i\in Y(\R), y_2=\bar\sigma(y_1)\neq y_1\}$,
		\item[ii)] $\pi^{-1}(y)(\C)=\{z, \bar{z}\mid z\in Y(\C)\}$.
	\end{enumerate}
	In case (i), $y$ is in $(\<\bar\sigma\>\backslash Y)(\R)^+$ and in case (ii), $y$ is in $\<\bar\sigma\>\backslash Y(\C)^-$. If $x$ is in $Y(\C)^-$, then since the $\bar\sigma$-action is free, we must have $\bar\sigma(x)=\bar{x}\neq x$, so $x$ is not an $\R$-point of $Y$, hence 
	$(\<\bar\sigma\>\backslash Y)(\R)^+\cap (\<\bar\sigma\>\backslash Y)(\R)^-=\0$. 
	
	For (2), we have the canonical isomorphism
	\[
	N\backslash W\cong \<\bar\sigma\>\backslash (T_1\backslash W)= \<\bar\sigma\>\backslash Y,
	\]
	so 
	\[
	(N\backslash W)(\R)=(\<\bar\sigma\>\backslash Y)(\R)
	\]
	and (2) follows from (1).
\end{proof}

\begin{lem}\label{lem:RealConnComp} 1. $(N\backslash\SL_2)(\R)^+$ and $(N\backslash\SL_2(\R))^-$ are both connected.\\[2pt]
	2. For all triples of odd positive integers $a,b,c$, $(N\backslash_{\iota_{a,b,c}}\SL_2^3(\R))^+$ and
	$(N\backslash_{\iota_{a,b,c}}\SL_2^3)(\R)^-$  are both connected.\\[2pt]
	3. Via the bijection  \eqref{align:pi0N}, we can write $BN(\R)$ as the disjoint union of two connected components, $BN(\R)=BN(\R)^+\amalg BN(\R)^-$, with $BN(\R)^+$ the image of $(N\backslash\SL_2)(\R)^+$ and $BN(\R)^-$ the image of $(N\backslash\SL_2)(\R)^-$.
	\\[2pt]
	4. The bijection \eqref{align:pi0Nabc} sends  $(N\backslash_{\iota_{a,b,c}}\SL_2^3)(\R)^+$ to $BN(\R)^+$ and $(N\backslash_{\iota_{a,b,c}}\SL_2^3)(\R)^-$ to $BN(\R)^-$.
\end{lem}

\begin{proof} (2) follows from Lemma~\ref{lem:PlusMinus} and (1), using the bijections \eqref{align:pi0N} and \eqref{align:pi0Nabc}. 
	
	For (1), we make an explict computation. By Hilbert's Theorem 90, we have
	\[
	(T_1\backslash \SL_2)(\R)=T_1(\R)\backslash\SL_2(\R)
	\]
	so the map $\SL_2(\R)\to (N\backslash \SL_2)(\R)^+$ is surjective (and continuous). Similarly, projecting on the first column writes $\SL_2$ as an affine $\A^1$-bundle over $\A^2\setminus\{0\}$, writing $\SL_2(\R)$ as a fiber bundle $\R\to \SL_2(\R)\to \R^2\setminus\{0\}$ over $\R^2\setminus\{0\}$, so $\SL_2(\R)$ is connected, and hence so is 
	$(N\backslash \SL_2)(\R)^+$. 
	
	For the minus part,  let 
	\[
	\SL_2(\C)^-= \{g\in \SL_2(\C)\mid \exists t\in T_1(\C),\text{ such that } \begin{pmatrix}0&1\\-1&0\end{pmatrix}\cdot g=\begin{pmatrix}t&0\\0&t^{-1}\end{pmatrix}\bar{g}\}.
	\]
	Then since $\C$ is algebraically closed,  $(N\backslash \SL_2)(\R)^-$ is the image of $\SL_2(\C)^-$ under the quotient map $\SL_2(\C)\to (N\backslash \SL_2)(\C)=N(\C)\backslash \SL_2(\C)$. For $g=\begin{pmatrix}a&b\\c&d\end{pmatrix}\in \SL_2(\C)^-$,  the equation defining $\SL_2(\C)^-$ gives
	\[
	\begin{pmatrix}c&d\\-a&-b\end{pmatrix}=\begin{pmatrix}t\bar{a}&t\bar{b}\\t^{-1}\bar{c}&t^{-1}\bar{d}\end{pmatrix}
	\]
	for some $t\in \C^\times$. This is the same as 
	\[
	\frac{\bar{t}}{t}=-1,\ c=t\bar{a},\ d=t\bar{b}
	\]
	so $t=i\cdot s$ for some $s\in \R^\times$. Letting $\im(a\bar{b})$ denote the imaginary part of $a\bar{b}$, the condition $\det g=1$ says
	\[
	s=-\frac{1}{2\im(a\bar{b})}
	\]
	so 
	\begin{equation}\label{eqn:CorrespMatrix}
		g=\begin{pmatrix} a&b\\ \frac{1}{2i\im(a\bar{b})}\cdot \bar{a}&\frac{1}{2i\im(a\bar{b})}\cdot \bar{b}\end{pmatrix}
	\end{equation}
	giving a bijection of $\SL_2(\C)^-$ with $S:=\{(a,b)\in \C^2\mid \im(a\bar{b})\neq 0\}$. $S$ in turn is the inverse image of $\C\setminus \R$ under the map $m:\C^2\to \C$, $m(x,y)=x\bar{y}$. As $m$ is surjective with connected fibers, we see that $S$ has two connected components:
	\[
	S_+=\{(a,b)\in \C^2\mid \im(a\bar{b})> 0\},\ S_-=\{(a,b)\in \C^2\mid \im(a\bar{b})< 0\}.
	\]
	Finally, for $(a,b)\in S_+$, let $g$ be the matrix \eqref{eqn:CorrespMatrix} corresponding to $(a,b)$. Then 
	\[
	\sigma\cdot g=\begin{pmatrix}0&1\\-1&0\end{pmatrix}\cdot \begin{pmatrix} a&b\\ \frac{1}{2i\im(a\bar{b})}\cdot \bar{a}&\frac{1}{2i\im(a\bar{b})}\cdot \bar{b}\end{pmatrix}=
	\begin{pmatrix} \frac{1}{2i\im(a\bar{b})}\cdot \bar{a}&\frac{1}{2i\im(a\bar{b})}\cdot \bar{b}\\-a& -b\end{pmatrix}.
	\]
	Since 
	\[
	\left(\frac{1}{2i\im(a\bar{b})}\cdot \bar{a}\right)\cdot \overline{\left(\bar{a}\frac{1}{2i\im(a\bar{b})}\cdot \bar{b}\right)}=
	\frac{1}{4(\im(a\bar{b}))^2}\cdot \bar{a}\cdot b,
	\]
	we see that via \eqref{eqn:CorrespMatrix},  $\sigma$ interchanges $S_+$ and $S_-$, and thus 
	$(N\backslash \SL_2)(\R)^-$ is connected, completing the proof of (1).
	
	(3) follows from (1), using the bijection \eqref{align:pi0N}.
	
	For (4), since there are only two components to consider, we need only show that the  bijection \eqref{align:pi0Nabc} sends  $(N\backslash_{\iota_{a,b,c}}\SL_2^3(\R))^+$ to $BN(\R)^+$. Since $\SL_2^3$ is a special group, and using Hilbert's Theorem 90 again, we have
	\begin{align*}
		((T_1\backslash_{\iota_{a,b,c}}\SL_2^3)\times^{\SL_2^3}U_{6,m})(\R)&=
		(T_1\backslash_{\iota_{a,b,c}}\SL_2^3)(\R)\times^{\SL_2^3(\R)}U_{6,m}(\R)\\
		&=(T_1(\R)\times \SL_2^3)(\R)) \backslash (\SL_2^3(\R)\times U_{6,m})(\R)\\
		&=T_1(\R)  \backslash_{\iota_{a,b,c}}  U_{6,m}(\R).
	\end{align*}
	Similarly,
	\begin{align*}
		((T_1\backslash_{\iota}\SL_2)\times^{\SL_2}U_{2,m})(\R)&=
		(T_1\backslash_{\iota}\SL_2)(\R)\times^{\SL_2(\R)}U_{2,m}(\R)\\
		&=(T_1(\R)\times \SL_2(\R)) \backslash (\SL_2(\R)\times U_{2,m}(\R))\\
		&=T_1(\R)  \backslash_\iota  U_{2,m}(\R),
	\end{align*}
	and the respective images in $BN(\R)$ factor through the respective quotients by $\sigma$.
	
	We compare $T_1\backslash_{\iota_{a,b,c}}  U_{6,m}$ and 
	$T_1\backslash_{\iota}U_{2,m}$ via the $T_1$-action $(\iota,\iota_{a,b,c})$ on the product $U_{2,m}\times U_{6,m}$. This gives us the commutative diagram
	\[
	\xymatrix{
		U_{2,m}(\R)\times U_{6,m}(\R)\ar[r]^-{p_2}\ar[d]^{p_1}\ar[dr]^-p&U_{6,m}(\R)\ar[d]\\
		U_{2,m}(\R)\ar[r]&BN(\R).
	}
	\]
	Since $(U_{2,m}\times U_{6,m})(\R)$ is connected for $m\gg0$,  we see   that the map $\pi_0(p_1)$, 
	\[
	\{*\}=\pi_0(U_{2,m}(\R)\times U_{6,m}(\R))\xrightarrow{\pi_0(p_1)} \pi_0(BN(\R)),
	\]
	sends $*$ to $[BN(\R)^+]\in\pi_0(BN(\R))$, while via $p_2$, we see that $\pi_0(p_2)(*)$ is the image of 
	$(N\backslash_{\iota_{a,b,c}}\SL_2^3(\R))^+$ in $\pi_0(BN(\R))$, proving (4).
\end{proof}

\begin{nota}
	We let $H:=((\Z/4)^\times)^3/\{\pm1\}$, with $x\in \{\pm1\}=(\Z/4)^\times$ acting by    $x\cdot (\epsilon_1, \epsilon_2, \epsilon_3)=
	(x\epsilon_1, x\epsilon_2, x\epsilon_3)$. We have $H\cong (\Z/2)^2$ (non-canonically). Given a triple of odd integers $(a,b,c)$ we send $(a,b,c)$  to $H$ by first reducing modulo 4, and then mapping to $H$ under the quotient map $((\Z/4)^\times)^3\to H$. If $\tau\in H$ is the image of $(a,b,c)$ under this map, we write this as
	\[
	(a,b,c)\equiv \tau\mod 4/ \{\pm 1\}.
	\]
\end{nota}

\begin{lem} \label{lem:pi0Ind}
	Let $a,b,c$ be odd positive integers generating the unit ideal in $\Z$. We have the map
	\[
	\pi_{a,b,c}:N\backslash_{\iota_{a,b,c}}\SL_2^3\to N_3\backslash \SL_2^3
	\]
	induced by $\iota_{a,b,c}:N\to N_3$. Then
	\begin{enumerate}
		\item $(N_3\backslash \SL_2^3)(\R)^+$ is connected.
		\item  $\pi_{a,b,c}(\R)$ maps 
		$(N\backslash_{\iota_{a,b,c}} \SL_2^3)(\R)^+$ to $(N_3\backslash \SL_2^3)(\R)^+$.
		\item The map 
		\[
		\pi_0((N\backslash_{\iota_{a,b,c}}\SL_2^3)(\R))^- \xrightarrow{\pi_0(\pi_{a,b,c}(\R))} \pi_0((N_3\backslash \SL_2^3)(\R))^-
		\]
		depends only on $(a,b,c)\mod 4/ \{\pm1\}$.
	\end{enumerate}
	In consequence, the map $\pi_0(\pi_{a,b,c}(\R)):\pi_0((N\backslash_{\iota_{a,b,c}}\SL_2^3)(\R))\to \pi_0((N_3\backslash \SL_2^3)(\R))$ depends only on $(a,b,c)\mod 4/ \{\pm1\}$.
\end{lem}

\begin{proof} We proceed as in the description of the connected components  $(N\backslash \SL_2)(\R)^\pm$ of  $(N\backslash \SL_2)(\R)$ in the proof of Lemma~\ref{lem:RealConnComp}. 
	
	Lemma~\ref{lem:RealConnComp}(2) tells us that $(N\backslash_{\iota_{a,b,c}}\SL_2^3)(\R)$ consists of two connected components $(N\backslash_{\iota_{a,b,c}}\SL_2^3)(\R)^+$ and $(N\backslash_{\iota_{a,b,c}}\SL_2^3)(\R)^-$.
	
	We have the commutative diagram
	\[
	\xymatrix{
		&\SL_2^3(\R)\ar[dl]_{q_N(\R)}\ar[dr]^{q_{N_3}(\R)}\\
		(N\backslash_{\iota_{a,b,c}}\SL_2^3)(\R)^+\ar[rr]^{\pi_{a,b,c}(\R)}&&(N_3\backslash \SL_2^3)(\R)^+
	}
	\]
	and the maps $q_N(\R)$ and $q_{N_3}(\R)$ are surjective by Hilbert's Theorem 90. Since $\SL_2^3(\R)$ is connected,  $(N_3\backslash \SL_2^3)(\R)^+$ is connected as well. Since  the map $\pi_{a,b,c}(\R)$ sends $(N\backslash_{\iota_{a,b,c}}\SL_2^3)(\R)^+$ to $(N_3\backslash \SL_2^3)(\R)^+$, this proves (1) and (2).

	For (3), we analyze $(N_3\backslash \SL_2^3)(\R)^-$ and its relation to 
	$(N\backslash_{\iota_{a,b,c}}\SL_2^3)(\R)^-$.
	
	Let 
	\[
	\SL_2^3(\C)^-=\{(g_1, g_2, g_3)\in \SL_2^3(\C)\mid \exists t_1, t_2,t_3\in \C^\times\text{ with }\sigma\cdot g_i=t_i\cdot \overline{g_i},\ i=1,2,3\}.
	\]
	Give $T\backslash \SL_2^3$ the $\Z/2$-action via $\<\bar\sigma_\delta\>$. As in the proof of Lemma~\ref{lem:RealConnComp}, we have
	\[
	(T\backslash \SL_2^3)(\C)^-=T(\C)\backslash \SL_2^3(\C)^-,
	\]
	and $(N_3\backslash \SL_2^3)(\R)^-=\<\bar\sigma_\delta\>\backslash (T\backslash \SL_2^3)(\C)^-$. 
	
	We make the conditions defining $(T\backslash \SL_2^3)(\C)^-$ explicit: write
	\[
	g_j=\begin{pmatrix}a_j&b_j\\c_j&d_j\end{pmatrix},\ j=1,2,3
	\]
	so   the equation $\sigma\cdot g_j=t_i\cdot \overline{g_j}$ becomes
	\[
	\begin{pmatrix}c_j&d_j\\-a_j&-b_j\end{pmatrix}=\begin{pmatrix} t_j\overline{a_j}&t_j\overline{b_j}\\t_j^{-1}\overline{c_j}&t_j^{-1}\overline{d_j}\end{pmatrix}.
	\]
	As in the computation \eqref{eqn:CorrespMatrix} in the proof of  Lemma~\ref{lem:RealConnComp}, this translates into $t_j=is_j=\frac{1}{2i\cdot \im(a_j\overline{b_j})}$ with $s_j\in \R^\times$, and
	\[
	g_j=\begin{pmatrix} a_j&b_j\\ \frac{1}{2i\cdot \im(a_j\overline{b_j})}\cdot \overline{a_j}&\frac{1}{2i\cdot \im(a_j\overline{b_j})}\cdot \overline{b_j}\end{pmatrix}.
	\]
	
	Thus, we can parametrize $(T\backslash \SL_2^3)(\C)^-$ via triples $((a_1, b_1), (a_2, b_2), (a_3,b_3))\in (\C^2)^3$ such that $\im(a_j\overline{b_j})\neq0$ for $j=1,2,3$. This breaks up the set $(T\backslash \SL_2^3)(\C)^-$  into 8 connected components, indexed by the triple of signs
	\[
	\left(\frac{|\im(a_1\overline{b_1})|}{\im(a_1\overline{b_1})}, \frac{|\im(a_2\overline{b_2})|}{\im(a_2\overline{b_2})}, \frac{|\im(a_3\overline{b_3})|}{\im(a_3\overline{b_3})}\right)\in\{\pm1\}^3; 
	\]
	see the argument of Lemma~\ref{lem:RealConnComp} to see that fixing a triple of signs does yield a connected space. As the action of $\sigma_\delta=(\sigma, \sigma,\sigma)$ on $\SL_2^3(\C)$  induces the action
	\[
	(\epsilon_1, \epsilon_2, \epsilon_3)\mapsto (-\epsilon_1, -\epsilon_2, -\epsilon_3)
	\]
	on the triple of signs, we see that $(N_3\backslash\SL_2^3)(\R)^-$ has four connected components, in bijection with the set $H=(\Z/4)^\times/\{\pm1\}= \{\pm1\}^3/\{\pm1\}$, by sending the triple of matrices $(g_1,g_2, g_3)\in \SL_2^3(\C)$ parametrized by $((a_1, b_1), (a_2, b_2), (a_3,b_3))\in (\C^2)^3$ to 
	\[
	\left(\frac{|\im(a_1\overline{b_1})|}{\im(a_1\overline{b_1})}, \frac{|\im(a_2\overline{b_2})|}{\im(a_2\overline{b_2})}, \frac{|\im(a_3\overline{b_3})|}{\im(a_3\overline{b_3})}\right)\in  \{\pm1\}^3/\{\pm1\}=H.
	\]

	We now analyze $(T\backslash_{\iota_{a,b,c}} \SL_2^3)(\C)^-$ in similar fashion. To aid in the indexing, we write $n_1=a$, $n_2=b$, $n_3=c$.  By assumption, the ideal $(n_1, n_2, n_3)\subset \Z$ is the unit ideal, so we have $m_1, m_2, m_3\in \Z$ with $\sum_jm_jn_j=1$. Let  $(g_1, g_2, g_3)\in  \SL_2^3(\C)$ be a triple, and as above, write 
	\[
	g_j=\begin{pmatrix}a_j&b_j\\c_j&d_j\end{pmatrix},\ j=1,2,3.
	\]
	Then $(g_1, g_2, g_3)$ represents an element of  $(T\backslash_{\iota_{a,b,c}} \SL_2^3)(\C)^-$ if there is a $t\in \C^\times$ with 
	\[
	\begin{pmatrix}c_j&d_j\\-a_j&-b_j\end{pmatrix}=\begin{pmatrix} t^{n_j}\overline{a_j}&t^{n_j}\overline{b_j}\\t^{-n_j}\overline{c_j}&t^{-n_j}\overline{d_j}\end{pmatrix},\ j=1,2,3.
	\]
	This implies that $(\bar{t}/t)^{n_j}=-1$ for each $j$, hence
	\[
	\frac{\bar{t}}{t}=\left(\frac{\bar{t}}{t}\right)^{\sum_jm_jn_j}=(-1)^{\sum_jm_j}=-1
	\]
	since reducing the identity $\sum_jm_jn_j=1$ modulo two yields $\sum_jm_j\equiv 1\mod 2$. 
	
	Thus $t=is$ for some $s\in \R^\times$,and each matrix $g_j$ is of the form
	\[
	g_j=\begin{pmatrix} a_j&b_j\\ \frac{1}{2i\cdot \im(a_j\overline{b_j})}\cdot \bar{a}_j&\frac{1}{2i\cdot \im(a_j\overline{b_j})}\cdot \overline{b_j}\end{pmatrix}
	\]
	with the additional conditions that 
	\[
	\frac{1}{2i\cdot \im(a_j\overline{b_j})}=t^{n_j}=i^{n_j}s^{n_j},\ j=1,2,3,
	\]
	or 
	\[
	\frac{1}{2\cdot \im(a_j\overline{b_j})}=i^{n_j+1}\cdot s^{n_j},\ j=1,2,3.
	\]
	Thus, if $s>0$, we have
	\begin{align*} 
		\im(a_j\overline{b_j})<0&\text{ if }n_j\equiv 1\mod 4\\\notag
		\im(a_j\overline{b_j})>0&\text{ if }n_j\equiv -1\mod 4,
	\end{align*}
	and if $s<0$, we have 
	\begin{align*}
		\im(a_j\overline{b_j})>0&\text{ if }n_j\equiv 1\mod 4\\\notag
		\im(a_j\overline{b_j})<0&\text{ if }n_j\equiv -1\mod 4.
	\end{align*}
	
	Now take   $x\in (T_1\backslash_{\iota_{n_1,n_2,n_3}} \SL_2^3)(\C)^-$. Choose a  representative $(g_1, g_2, g_3)\in \SL_2^3(\C)$ for $x$, giving the element $s\in \R^\times$ as above. Then $(g_1, g_2, g_3)$ is in $(\SL_2^3)(\C)^-$ and under the quotient map
	\[
	(\SL_2^3)(\C)^-\to T(\C)\backslash (\SL_2^3)(\C)^-= (T\backslash \SL_2^3)(\C)^-,
	\]
	$(g_1, g_2, g_3)$ goes to the component of $(T\backslash \SL_2^3)(\C)^-$ parametrized by 
	$(-n_1, -n_2, -n_3)\mod 4$ if $s>0$ and to the component  parametrized by $(n_1, n_2, n_3)\mod 4$ if $s<0$. Thus, 
	the image of $x$ in $(N_3\backslash \SL_2^3)(\R)^-$ lands in the component parametrized by the image of $(n_1, n_2, n_3)$ in $H$, proving the Lemma.
\end{proof}

\begin{prop}\label{prop:RealRealization} Suppose our base-field $k$ is a subfield of $\R$. Let $(a,b,c)$ be odd positive integers, generating the unit ideal in $\Z$, and take $x\in H^0(BN_3,\sW)$. Then $B\iota_{a,b,c}^*(x)\in H^0(BN, \sW)[1/2]$ depends only on the image of $(a,b,c)$ in $H$.
\end{prop}

\begin{proof} We write $BN/k$ to emphasize the choice of base-field $k$. Recall from Proposition~\ref{prop:BNCoh} that we have an isomorphism $H^0(BN/k, \sW)\cong W(k)^2$, natural in $k$. In addition, letting $R(k)$ denote the set of embeddings   $k\hookrightarrow \bar{k}^r_\alpha$  with $\bar{k}^r_\alpha$ a real closed field,  the map sending $q\in W(k)$ to the collection of signatures $\{\sig_\alpha(q)\in \Z\}_{\alpha\in R(k)}$  has kernel the 2-primary torsion subgroup of $W(k)$, which is equal to the entire torsion subgroup of $W(k)$ \cite[Theorem II.7.3]{Scharlau}. Since $W(\bar{k}^r_\alpha)\cong \Z$ via the signature map, it suffices to prove the result for $k$ real closed. As $k$ is assumed to be a subfield of $\R$, and base-extension of real closed fields induces an isomorphism on $W(-)$, we may assume that we are working over the base-field $k=\R$.  
	
	Let $X$ be a smooth finite type scheme over $\R$. We have the real spectrum $X_r$ of $X$ and the real \'etale site on $X$, $X_\ret$ (see, e.g., \cite[\S 1]{Scheiderer}), with corresponding topoi of sheaves $\widetilde{X}_r$ and $\widetilde{X}_{\ret}$. By  \cite[Theorem 1.3]{Scheiderer}, these are equivalent topoi. We have the change of topology maps of sites $X_\ret\to X_\Nis\to X_\Zar$, and it follows from the construction of the comparison isomorphism $\widetilde{X}_r\cong \widetilde{X}_{\ret}$ in {\it loc.cit.} that  the composition 
	\[
	\widetilde{X}_r\cong \widetilde{X}_{\ret}\to \widetilde{X}_{\Nis}\to  \widetilde{X}_{\Zar}
	\]
	is induced by the continuous map $\supp:X_r\to X_\Zar$ (see \cite[Definition 4.1]{Jacobson} for the definition of $\supp$). 
	By \cite[Theorem 8.9]{Jacobson}, the signature map \cite[Definition 8.1]{Jacobson} induces an isomorphism $\sW[1/2]\xymatrix{\ar[r]^\sig_\sim&}\supp_*\Z[1/2]$ and by \cite[Lemma 4.6]{Jacobson} this in turn induces an isomorphism
	\[
	H^*(X_\Zar, \sW[1/2]) \xymatrix{\ar[r]^-{H^*(\sig)}_-\sim&}H^*(X_r, \Z[1/2]).
	\]
	Thus, the change of topology map $X_\ret\to X_\Zar$ induces a natural isomorphism 
	\[
	H^*(X_\Zar,\sW)[1/2]\xrightarrow{\sim} H^*(X_\ret, \Z[1/2]). 
	\]
	Since the change of topology map $H^*(X_\Zar,\sW)\to H^*(X_\Nis, \sW)$ is also an isomorphism, by \cite[Corollary 5.43]{MorelATF}, we have the natural isomorphism $H^*(X_\Nis, \sW[1/2])\cong H^*(X_\ret, \Z[1/2])$. Finally, by \cite[Theorem II.5.7]{Delfs}, we have the natural isomorphism 
	$H^*(X_\ret, \Z[1/2])\cong H^*(X(\R), \Z[1/2])$. Composing these isomorphisms and taking the cohomology derived limit over the smooth, finite-type approximations defining $BN_3$ and $BN$, we thus have the natural isomorphisms
	\begin{gather*}
		H^*(BN,\sW)[1/2]\xymatrix{\ar[r]^{Re_\R}_\sim&} H^*(BN(\R),\Z[1/2]),\\
		H^*(BN_3,\sW)[1/2] \xymatrix{\ar[r]^{Re_\R}_\sim&} H^*(BN_3(\R),\Z[1/2]).
	\end{gather*}
	These isomorphisms can also be constructed using  \cite[Theorem 35, Proposition 36]{Bachmann}.

	Thus, we can compute $B\iota_{a,b,c}^*(x)\in H^0(BN,\sW)[1/2]$ by computing the image of $x\in H^0(BN(\R), \Z[1/2])$ under the map
	\[
	H^0((B\iota_{a,b,c})(\R)^*):H^0(BN_3(\R), \Z[1/2])\to H^0(BN(\R), \Z[1/2]).
	\]
	But the map $H^0((B\iota_{a,b,c})(\R)^*)$ is determined by the map 
	\[
	\pi_0((B\iota_{a,b,c})(\R)):\pi_0(BN(\R))\to \pi_0(BN_3(\R)),
	\]
	which, by Lemma~\ref{lem:pi0Ind} and the commutativity of the diagram \eqref{eqn:Pi0Commute}, depends only on the image of $(a,b,c)$ in $H$.
\end{proof}

\subsection{Equivariant DT invariants}  We now take $k=\R$.
Let $X$ be an $N_3$-oriented smooth and proper toric threefold over $\R$, with its $N_3$-action $\rho_X:N_3\to \Aut_k(X)$, its $N_3$-equivariant spin structure $\tau:K_X\xrightarrow{\sim}L^{\otimes 2}$ and an $N_3$-linearized very ample invertible sheaf $\sO_X(1)$. We then have for each $n\ge1$ the Hilbert scheme $\Hilb^n(X)$ with induced $N_3$-action, the $N_3$-equivariant perfect obstruction theory $\phi_n(X): E_\bullet^{DT(n)}(X)\to \L_{\Hilb^n(X)/k}$ and the $N_3$-equivariant orientation $\rho_n(X):\det E_\bullet^{DT(n)}(X)\xrightarrow{\sim} L_n(X)^{\otimes 2}$.

\begin{defn} Let $\tilde{I}_n(X)\in W(\R)=\Z$ be the DT invariant 
	\[
	\tilde{I}_n(X):=\deg_\R^{\rho_n(X)}([\Hilb^n(X),\phi^{DT}_n(X)]^\vir_{\EM(\sW_*)}),
	\]
	and let $\tilde{I}_n^{N_3}(X)\in H^0(BN_3,\sW)$ be the $N_3$-equivariant version. Finally, given odd positive integers $a,b,c$, let $\tilde{I}_n^{N, a,b,c}(X)\in H^0(BN,\sW)$ be the $N$-equivariant DT invariant given by the $N$-action and spin structure induced via restriction from $N_3$ via $\iota_{a,b,c}:N\to N_3$. 
\end{defn}
In what follows, we omit the mention of the orientation $\rho_n(X)$ from the notation.

\begin{lem}\label{lem:IndABC} If $a,b,c$ are odd positive integers that generate the unit ideal in $\Z$, then  $\tilde{I}_n^{N, a,b,c}(X)\in H^0(BN,\sW)$ depends only on the image of $(a,b,c)$ in $((\Z/4)^\times)^3/\{\pm1\}$.
\end{lem}

\begin{proof} Since we are taking $k=\R$, we have $H^0(BN,\sW)=W(\R)^2=\Z^2$, so the map $H^0(BN,\sW)\to H^0(BN,\sW)[1/2]$ is injective, and it suffices to see that $\tilde{I}_n^{N, a,b,c}(X)\in H^0(BN,\sW)[1/2]$ depends only on the image of $(a,b,c)$ in $H=((\Z/4)^\times)^3/\{\pm1\}$. Since $\tilde{I}_n^{N, a,b,c}(X)=\iota_{a,b,c}^*(\tilde{I}_n^{N_3}(X))$, this follows from Proposition~\ref{prop:RealRealization}.
\end{proof}

Applying Proposition~\ref{prop:BNCoh} in case $k=\R$ gives us the canonical isomorphism  $H^0(BN, \sW)[1/e]\cong W(\R)$, which in turn is isomorphic to $\Z$ by the signature map.

\begin{defn} 1. For $\tau\in H$, let  $a,b,c$ be odd positive integers that generate the unit ideal in $\Z$ such that $(a,b,c)\equiv\tau\mod 4/ \{\pm1\}$. Define
	\[
	\tilde{I}_n^{N, \tau}(X)\in H^0(BN,\sW)[1/e]=W(\R)=\Z
	\]
	to be the image of $\tilde{I}_n^{N, a,b,c}(X)$ in $H^0(BN,\sW)[1/e]$; this is well-defined by Lemma~\ref{lem:IndABC}.
	\\[2pt]
	2. Let $\tilde{Z}^{N,\tau}(X,q)$ be the generating series
	\[
	\tilde{Z}^{N,\tau}(X,q)=1+\sum_{n\ge1}\tilde{I}_n^{N, \tau}(X)\in \Z[[q]].
	\]
\end{defn}

\begin{rem} Let $a,b,c$ be odd positive integers generating the unit ideal, giving the $N$-action on $X$ via $\iota_{a,b,c}$. We have the obstruction bundle $\text{ob}_2(X)$ on $\Hilb^2(X)$ and the degree of the corresponding equivariant Euler class (after localizing) 
	\[
	\deg_\R(e^{N,a,b,c}(\text{ob}_2(X))\in  H^0(BN, \sW)[1/e]=W(\R)=\Z.
	\]
	As $\Hilb^2(X)$ is smooth, it follows from the proof of  \cite[Proposition 5.6]{BehrendFantechiISNC} that  $e^{N,a,b,c}(\text{ob}_2(X))$ is the virtual fundamental class $[\Hilb^2(X),\phi^{DT}_2(X)]^{\vir, N}_{\EM(\sW_*)}$, so we have 
	\[
	\deg_\R(e^{N,a,b,c}(\text{ob}_2(X))=\tilde{I}_2^{N,a,b,c}(X),
	\]
	and thus $\deg_\R(e^{N,a,b,c}(\text{ob}_2(X)))$ depends only on the image $\tau$ of $(a, b,c)$ in $H$. We therefore write $\deg_\R(e^{N,\tau}(\text{ob}_2(X)))$ for $\deg_\R(e^{N,a,b,c}(\text{ob}_2(X)))$.
\end{rem}

Our goal is to prove the following theorem. 

\begin{theorem}\label{thm:MainBlowup} Let $X$ be an $N_3$-oriented smooth proper toric threefold over $\R$.  For each $\tau\in ((\Z/4)^\times)^3/\{\pm1\}$, we have
	\[
	\tilde{Z}^{N,\tau}(X,q)=M(q^2)^{\deg_\R(e^{N,\tau}(\text{ob}_2(X)))}.
	\]
	Moreover, we have
	\[
	\deg_\R(e^{N,\tau}(\text{ob}_2(X)))=-\frac{1}{2}\deg_\R(c_3((\text{ob}_1(X)))=\frac{1}{2}\deg_\R(c_3(T_X\otimes K_X)),
	\]
	and
	\[
	\tilde{Z}^{N,\tau}(X,q)=M(q^2)^{\frac{1}{2}\deg_\R(c_3(T_X\otimes K_X))}.
	\]
	In particular, $\deg_\R(e^{N,\tau}(\text{ob}_2(X)))$ and $\tilde{Z}^{N,\tau}(X,q)$ are independent of the choice of $\tau\in H$. 
\end{theorem}

\begin{defn}\label{defn:CommonValue} We let $\widetilde{\deg}_\R(e^N(\text{ob}_2(X)))\in \Z$ and $\tilde{Z}^N(X,q)\in \Z[[q]]$ denote   the respective common values of  $\deg_\R(e^{N,\tau}(\text{ob}_2(X)))$ and $\tilde{Z}^{N,\tau}(X,q)$, for $\tau\in H$.
\end{defn}

As immediate consequence (see Example~\ref{exa:P13Blowup}) we find the following corollary. 

\begin{cor}\label{cor:MainBlowup}  1. Let $X$ be an $N_3$-equivariant  iterated  blow-up of $(\P^1_\R)^3$. Then 
	\[
	\tilde{Z}^N(X,q)=M(q^2)^{\widetilde{\deg}_\R(e^N(\text{ob}_2(X)))}\in W(\R)[[q]]=\Z[[q]].
	\]
	2. Let $a,b,c$ be odd positive integers generating the unit ideal in $\Z$, and define  $\widetilde{I}^{N,a,b,c}_n(X)^*\in \Z$ to be the image of $\widetilde{I}^{N,a,b,c}_n(X)\in H^0(BN, \sW)$ under the localization map $H^0(BN, \sW)\to H^0(BN, \sW)[1/e]\cong W(\R)\cong \Z$.   
	Then $\widetilde{I}^{N,a,b,c}_n(X)^*\in\Z$ is the coefficient of $q^n$ in $\tilde{Z}^N(X,q)$. In particular, $\widetilde{I}^{N,a,b,c}_n(X)^*$ is independent of the choice of $a,b,c$ and is zero for $n$ odd.
\end{cor}

\begin{exa}
	Let $X$ be the blowup of $(\P^1)^3$ in the $\sigma$-orbit $\alpha_1$. Using a SAGE-implementation, one can compute that  $\widetilde{I}^{N,a,b,c}_2(X)^* = -3$, $\widetilde{I}^{N,a,b,c}_4(X)^* = -3$ and $\widetilde{I}^{N,a,b,c}_6(X)^* = 8$. Using Corollary \ref{cor:MainBlowup} and the fact that $\widetilde{I}^{N,a,b,c}_2(X)^* = -3$, we have that
	$\tilde{Z}^N(X,q)=M(q^2)^{-3}$, which is compatible with the above explicit calculations.
	
	Similarly, let $X'$ be the blowup of $X$ in the $\sigma$-orbit $\alpha_2$, then one  computes that $\widetilde{I}^{N,a,b,c}_2(X')^* = -6$, $\widetilde{I}^{N,a,b,c}_4(X')^* = 3$ and $\widetilde{I}^{N,a,b,c}_6(X')^* = 34$, and we have that 
	$\tilde{Z}^N(X',q)=M(q^2)^{-6}$. 
\end{exa}

The proof of Theorem~\ref{thm:MainBlowup} follows the same line of argument was used to prove Theorem~\ref{theorem: main theorem}. Since all the main ideas are essentially the same, we will content ourselves with pointing out the main steps and indicating what needs to be modified.

As in  \S\ref{subsec:ToricRecoll}, we have our open cover   $\{U_\kappa\}_{\kappa\in \Sigma(3)}$ of $X=X_\Sigma$. For an ideal sheaf $\sI$ on $X$, we let $\sI_\kappa$ denote the restriction to $U_\kappa$. If $\sI$ is supported in $X^T$, we write $\sI=\cap_{\kappa\in \Sigma(3)}\sI_{0_\kappa}$ where the ideal sheaf $\sI_{0_\kappa}\subset \sO_X$ in this intersection is uniquely determined by requiring that $\sI_{0_\kappa}$ has support $\{0_\kappa\}$.

\begin{lem}\label{lem:1} Given $M>0$, and $\tau\in H$,  there are infinitely many positive odd integers $a,b,c$ such that, letting $T_1, N$ act on $X$ via $\iota_{a,b,c}$, we have
	\begin{enumerate}
		\item $(a,b,c)\equiv \tau\mod 4/ \{\pm1\}$, and $a,b,c$ generate the unit ideal in $\Z$.
		\item  $X^{T}=X^{T_1}$.
	\end{enumerate}
	Moreover, for each ideal sheaf $\sI$ on $X$, fixed by $N$ and of co-length $2n\le 2M$, we have
	\begin{enumerate}
		\item[(3)] $[\sI]\in \Hilb^{2n}(X)$ has residue field $\R$.
		\item[(4)] For each $0_\kappa\in X^T$, the restriction $\sI_\kappa$ of $\sI$ to $U_\kappa\cong \Spec \R[x_1^\kappa,x_2^\kappa,x_3^\kappa]$ is a monomial ideal.
		\item[(5)] For each $0_\kappa\in X^T$, there are no $T_1$-trivial subrepresentations in $\Ext^i(\sI_\kappa, \sI_{|U_\kappa})$, $i=1,2$.
	\end{enumerate}
\end{lem}
The proof follows the same arguments as used to prove Lemma~\ref{lem:NumericalConditions}, where instead of requiring only that $a,b,c$ are odd and positive, we use suitable conditions in Lemma~\ref{lem:Nonzero} to give infinitely many positive integers $a,b,c$ such that $(a,b,c)\equiv\tau\mod 4/ \{\pm1\}$, and $a,b,c$ generate the unit ideal, with $a,b,c$ satisfying (2)-(5). 

\begin{lem}\label{lem:2} Fix $M>0$,  $\tau\in H$, and  a triple of positive integers $(a,b,c)$ satisfying the conditions of Lemma~\ref{lem:1}. We let $N$ act on $X$ via $\iota_{a,b,c}$ and let $n>0$ be an integer with $2n\le 2M$.   Then in $H^0(BN, \sW)[1/e]$ we have
	\begin{equation}\label{equation: virtual localization2} 
		\tilde{I}_{2n}^{N,\tau} = \sum_{[\sI]\in \text{Hilb}^n(X)^{N}} \frac{e_N(\Ext^2(\sI,\sI))}{e_N(\Ext^1(\sI,\sI))}\in W(\R)=\Z.
	\end{equation} 
	Moreover $\tilde{I}_{n'}^{N,\tau}=0$ for every odd integer $n'>0$.
\end{lem}
The proof is the same as the proof of Proposition~\ref{prop:LocFormula}, where we use in addition that fact that $W(\R)$ injects into $W(\R)[1/m]$ for each positive integer $m$.

Similarly, the same argument as used for Proposition~\ref{prop:NEulerClassLocal} gives us the next result. 
\begin{lem}\label{lem:3} We use the embedding $\iota_{a,b,c}$ to define the $N$-action on $X$,  and suppose $(a,b,c)$ satisfies the conditions of Lemma~\ref{lem:1} for given $M>0$ and $\tau\in H$. Let $\sI$ be an $N$-fixed ideal sheaf on $X$ of co-length $2n\le 2M$ with $\sI$ supported on $\alpha_\kappa=\{0_\kappa, 0_{-\kappa}\}$,  let $s:=(s_1, s_2, s_3)$ be the coordinate weights for $T_1$ at $0_\kappa$. 
	
	We write the trace $V_{\sI_\kappa}(t)$ for the virtual $T_1$-representation $\Ext^1(\sI_\kappa,\sI_\kappa)- \Ext^2(\sI_\kappa,\sI_\kappa)$ as the finite sum  $V_{\sI_\kappa}(t) = \sum_{w \in \Z} v_w t^w$, $v_w\in \Z$. Then  
	\begin{equation}
		\frac{e_N(\Ext^2(\sI,\sI))}{e_N(\Ext^1(\sI,\sI))} = \prod_{w\in \Z}(\epsilon(w)\cdot w)^{-v_w}\in \Z.
	\end{equation}
\end{lem}
The corresponding analog of Corollary~\ref{cor:SumFormula} holds as well. 

Next, we have an analog of Proposition~\ref{proposition: signs}. We first prove an elementary lemma. 

\begin{lem} Let $a,b,c$ be odd positive integers and let $N$ act on $X=X_\Sigma$ via the embedding $\iota_{a,b,c}:N\to N_3$. For every $\kappa\in \Sigma(3)$, let $s_1^\kappa, s_2^\kappa, s_3^\kappa$ be the coordinate weights for $T_1$ at $0_\kappa$, with respect to  standard affine coordinates $x_1^\kappa, x_2^\kappa, x_3^\kappa$ on $U_\kappa$. Then each $s_j^\kappa$ is even and $\sum_{j=1}^3s_j^\kappa\equiv2\mod 4$.
\end{lem}

\begin{proof} Let $(s_{ij}^\kappa)_{1\le i,j\le 3}$ be the coordinate weight matrix for $T$ at $0_\kappa$ (with respect to $x_1^\kappa, x_2^\kappa, x_3^\kappa$). Then
	\[
	(s_1^\kappa, s_2^\kappa, s_3^\kappa)=(a,b,c)\cdot (s_{ij}^\kappa).
	\]
	By  Definition~\ref{defn:N3Orient},  $\sum_is_{ij}^\kappa$ is even for each $j=1,2,3$, $\sum_js_{ij}^\kappa$ is even for each $i=1,2,3$,  and 
	$\sum_{i,j}s_{ij}^\kappa\equiv 2\mod 4$. But
	\[
	s_j^\kappa=a\cdot s_{1j}^\kappa+b\cdot s_{3j}^\kappa+c\cdot s_{3j}^\kappa\equiv \sum_is_{ij}^\kappa\equiv 0\mod 2
	\]
	so each $s_j^\kappa$ is even. Similarly 
	\[
	\sum_js_j^\kappa=a\sum_j s_{1j}^\kappa+b\sum_j s_{2j}^\kappa+c \sum_j s_{3j}^\kappa\equiv
	\sum_{i,j} s_{ij}^\kappa\equiv 2 \mod 4
	\]
	since $\sum_j s_{ij}^\kappa$ is even for each $i$.
\end{proof}

Thus, given odd positive integers $a,b,c$ and $0_\kappa\in X^T$, letting $(s_1^\kappa, s_2^\kappa, s_3^\kappa)$  be the coordinate weights at $0_\kappa$ for the $N$-action given by $\iota_{a,b,c}$,  each $s_i^\kappa$ is even and $\sum_is_i^\kappa\equiv 2\mod 4$, just as for $X=(\P^1)^3$. These were the only properties of $s_1, s_2, s_3$ needed for the proof of Proposition~\ref{proposition: signs} beyond the hypotheses in Proposition~\ref{proposition: signs}(1), so the same argument gives the following result. 

\begin{lem}\label{lem:4} We use the embedding $\iota_{a,b,c}$ to define the $N$-action on $X$,  and suppose $(a,b,c)$ satisfies the conditions of Lemma~\ref{lem:1} for given $M>0$ and $\tau\in H$.   Let $\sI$ be an $N$-invariant  ideal sheaf supported on $\alpha_\kappa=\{0_\kappa, \sigma(0_\kappa)\}$ of co-length $2n\le 2M$.
	
	Write $\sI=I_{0_\kappa}\cap\sigma(I_{0_\kappa})$ with
	$I_{0_\kappa}$ supported at $0_\kappa$.  $V_{\sI_\kappa}(t) = \sum_{w\in \Z} v_w(\sI_\kappa) t^w$ be the trace for $\sI_\kappa$ with respect to the induced $T_1$-action. Then  we have  
	\begin{enumerate}
		\item[i)] $\prod_{w\in\Z}\epsilon(w)^{v_w(\sI_\kappa)}=(-1)^n$.
		\item[ii)]  
		\[
		\frac{e_N(\Ext^2(\sI_\kappa,\sI_\kappa))}{e_N(\Ext^1(\sI_\kappa, \sI_\kappa)} = (-1)^n\prod_{w\in \Z} w^{-v_w(\sI_\kappa)}.
		\]
	\end{enumerate}
\end{lem}

Lemma~\ref{lem:4} gives us the following analog of Corollary~\ref{cor: signs}. As above, we use the embedding $\iota_{a,b,c}$ to define the $N$-action on $X$,  and suppose $(a,b,c)$ satisfies the conditions of Lemma~\ref{lem:1} for given $M>0$ and $\tau\in H$. 

\begin{cor}\label{cor: signs'} Let $\sI$ be an $N$-stable ideal sheaf on $X$ of co-length $2n\le 2M$, and write $\sI=\cap_{0_\kappa\in X^T}\sI_{0_\alpha}$. Let $\Sigma(3)^0\subset \Sigma(3)=X^T$ be a set of representatives for the   action of $\sigma$ on $X^T$. For each $0_\kappa\in \Sigma(3)^0$, 
	let $s^\kappa:=(s_1^\kappa, s_2^\kappa, s_3^\kappa)$ be the coordinate weights at $0_\kappa$, and let $V_\kappa(t)=\sum_{w\in \Z}v_w^\kappa t^w$ be the trace of $\Ext^1(\sI_{0_\kappa},\sI_{0_\kappa})-\Ext^2(\sI_{0_\kappa},\sI_{0_\kappa})$ as virtual $T_1$-representation. Then
	\[
	\frac{e_N(\Ext^2(\sI,\sI))}{e_N(\Ext^1(\sI,\sI))}=(-1)^n\prod_{\kappa\in \Sigma(3)^0}
	\prod_{w\in \Z} w^{-v^\kappa_w}\in \Z.
	\]
\end{cor}

We can now complete the proof of Theorem~\ref{thm:MainBlowup}.

\begin{proof}[Proof of Theorem~\ref{thm:MainBlowup}]
	We proceed as in \S\ref{sec:MainProof}, retaining the definition of $W(s, q)$ and $W'(s,q)$ from \S\ref{sec:MainProof}. Putting together Lemma~\ref{lem:2}, Corollary~\ref{cor: signs'},  Proposition~\ref{proposition: signs}(2), and \eqref{eqn:VertexIdentity} gives
	\[
	\prod_{\kappa\in \Sigma(3)^0}W(s_1^\kappa, s_2^\kappa, s_3^\kappa, q)=M(q^2)^{-\sum_{\kappa\in \Sigma(3)^0} \frac{(s_1^\kappa+s_2^\kappa)(s_1^\kappa+s_3^\kappa)(s_2^\kappa+s_3^\kappa)}{s_1^\kappa s_2^\kappa s_3^\kappa}},
	\]
	and for each positive integer $M>0$,
	\[
	\tilde{Z}^{N, \tau}(X,q)\equiv M(q^2)^{-\sum_{\kappa\in \Sigma(3)^0} \frac{(s_1^\kappa+s_2^\kappa)(s_1^\kappa+s_3^\kappa)(s_2^\kappa+s_3^\kappa)}{s_1^\kappa s_2^\kappa s_3^\kappa}}
	\in \Z[[q]]/(q^{2M+1}).
	\]
	Taking $M\ge 1$ and looking at the coefficient of $q^2$, this gives
	\[
	\tilde{I}_2^{N,\tau}(X)=-\sum_{\kappa\in \Sigma(3)^0} \frac{(s_1^\kappa+s_2^\kappa)(s_1^\kappa+s_3^\kappa)(s_2^\kappa+s_3^\kappa)}{s_1^\kappa s_2^\kappa s_3^\kappa},
	\]
	and since $M>0$ can be taken arbitrarily large, we have
	\[
	\tilde{Z}^{N, \tau}(X,q)= M(q^2)^{\tilde{I}_2^{N,\tau}(X)}.
	\]

	On the other hand, consider the $T_1$-action on $X$ and the equivariant Chern classes $c_3^{T_1}((T_X\otimes K_X)_\kappa)$ and $c_3^{T_1}((T_X)_\kappa)$ at the $T_1$-fixed point $0_\kappa$. Since $(T_X)_\kappa$ has weights $-s_1^\kappa,-s_2^\kappa,-s_3^\kappa$ and $K_X$ has 
	weight $s_1^\kappa+s_2^\kappa +s_3^\kappa$, the Bott residue formula for this $T_1$-action on $X$ gives us
	\begin{align*}
		\deg_\R(c_3(T_X\otimes K_X))&=\sum_{\kappa\in X^T}\frac{c_3^{T_1}((T_X\otimes K_X)_\kappa)}{c_3^{T_1}((T_X)_\kappa)}\\
		&=-\sum_{\kappa\in X^T}\frac{(s_1^\kappa+s_2^\kappa)(s_1^\kappa+s_3^\kappa)(s_2^\kappa+s_3^\kappa)}{s_1^\kappa s_2^\kappa s_3^\kappa}.
	\end{align*}
	The $\sigma$-action satisfies
	\[
	s_j^{\sigma(\kappa)}=-s_j^\kappa
	\]
	and since the expression for $\frac{c_3^{T_1}((T_X\otimes K_X)_\kappa)}{c_3^{T_1}((T_X)_\kappa)}$ is homogeneous of weight 0 in the $s_i^\kappa$, we see that
	\[
	\deg_\R(c_3(T_X\otimes K_X))=-2\sum_{\kappa\in \Sigma(3)^0}\frac{(s_1^\kappa+s_2^\kappa)(s_1^\kappa+s_3^\kappa)(s_2^\kappa+s_3^\kappa)}{s_1^\kappa s_2^\kappa s_3^\kappa}=
	2\tilde{I}_2^{N,\tau}(X),
	\]
	showing that $\tilde{I}_2^{N,\tau}(X)=\frac{1}{2}\deg_\R(c_3(T_X\otimes K_X))$ is independent of the choice of $\tau\in H$. Thus 
	\[
	\tilde{Z}^{N,\tau}(X,q)=M(q^2)^{\frac{1}{2}\deg_\R(c_3(T_X\otimes K_X))}
	\]
	is also independent of the choice of $\tau\in H$. Finally, the obstruction bundle $\text{ob}_1(X)$ on $X$ is the dual of $T_X\otimes K_X$ (see Lemma~\ref{lem:Ob1} below), so we have
	\[
	\tilde{I}_2^{N,\tau}(X)=\frac{1}{2}\deg_\R(c_3(T_X\otimes K_X))=
	-\frac{1}{2}\deg_\R(c_3(\text{ob}_1(X))),
	\]
	completing the proof of Theorem~\ref{thm:MainBlowup}.
\end{proof}

We conclude with a proof of this last assertion about $\text{ob}_1(X)$.

\begin{lem}\label{lem:Ob1} Let $X$ be a smooth projective threefold over $\R$ with $H^i(X,\sO_X)=0$ for $i=1,2$, and let $\text{ob}_1(X)$ be the obstruction bundle on $\Hilb^1(X)=X$ for the Donaldson-Thomas perfect obstruction theory $E_\bullet^{DT(1)}(X)$. There is a natural isomorphism
	\[
	\text{ob}_1(X)\cong (T_X\otimes K_X)^\vee.
	\]
\end{lem}

\begin{proof} Let $i_\Delta:X\to X\times_kX$ be the diagonal morphism and let  $\sI_\Delta$ be the ideal of the diagonal. We have the trace map $R\sHom(\sI_\Delta, \sI_\Delta)\to \sO_{X\times X}$, which gives a splitting of $R\sHom(\sI_\Delta, \sI_\Delta)$ as
	\[
	R\sHom(\sI_\Delta, \sI_\Delta)=R\sHom(\sI_\Delta, \sI_\Delta)_0\oplus \sO_{X\times X}.
	\]
	By the definition of $E_\bullet^{DT(1)}(X)$ following \cite[Lemma 1.22]{Behrend-FantechiSOTHSPT}, we have 
	\[
	E_\bullet^{DT(1)}(X)=Rp_{1*}(R\sHom(\sI_\Delta, \sI_\Delta)_0\otimes p_2^*K_X[2]).
	\]
	By Grothendieck-Serre duality, noting that $R\sHom(\sI_\Delta, \sI_\Delta)_0$ is self-dual, this gives
	\[
	E_\bullet^{DT(1)}(X)^\vee=Rp_{1*}(R\sHom(\sI_\Delta, \sI_\Delta)_0[1])
	\]

	Since $X$ is smooth,  $h^i(E_\bullet^{DT(1)}(X)^\vee)$ is a locally free sheaf on $X$ for $i=0,1$, $h^0(E_\bullet^{DT(1)}(X)^\vee)=T_X$ and 
	$h^1(E_\bullet^{DT(1)}(X)^\vee)=\text{ob}_1(X)$ (see \cite[Proposition 1.20]{Behrend-FantechiSOTHSPT}). Using the vanishing of $H^i(X, \sO_X)$, $i=1,2$, we have
	\[
	h^i(E_\bullet^{DT(1)}(X)^\vee)=R^{i+1}p_{1*}R\sHom(\sI_\Delta, \sI_\Delta)
	\]
	for $i=0,1$.  
	
	Similarly, we have 
	\begin{align*}
		\text{ob}_1(X)^\vee&=R^1p_{1*}R\sHom(R\sHom(\sI_\Delta, \sI_\Delta)_0, p_2^*K_X)\\
		&\cong R^1p_{1*}(R\sHom(\sI_\Delta, \sI_\Delta)_0\otimes p_2^*K_X).
	\end{align*}
	Using the spectral sequence
	\[
	R^ip_{1*}(R^j\sHom(\sI_\Delta,\sI_\Delta)_0\otimes p_2^*K_X)\Rightarrow R^{i+j}p_{1*}(R\sHom(\sI_\Delta,\sI_\Delta)_0\otimes p_2^*K_X)
	\]
	and the vanishing of $R^0\sHom(\sI_\Delta,\sI_\Delta)_0$ and $H^1(X, K_X)$, we see that 
	\[
	R^1p_{1*}(R\sHom_0(\sI_\Delta,\sI_\Delta)_0\otimes p_2^*K_X)\cong p_{1*}(R^1\sHom(\sI_\Delta, \sI_\Delta)\otimes p_2^*K_X).
	\]
	
	Similarly, using the spectral sequence
	\[
	R^ip_{1*}(R^j\sHom(\sI_\Delta,\sI_\Delta)_0)\Rightarrow R^{i+j}p_{1*}(R\sHom(\sI_\Delta,\sI_\Delta)_0)
	\]
	and the vanishing of  $R^0\sHom(\sI_\Delta,\sI_\Delta)_0$ and $H^1(X,\sO_X)$, we have
	\[
	T_X=h^0(E_\bullet^{DT(1)}(X)^\vee)= R^1p_{1*}(R\sHom(\sI_\Delta,\sI_\Delta)_0)=p_{1*}(R^1\sHom(\sI_\Delta, \sI_\Delta)).
	\]
	Since   
	\[
	R^1\sHom(\sI_\Delta, \sI_\Delta)\cong R^2\sHom(i_{\Delta*}\sO_X,\sI_\Delta), 
	\]
	we see that $\sI_\Delta\cdot R^1\sHom(\sI_\Delta, \sI_\Delta)=0$, so 
	\[
	T_X=p_{1*}(R^1\sHom(\sI_\Delta, \sI_\Delta))\Rightarrow R^1\sHom(\sI_\Delta, \sI_\Delta)\cong i_{\Delta*}T_X. 
	\]
	
	Thus
	\begin{align*}
		\text{ob}_1(X)^\vee&\cong p_{1*}(R^1\sHom(\sI_\Delta, \sI_\Delta)\otimes p_2^*K_X)\\
		&\cong p_{1*}(i_{\Delta*}T_X\otimes p_2^*K_X)\\
		&\cong p_{1*}(i_{\Delta*}(T_X\otimes i_{\Delta}^*p_2^*K_X))\\
		&\cong T_X\otimes K_X. \qedhere 
	\end{align*}
\end{proof}

\section{Constructing $N_3$-oriented toric threefolds}\label{sec:Construction}
In this section, we find conditions on a complete  complete regular simplicial fan $\Sigma\subset \sN\cong \Z^3$ for the toric variety $X_\Sigma$ over $k$ to admit an $N_3$-orientation.

\subsection{Statement of the main result}

We let   $\Sigma\subset \sN$ be a complete regular simplicial fan in the co-character lattice $\sN$ for $\G_m^3$. Let $\sM$ be the  character lattice, dual to $\sN$. 

Recall that $\G_m^3=\Spec k[x_1^{\pm},x_2^{\pm}, x_3^{\pm}]$, where $x_i:\G_m^3\to \G_m$ is the $i$th projection. Using $x_1, x_2, x_3$ as a $\Z$-basis for $\sM$, gives the dual $\Z$-basis $e_1, e_2, e_3$ for $\sN$, identifying $\sN$ with $\Z^3$. 

For each $\kappa\in \Sigma(3)$, fix generators $v_1^\kappa, v_2^\kappa, v_3^\kappa$ of $\kappa$, with corresponding standard coordinates $x_1^\kappa, x_2^\kappa, x_3^\kappa$ for $U_\kappa$. Write
\[
v_i^\kappa=\sum_{j=1}^3 v_{ij}^\kappa\cdot e_j.
\]
Then $(v_{ij}^\kappa)$ is invertible in $M_{3\times 3}(\Z)$, so defining $(r_{ij}^\kappa)=(v_{ij}^\kappa)^{-1}$, $r_{ij}^\kappa$ is an integer for all $i,j,\kappa$, and 
$\pm1=\det (v_{ij}^\kappa)=\det (r_{ij}^\kappa)$. Moreover, we have
\begin{equation}\label{eqn:CoordTrans}
x_j^\kappa=\prod_{i=1}^3x_i^{r_{ij}^\kappa},\ j=1,2,3.
\end{equation}

As in \S\ref{subsec:ToricRecoll}, we have the open subscheme $U_0=\Spec k[x_1^{\pm1}, x_2^{\pm1}, x_3^{\pm1}]\subset X_\Sigma$. We consider $\G_m^3$ as a subgroup scheme of $\Aut_kX_\Sigma$ via its canonical action on $X_\Sigma$; this action restricts to the $\G_m^3$-action on $U_0$ given by $(t_1, t_2, t_3)^*(x_i)=t_ix_i$. We let $\rho_T:T=\G_m^3\to \G_m^3\subset \Aut_kX_\Sigma$ be the map $\rho_T(t_1, t_2, t_3)=(t_1^{-2}, t_2^{-2}, t_3^{-2})$, and we let $\rho_0(\sigma_\delta)\in \Aut_k(U_0)$ be defined by $\rho_0(\sigma_\delta)^*(x_i)=-x_i^{-1}$.

\begin{prop}\label{prop:Construction} Let $\Sigma\subset \sN=\Z^3$ be a regular complete simplicial fan defining the smooth proper toric variety $X= X_\Sigma$ over $k$. Suppose that $-\id_\sN(\Sigma)=\Sigma$. Then we have the following.
\begin{enumerate}
\item $\rho_0(\sigma_\delta):U_0\to U_0$ extends to an automorphism $\rho(\sigma_\delta):X_\Sigma\to X_\Sigma$, and there is a unique homomorphism $\rho:N_3\to \Aut_k(X)$ that restricts to $\rho_T$ on $T$ and sends $\sigma_\delta$ to  $\rho(\sigma_\delta)$. 
\item Suppose $X$ has a very ample $\G_m^3$-linearized invertible sheaf. Then with the $N_3$-action of (1), $X$ admits an $N_3$-orientation if the following conditions hold:
\begin{enumerate}
\item\label{enum:Odd} For each $\kappa\in \Sigma(3)$, $\sum_{i,j=1}^3r_{ij}^\kappa$ is odd.
\item\label{enum:Even} For $\kappa, \kappa'\in\Sigma(3)$ and $i=1,2,3$, $\sum_{j=1}^3(r_{ij}^\kappa-r_{ij}^{\kappa'})$ is even.
\end{enumerate}
\end{enumerate}
\end{prop}

\begin{rem} By
Lemma~\ref{lem:SigmaInvolution}, in order for $X_\Sigma$ to admit an $N_3$-orientation, one must have $-\id_\sN(\Sigma)=\Sigma$, so this assumption in the statement of Proposition~\ref{prop:Construction} is a necessary condition. 
\end{rem}

\begin{rem} Going back to our $N_3$-action on $(\P^1)^3$, if we let $\Sigma$ be the complete simplicial fan generated by taking $\Sigma(1)$ to be the positive rays generated by $(\pm 1,0,0)$, $(0,\pm1,0)$ and $(0,0,\pm1)$, then $X_\Sigma=(\P^1)^3$. If we use the standard multi-homogeneous coordinates $([X_0: X_1], [Y_0: Y_1]: [Z_0:Z_1])$ for $(\P^1)^3$, then letting $x_1=X_1/X_0, x_2 = Y_1/Y_0$ and $x_3 = Z_1/Z_0$, we have that $x_1, x_2, x_3$ is our standard basis for $\sM$, and in terms of $x_1, x_2, x_3$, our $N_3$-action on $(\P^1)^3$ is given by
\[
(t_1, t_2, t_3)^*(x_i)=t_i^{-2}x_i,\ \sigma_\delta^*(x_i)=-x_i^{-1}.
\]

Moreover, for each $\kappa=(klm)$, we have $r_{ij}^\kappa=\pm\delta_{ij}$, with $r_{ii}^\kappa=+1$ if the $i$'th term in $klm$ is $0$ and $r_{ii}^\kappa=-1$ if the $i$'th term in $klm$ is $1$. Thus $r_{ij}^\kappa-r_{ij}^{\kappa'}$ is even for all $i,j$ and $\sum_{ij}r_{ij}^\kappa$ is always odd, so our base-case of $(\P^1)^3$ is a special case of Proposition~\ref{prop:Construction}. 

Similarly, it is easy to show that if $\Sigma$ satisfies the conditions of Proposition~\ref{prop:Construction}, and if we let $\Sigma'$ be the fan corresponding to the blow-up of $X_\Sigma$ along an $N_3$-stable set of $T$-fixed points, then the $N_3$-action on $X_{\Sigma'}$ induced by our given one on $X_\Sigma$ is given by our construction, and $\Sigma'$ also satisfies the conditions of Proposition~\ref{prop:Construction}.

Unfortunately, we do not at present know of an example of a $\Sigma$ satisfying the conditions of Proposition~\ref{prop:Construction} other than those arising from repeated blow-ups of  $(\P^1)^3$ along $N_3$-stable sets of $T$-fixed points. We present Proposition~\ref{prop:Construction} and its proof in the hope that this may change in the future.
\end{rem}

\subsection{A useful lemma}
As an aid to the proof of Proposition~\ref{prop:Construction}, we first state a lemma on conditions for the existence of an $N_3$-linearization for a line bunde on $X_\Sigma$.

Let $\{\tau_{\kappa, \kappa'}\in \sO_{X_\Sigma}(U_\kappa\cap U_{\kappa'})^\times\}_{\kappa,\kappa'\in \Sigma(3)}$ be a 1-cocycle defining an invertible sheaf $L_\tau$ on $X_\Sigma$, that is,
\[
\tau_{\kappa, \kappa'}\cdot \tau_{\kappa', \kappa''}=\tau_{\kappa, \kappa''}
\]
on $U_\kappa\cap U_{\kappa'}\cap  U_{\kappa''}$ and
 for $U\subset X_\Sigma$,
\[
L_\tau(U)=\{\{s_\kappa\in \sO_{X_\Sigma}(U\cap U_\kappa)\}_\kappa\mid s_\kappa=\tau_{\kappa, \kappa'}s_{\kappa'} \text{ on }U\cap U_\kappa\cap U_{\kappa'}\}.
\]
We suppose we are given an action of $N_3$ on $X_\Sigma$, via a homomorphism
\[
\rho:N_3\to \Aut_k(X_\Sigma)
\]
satisfying the following conditions.
\begin{equation}\label{eqn:ActionConditions}
\vbox{
\begin{itemize}
\item With respect to the inclusion of $\G_m^3$ defined by the canonical $\G_m^3$-action on $X_\Sigma$, $\rho$ restricts to a surjective homomorphism $\rho:T\to \G_m^3\subset \Aut_k(X_\Sigma)$.
\item $\rho(-\id_T)=\id_{X_\Sigma}$, hence $\rho(\sigma_\delta)^2=\id_{X_\Sigma}$.
\item $\rho(\sigma_\delta)(U_\kappa)=U_{-\kappa}$.
\end{itemize}
}
\end{equation}

\begin{lem}\label{lem:N3Linearization} Suppose we have an $N_3$-action on $X_\Sigma$ satisfying the conditions \eqref{eqn:ActionConditions}. Then an $N_3$-linearization of $L_\tau$ is given by the following data: For each $\kappa\in \Sigma(3)$ a character $\lambda_\kappa:T\to \G_m$, and an element $\lambda_\kappa(\sigma_\delta)\in k^\times$. These satisfy:
\begin{enumerate}
\item\label{lem:N3Linearization;2}  Given $\kappa, \kappa'\in \Sigma(3)$, then for $t\in T$, we have 
\[
\rho(t)^*(\tau_{\kappa, \kappa'})=\frac{\lambda_{\kappa'}(t)}{\lambda_\kappa(t)}\cdot \tau_{\kappa, \kappa'},
\]
and
\[
\rho(\sigma_\delta)^*(\tau_{-\kappa, -\kappa'})=\frac{\lambda_{\kappa'}(\sigma_\delta)}{\lambda_\kappa(\sigma_\delta)}\cdot \tau_{\kappa, \kappa'}.
\]
\item \label{lem:N3Linearization;3}
$\lambda_\kappa(\sigma_\delta)\cdot \lambda_{-\kappa}(\sigma_\delta)=\lambda_\kappa(-\id_T)$ for all $\kappa\in \Sigma(3)$.
\item\label{lem:N3Linearization;1} $\lambda_{-\kappa}=\lambda_\kappa^{-1}$ on $T$.
\end{enumerate}
Given this data, define
\[
\Psi_t:\rho(t)^*L_\tau\to L_\tau
\]
by the collection of map $\Psi_t^\kappa=\times \lambda_\kappa(t):k[U_\kappa]\to k[U_\kappa]$ and define 
\[
\Psi_{\sigma_\delta}:\rho(\sigma_\delta)^*L_\tau\to L_\tau
\]
by the collection of maps $\Psi_{\sigma_\delta}^\kappa:\times \lambda_\kappa(\sigma_\delta):k[U_\kappa]\to k[U_\kappa]$.
Then these maps fit together to define an $N_3$-linearization of $L_\tau$, and all $N_3$-linearizations of $L_\tau$ occur in this way.
\end{lem}

\begin{proof} For each $\kappa\in \Sigma(3)$,  our assumptions on $\rho$ imply that $\rho(t)(U_\kappa)=U_\kappa$ for each $t\in T$ and $\rho(\sigma_\delta)(U_\kappa)=U_{-\kappa}$. This gives the canonical isomorphism
\[
(\rho(t)^*L_\tau)_{|U_\kappa}= k[U_\kappa]\otimes_{\id, k[U_\kappa], \rho(t)^*} k[U_\kappa]\cong k[U_\kappa]
\]
and
\[
(\rho(\sigma_\delta)^*L_\tau)_{|U_\kappa}= k[U_{-\kappa}]\otimes_{\id, k[U_{-\kappa}], \rho(\sigma_\delta)^*} k[U_\kappa]\cong 
k[U_\kappa].
\]
This in turn identifies $\rho(t)^*L_\tau$ with $L_{\rho(t)^*\tau}$ and $\rho(\sigma_\delta)^*L_\tau$ with $L_{\rho(\sigma_\delta)^*\tau}$, where
\[
(\rho(t)^*\tau)_{\kappa, \kappa'}=\rho(t)^*(\tau_{\kappa, \kappa'}),\
(\rho(\sigma_\delta)^*\tau)_{\kappa, \kappa'}=\rho(\sigma_\delta)^*(\tau_{-\kappa, -\kappa'}).
\]
By \eqref{lem:N3Linearization;2}, the maps
\[
 \times \lambda_\kappa(t):k[U_\kappa]\to k[U_\kappa],\ \kappa\in \Sigma(3)
\]
give a well-defined map $\Psi_t:\rho(t)^*L_\tau\to L_\tau$, and the maps
\[
 \times \lambda_\kappa(\sigma_\delta):k[U_\kappa]\to k[U_\kappa],\ \kappa\in \Sigma(3)
\]
give a well-defined map $\Psi_{\sigma_\delta}:\rho(\sigma_\delta)^*L_\tau\to L_\tau$.

For $t_1, t_2\in T$, $\kappa\in \Sigma(3)$, we have 
\[
\lambda_\kappa(t_1)\cdot \rho(t_1)^*(\lambda_\kappa(t_2))=
\lambda_\kappa(t_1)\cdot \lambda_\kappa(t_2)=\lambda_\kappa(t_1t_2)
\]
since $\rho(t_1)^*$ acts by the identity on $k\subset k[U_\kappa]$ and $\lambda_\kappa:T\to \G_m$ is a character. Thus
\[
\Psi_{t_1}\circ\rho(t_1)^*\Psi_{t_2}=\Psi_{t_1t_2}.
\]
Similarly, since $\rho(\sigma_\delta)^*$ acts by the identity on $k$, it follows from  \eqref{lem:N3Linearization;3} that
\[
\lambda_{\kappa}(\sigma_\delta)\cdot \rho(\sigma_\delta)^*(\lambda_{-\kappa}(\sigma_\delta))=
\lambda_{\kappa}(\sigma_\delta)\cdot \lambda_{-\kappa}(\sigma_\delta)=\lambda_\kappa(-\id_T).
\]
As $\sigma_\delta^2=-\id_T$, we thus have
\[
\Psi_{\sigma_\delta}\circ\rho(\sigma_\delta)^*\Psi_{\sigma_\delta}=\Psi_{\sigma_\delta^2}.
\]
Finally, for $t\in T$, we have by \eqref{lem:N3Linearization;1}
\[
\lambda_{\kappa}(t^{-1})\cdot \rho(t^{-1})^*(\lambda_\kappa(\sigma_\delta))=
\lambda_{-\kappa}(t)\cdot \lambda_\kappa(\sigma_\delta)=
\lambda_\kappa(\sigma_\delta)\cdot \rho(\sigma_\delta)^*(\lambda_{-\kappa}(t)),
\]
so
\[
\Psi_{t^{-1}}\circ  \rho(t^{-1})^*(\Psi_{\sigma_\delta})=\Psi_{\sigma_\delta}\circ\rho(\sigma_\delta)^*\Psi_t.
\]

Defining
\[
\Psi_{\sigma_\delta\cdot t}:=\Psi_{\sigma_\delta}\circ\rho(\sigma_\delta)^*\Psi_t
\]
and
\[
\Psi_{t\cdot \sigma_\delta}:=\Psi_t\circ\rho(t)^*\Psi_{\sigma_\delta}
\]
we then have 
\[
\Psi_{\sigma_\delta\cdot t}=\Psi_{t^{-1}\sigma_\delta}.
\]

In total, we see that  the relations defining $N_3$ as a quotient of the free product of $\G_m^3$ and $\<\sigma_\delta\>=\Z/4$ are respected by $\Psi$, and $\Psi_{\alpha\cdot \beta}=\Psi_\alpha\circ \rho(\alpha)^*\Psi_\beta$ for all $\alpha,\beta\in N_3$,  giving a well-defined $N_3$-linearization of $L_\tau$.

Conversely, $L_\tau$, $\rho(\sigma_\delta)^*L_\tau$ and $\rho(t)^*L_\tau$ are all trivialized on the cover $\{U_\kappa\}_{\kappa\in \Sigma(3)}$. Given an $N_3$-linearization $\Psi$ of $L_\tau$ and an $\alpha\in N_3$, the map $\Psi_\alpha:\rho(\alpha)^*L_\tau\to L_\tau$ is thus given by a collection of units $\lambda_\kappa(\alpha)\in k[U_\kappa]^\times =k^\times$. By reading the above argument in reverse, we see that $\Psi$ is determined by the elements $\lambda_\kappa(t)$, $t\in T$, and $\lambda_\kappa(\sigma_\delta)$, and that these must satisfy the relations \eqref{lem:N3Linearization;2}, \eqref{lem:N3Linearization;3}, \eqref{lem:N3Linearization;1}, so the data described in the Lemma gives all possible $N_3$-linearizations of $L_\tau$.
\end{proof}

\subsection{The proof of Proposition~\ref{prop:Construction}}

Before launching into the proof, we introduce the notion of {\em oriented} standard coordinates and some other notations. 

Recall that for each $\kappa$, the standard coordinates $x_1^\kappa, x_2^\kappa, x_3^\kappa$ are uniquely determined up to reordering. Noting that the conditions \eqref{enum:Odd} and \eqref{enum:Even} are invariant under reordering of the standard coordinates, we are thus free to fix a choice of standard coordinates for each $\kappa$.

Let 
\[
\omega:=\frac{dx_1}{x_1}\frac{dx_2}{x_2}\frac{dx_3}{x_3}.
\]
We call the coordinates $x_1^\kappa, x_2^\kappa, x_3^\kappa$ {\em oriented} if $\det (r_{ij}^\kappa)=1$, and we similarly call dual generators $v_1^\kappa, v_2^\kappa, v_3^\kappa$ for $\kappa$ oriented if equivalently $\det (v_{ij}^\kappa)=1$.

By assumption $-\id_\sN(\Sigma)=\Sigma$. 
Let $\Sigma(3)^0\subset \Sigma(3)$ be a set of representatives for the action of $-\id_\sN$ on $\Sigma(3)$. We will assume for all $\kappa\in \Sigma(3)^0$ that $v_1^\kappa, v_2^\kappa, v_3^\kappa$ are oriented generators for $\kappa$, and that $v_i^{-\kappa}=-v_i^\kappa$, $i=1,2,3$.  

For standard coordinates $x_1^\kappa, x_2^\kappa, x_3^\kappa$ we set
\[
x^\kappa:=x_1^\kappa\cdot x_2^\kappa\cdot x_3^\kappa,\ dx^\kappa:=
dx_1^\kappa\wedge dx_2^\kappa\wedge  dx_3^\kappa.
\]
Then from  \eqref{eqn:CoordTrans} we have
\begin{equation}\label{eqn:nForm1}
\frac{dx^\kappa}{x^\kappa}=\det(r_{ij}^\kappa)\cdot\omega=\pm \omega,
\end{equation}
with the sign $+1$ if and only if $x_1^\kappa,x_2^\kappa, x_3^\kappa$  is oriented.

\begin{rem}\label{rem:Orient} Since  $x_1^\kappa,x_2^\kappa, x_3^\kappa$ is oriented if and only if  $x_1^{-\kappa}=(x_1^\kappa)^{-1}$, $x_3^{-\kappa}=(x_2^\kappa)^{-1}$, $x_3^{-\kappa}=(x_3^\kappa)^{-1}$ is not oriented, we see that  $x_1^\kappa,x_2^\kappa, x_3^\kappa$ is oriented if and only if $\kappa$ is in $\Sigma(3)^0$. 
\end{rem}

\begin{proof}[Proof of Proposition~~\ref{prop:Construction}]
Given  $\kappa, \kappa'\in \Sigma(3)$, with corresponding standard coordinates  $x_1^{\kappa},x_2^{\kappa}, x_3^{\kappa}$ and $x_1^{\kappa'},x_2^{\kappa'}, x_3^{\kappa'}$, it follows from \eqref{eqn:CoordTrans} that
\begin{equation}\label{eqn:Cocycle}
x^{\kappa'}=\left[\prod_{i=1}^3x_i^{\sum_jr_{ij}^{\kappa'}-\sum_{ij}r_{ij}^\kappa}\right]\cdot x^\kappa.
\end{equation}

Let 
\[
\xi_{\kappa, \kappa'}:=\prod_{i=1}^3x_i^{\sum_jr_{ij}^{\kappa'}-\sum_{ij}r_{ij}^\kappa}.
\]
Since $x^\kappa$ and $ x^{\kappa'}$ define the same divisor $D:=\sum_{\kappa^*\in \Sigma(1)}D_{\kappa*}$ on $U_\kappa\cap U_{\kappa'}$,  it follows from \eqref{eqn:Cocycle} that  $\xi_{\kappa, \kappa'}$ is a unit on $U_\kappa\cap U_{\kappa'}$. Moreover, \eqref{eqn:Cocycle} gives a canonical isomorphism of $\sO_{X_\Sigma}(-D)$ with the invertible sheaf defined by the cocycle 
$\xi_{\kappa, \kappa'}\in \Gamma(U_\kappa\cap U_{\kappa'}, \sO_{X_\Sigma}^\times)$ for the cover $\{U_\kappa\}_{\kappa\in \Sigma(3)}$ of $X_\Sigma$, by using $\prod_{j=1}^3x_j^\kappa$ for the generator of $\sO_{U_\kappa}(-D_{|U_\kappa})$.

From  \eqref{eqn:nForm1}, we have
\[
\Div(\omega)=D,
\]
so using $\omega$ as global generator for $K_{X_\Sigma}(D)$ gives us 
the canonical isomorphism
\[
\phi_X: K_{X_\Sigma}(D)\cong \sO_{X_\Sigma},
\]
with 
\[
\phi_X(\frac{dx_1}{x_1}\frac{dx_2}{x_2}\frac{dx_3}{x_3})=1.
\]
Restricting to $U_\kappa$, we thus have
\[
\phi_X(\frac{dx^\kappa}{x^\kappa})=\pm 1,
\]
with value $+1$ exactly when $x_1^\kappa, x_2^\kappa, x_3^\kappa$ are oriented standard coordinates.

We have chosen oriented generators $v_1^\kappa, v_2^\kappa, v_3^\kappa$  for $\kappa\in \Sigma(3)^0$ with dual oriented  standard coordinates $x_1^\kappa,x_2^\kappa, x_3^\kappa$ for $U_\kappa$, and we have chosen 
$-v_1^\kappa, -v_2^\kappa, -v_3^\kappa$ as ``anti-oriented'' generators for $-\kappa$ with dual anti-oriented standard coordinates  $x_1^{-\kappa}, x_2^{-\kappa},x_3^{-\kappa}$ for $U_{-\kappa}$ (see Remark~\ref{rem:Orient}). Thus
\[
x_j^{-\kappa}=(x_j^\kappa)^{-1}
\]
and 
\begin{equation}\label{eqn:MinusAction}
(r_{ij}^{-\kappa})=(-r_{ij}^\kappa).
\end{equation}
In particular, the section $-\frac{dx^{-\kappa}}{x^{-\kappa}}$ of  $K_{X_\Sigma}(D)$ over $U_{-\kappa}$ maps to 1 under the isomorphism $\phi_X$.

We have the canonical action $\G_m^3\times X_\Sigma\to X_\Sigma$. For $(t_1,t_2, t_3)\in \G_m^3$, \eqref{eqn:CoordTrans} implies that
\[
(t_1, t_2, t_3)^*(x_j^\kappa)=(\prod_{i=1}^3t_i^{r_{ij}^\kappa})\cdot x_j^\kappa.
\]
To show (1), we have the homomorphism $\rho_T:T\to \G_m^3\subset \Aut_k(X_\Sigma)$ defined by
\[
\rho_T(t_1, t_2, t_3)=(t_1^2, t_2^2, t_3^2),
\]
so $\rho_T(-\id_T)=\id_{X_\Sigma}$. We have the action of $\sigma_\delta$  on $U_0=\Spec k[\sM]=\G_m^3$ defined by
\[
\rho_0(\sigma_\delta)^*(x_i)=-x_i^{-1}, i=1,2,3.
\]
It follows from \eqref{eqn:MinusAction} that $\rho_0(\sigma_\delta)$ extends to an involution $\rho(\sigma_\delta)$ of $X_\Sigma$, with $\rho(\sigma_\delta)(U_\kappa)=U_{-\kappa}$ and with
\[
 \rho(\sigma_\delta)^*(x_j^\kappa)=(-1)^{\sum_ir_{ij}^\kappa}x_j^{-\kappa}.
 \]
 We also have
\[
\rho_T(t^{-1})^*\circ \rho(\sigma_\delta)^*= \rho(\sigma_\delta)^*\circ \rho_T(t)^*
\]
for all $t\in T$,  hence $\rho_T$ and $\rho(\sigma_\delta)$ fit together to define a homomorphism
\[
\rho:N_3\to \Aut_k(X_\Sigma),
\]
proving (1). 

The isomorphism $\phi_X$ is {\em not} $\<\sigma_\delta\>$-equivariant, since 
\[
\rho(\sigma_\delta)^*(\frac{dx_1}{x_2}\frac{dx_1}{x_2}\frac{dx_3}{x_3})
=-\frac{dx_1}{x_2}\frac{dx_1}{x_2}\frac{dx_3}{x_3}
\]
while $\rho(\sigma_\delta)^*(1)=1$.  
Twisting by $\sO_X(-D)$ gives us the isomorphism
\[
\phi_X(-D):K_X\to \sO_X(-D),
\]
which sends the generator $dx^\kappa$ of $K_X|_{U_\kappa}$ to $x^\kappa$ for $\kappa\in \Sigma(3)^0$, and to $-x^\kappa$ for $\kappa\not\in \Sigma(3)^0$.  For the natural $\<\sigma_\delta\>$-linearization on $K_X$ and $\sO_X(-D)$ given by the action of $\rho(\sigma_\delta)$ on $X$, we have $\rho(\sigma_\delta)^*(dx^\kappa)=-dx^{-\kappa}$ and $\rho(\sigma_\delta)^*(x^\kappa)=-x^{-\kappa}$  for all $\kappa\in \Sigma(3)$. We therefore twist the natural  $\<\sigma_\delta\>$-linearization on $\sO_X(-D)$ by the character $\sigma_\delta\mapsto -1$, giving the  $\<\sigma_\delta\>$-linearized invertible sheaf $\sO_X(-D)^-$ with $\sigma_\delta^*(x^\kappa)=x^{-\kappa}$  for all $\kappa\in \Sigma(3)$. Thus $\phi_X(-D)$ defines the $\<\sigma_\delta\>$-equivariant isomorphism
\[
\phi_X(-D)^-:K_X\to \sO_X(-D)^-.
\]
The subsheaf $\sO_{X_\Sigma}(-D)$ of $\sO_{X_\Sigma}$ inherits a $T$-linearization from $\sO_{X_\Sigma}$, which induces a $T$-linearization on $\sO_{X_\Sigma}(-D)^-$ from the identity of underlying sheaves $\sO_{X_\Sigma}(-D)^-= \sO_{X_\Sigma}(-D)$, and $\phi_X(-D)^-$ is $T$-equivariant. 

Writing 
\[
\rho(t_1, t_2, t_3)^*(x_j^\kappa)=(t_1^{s_{1j}^\kappa}t_2^{s_{2j}^\kappa}t_3^{s_{3j}^\kappa})\cdot x_j^\kappa
\]
we have $s_{ij}^\kappa=-2r_{ij}^\kappa$, so the $s_{ij}^\kappa$ are all even. The given $T$-linearization and   $\<\sigma_\delta\>$-linearization for $\sO_{X_\Sigma}(-D)^-$ fit together to give $\sO_{X_\Sigma}(-D)^-$ an $N_3$-linearization, for which $\phi_X(-D)^-$ is $N_3$-equivariant.

We now assume the $r_{ij}^\kappa$ satisfy the conditions \eqref{enum:Odd} and \eqref{enum:Even}. The matrices $(s_{ij}^\kappa)$ then satisfy Definition~\ref{defn:N3Orient}\eqref{defn:N3Orient:4} since each $s_{ij}^\kappa$ is even, and
\[
\sum_{i,j=1}^3s_{ij}^\kappa=-2\cdot \sum_{i,j=1}^3r_{ij}^\kappa\equiv 2\mod 4.
\]
In addition,  $\sum_j(r_{ij}^\kappa-r_{ij}^{\kappa'})$ is an even integer for each $i=1,2,3$ and for each $\kappa, \kappa'\in \Sigma(3)$. 

Let $L_\eta$ be the invertible sheaf on $X_\Sigma$ defined by the 1-cocycle
\[
\eta_{\kappa, \kappa'}:=\prod_{i=1}^3 x_i^{\frac{1}{2}\sum_j(r_{ij}^{\kappa'}-r_{ij}^{\kappa})}
\]
for the cover $\{U_\kappa\mid \kappa\in \Sigma(3)\}$ of $X_\Sigma$. Using the local generator $x^{\kappa}$ for $\sO_{X_\Sigma}(-D)|_{U_\kappa}$, \eqref{eqn:Cocycle} gives us the canonical isomorphism $\psi_X: \sO_{X_\Sigma}(-D)^-\xrightarrow{\sim} L_\eta^{\otimes 2}$
of invertible sheaves on $X_\Sigma$, sending the generator $x_\kappa$ of $\sO_X(-D)^-$ on $U_\kappa$ to the canonical generator $1\in k[U_\kappa]=L_\eta^{\otimes 2}|_{U_\kappa}$. Let \[
\tau_X:K_{X_\Sigma}\xrightarrow{\sim}L_\eta^{\otimes 2}
\]
be the composition of isomorphisms
\[
K_{X_\Sigma}\xymatrix{\ar[r]^{\phi_X(-D)^-}_\sim&}\sO_{X_\Sigma}(-D)^-\xymatrix{\ar[r]^{\psi_X}_\sim&} L_\eta^{\otimes 2}.
\]

For $\kappa\in \Sigma(3)$ define
\[
\lambda_\kappa(\sigma_\delta):=\begin{cases} 
+1&\text{ if }\sum_{ij}r_{ij}^\kappa\equiv1\mod 4\\
-1&\text{ if }\sum_{ij}r_{ij}^\kappa\equiv3\mod 4,
\end{cases}
\]
which is well-defined, since $\sum_{ij}r_{ij}^\kappa$ is odd for all $\kappa$.
Then
\[
\frac{\lambda_{\kappa'}(\sigma_\delta)}{\lambda_\kappa(\sigma_\delta)}=(-1)^{\frac{1}{2}\sum_{ij}(r_{ij}^{\kappa'}-r_{ij}^\kappa)}.
\] 
which gives us 
\begin{align*}
\rho(\sigma_\delta)^*(\eta_{-\kappa, -\kappa'})&=\prod_{i=1}^3 (-x_i)^{\frac{1}{2}\sum_j(r_{ij}^{\kappa'}-r_{ij}^{\kappa})}\\
&=\frac{\lambda_{\kappa'}(\sigma_\delta)}{\lambda_\kappa(\sigma_\delta)} \prod_{i=1}^3 x_i^{\frac{1}{2}\sum_j(r_{ij}^{\kappa'}-r_{ij}^\kappa)}\\
&=\frac{\lambda_{\kappa'}(\sigma_\delta)}{\lambda_\kappa(\sigma_\delta)}\eta_{\kappa, \kappa'}.
\end{align*}

For $\kappa\in \Sigma(3)$, and  $t=(t_1, t_2, t_3)\in T=\G_m^3$, define
\[
\lambda_\kappa(t):=\prod_{i=1}^3t_i^{-\sum_jr_{ij}^\kappa}.
\]
Then 
\[
\rho(t)^*(\eta_{\kappa, \kappa'})=\frac{\lambda_{\kappa'}(t)}{\lambda_\kappa(t)}\cdot
\eta_{\kappa, \kappa'}.
\]
Thus by Lemma~\ref{lem:N3Linearization}, the collection of units $\{\lambda_\kappa(t), \lambda_\kappa(\sigma_\delta)\}$ give us an $N_3$-linearization of $L_\eta$, and the collection
$\{\lambda_\kappa(t)^2, \lambda_\kappa(\sigma_\delta)^2=1\}$ defines the induced $N_3$-linearization of $L_\eta^{\otimes 2}$. If we trivialize $\sO_X(-D)^-$ using the generator $x^\kappa$ on $U_\kappa$, we see that the corresponding cocycle describing $\sO_X(-D)^-$ is $\{\eta_{\kappa, \kappa'}^2\}$ and that $\{\lambda_\kappa(t)^2, \lambda_\kappa(\sigma_\delta)^2=1\}$ is  the collection of units defining our given $N_3$-linearization of $\sO_X(-D)^-$.  Thus our isomorphism $\psi_X:\sO_X(-D)^-\to L_\eta^{\otimes 2}$ is   $N_3$-equivariant, and $\tau_X:K_X\xrightarrow{\sim}  L_\eta^{\otimes 2}$ defines an $N_3$-equivariant spin structure on $X_\Sigma$.

Finally, by assumption $X_\Sigma$ admits a $\G_m^3$-linearized very ample invertible sheaf $\sL$, which is then automatically $T$-linearized with corresponding isomorphisms $\phi_t:\rho(t)^*\sL\to \sL$ for $t\in T$. Let $\sO_X(1)$ be the very ample invertible sheaf $\sL\otimes\rho(\sigma_\delta)^*\sL$. The $T$-linearization of $\sL$ induces a $T$-linearization of $\sO_X(1)$ by
\[
\psi_t:=\phi_t\otimes\rho(\sigma_\delta)^*(\phi_{t^{-1}}):
\rho(t)^*(\sL\otimes\rho(\sigma_\delta)^*\sL)\xrightarrow{\sim} \sL\otimes\rho(\sigma_\delta)^*\sL,\ t\in T.
\]
We also have
\[
\psi_{\sigma_\delta}:\rho(\sigma_\delta)^*(\sL\otimes\rho(\sigma_\delta)^*\sL)\xrightarrow{\sim} \sL\otimes\rho(\sigma_\delta)^*\sL
\]
defined as the composition
\[
\rho(\sigma_\delta)^*(\sL\otimes\rho(\sigma_\delta)^*\sL)\cong
\rho(\sigma_\delta)^*(\sL)\otimes \sL\xymatrix{\ar[r]^\tau_\sim&}\sL\otimes\rho(\sigma_\delta)^*\sL,
\]
where $\tau$ is the exchange of factors. We define $\psi_{t\cdot \sigma_\delta}:=\psi_{\sigma_\delta}\circ \rho(\sigma_\delta)^*(\psi_t)$, and one can then check that with this definition of $\psi_t$, $\psi_{\sigma_\delta}$ and $\psi_{t\cdot \sigma_\delta}$, we have a well-defined $N_3$-linearization of $\sO_X(1)$. This completes the proof of the Proposition.
\end{proof}

\printbibliography
	
\end{document}